\documentclass{amsart}

\usepackage{amssymb}
\usepackage{amscd}
\usepackage{amsmath}
\usepackage{enumerate}
\usepackage{tensor}
\usepackage{mathtools}
\usepackage[all,cmtip, arrow, matrix, curve]{xy}
\usepackage{tikz-cd}
\usepackage{tikz-3dplot}
\usepackage{graphicx}
  \usepackage{wrapfig}
   \usetikzlibrary{decorations.pathreplacing,backgrounds, matrix}
  \usetikzlibrary{automata, patterns, calc}

\usepackage{afterpage}
\usepackage{setspace}

\usepackage{graphicx,import} 
\usepackage{caption}
\usepackage{subcaption}
\usepackage{color}
\usepackage{transparent}
\usepackage{hyperref}

\usepackage{color}

\newtheorem{theo}{Theorem}
\newtheorem{lema}[theo]{Lemma}
\newtheorem{cor}[theo]{Corollary}
\newtheorem{prop}[theo]{Proposition}

\newtheorem{definition}[theo]{Definition}

\newtheorem{remark}[theo]{Remark}
\newtheorem{notation}[theo]{Notation}

\newtheorem{example}[theo]{Example}
\newtheorem{problem}[theo]{Problem}
\newtheorem{innercustomcor}{Corollary}
\newenvironment{customcor}[1]
  {\renewcommand\theinnercustomcor{#1}\innercustomcor}
  {\endinnercustomcor}

\newcommand{\BB}{{\mathbb{B}}}
\newcommand{\CC}{{\mathbb{C}}}
\newcommand{\NN}{{\mathbb{N}}}
\newcommand{\PP}{{\mathbb{P}}}
\newcommand{\QQ}{{\mathbb{Q}}}
\newcommand{\RR}{{\mathbb{R}}}

\newcommand{\SSS}{{\mathbb{S}}}
\newcommand{\ZZ}{{\mathbb{Z}}}

\newcommand{\calB}{{\mathcal{B}}}
\newcommand{\calC}{{\mathcal{C}}}
\newcommand{\calD}{{\mathcal{D}}}

\newcommand{\calG}{{\mathcal{G}}}
\newcommand{\calH}{{\mathcal{H}}}

\newcommand{\calN}{{\mathcal{N}}}

\newcommand{\calS}{{\mathcal{S}}}
\newcommand{\calT}{{\mathcal{T}}}
\newcommand{\calU}{{\mathcal{U}}}
\newcommand{\calV}{{\mathcal{V}}}

\newcommand{\calY}{{\mathcal{Y}}}
\newcommand{\calZ}{{\mathcal{Z}}}
\newcommand{\comp}{{\circ}}

\newcommand{\dtr}{d_{\triangledown}}

\begin{document}
\title[]{Moderately Discontinuous Homology}

\author{J. Fern\'andez de Bobadilla}
\address{Javier Fern\'andez de Bobadilla:  
(1) IKERBASQUE, Basque Foundation for Science, Maria Diaz de Haro 3, 48013, 
    Bilbao, Bizkaia, Spain
(2) BCAM  Basque Center for Applied Mathematics, Mazarredo 14, E48009 Bilbao, 
Basque Country, Spain 
(3) Academic Colaborator at UPV/EHU} 
\email{jbobadilla@bcamath.org}

 \author{S. Heinze}
 \address{Sonja Heinze:  
BCAM  Basque Center for Applied Mathematics, Mazarredo 14, E48009 Bilbao, 
Basque Country, Spain }
\email{sheinze@bcamath.org}

 \author{M. Pe Pereira}
 \address{Facultad de Ciencias Matematicas- UCM, Plaza de Ciencias, 3,  Ciudad Universitaria, 28040 MADRID} 
\email{maria.pe@mat.ucm.es}

\author{J. E. Sampaio}
\address{J. Edson Sampaio:
	      (1) Departamento de Matem\'atica, Universidade Federal do Cear\'a,
	      Rua Campus do Pici, s/n, Bloco 914, Pici, 60440-900, 
	      Fortaleza-CE, Brazil. \newline
              (2) BCAM - Basque Center for Applied Mathematics,
	      Mazarredo, 14, E48009 Bilbao, Basque Country - Spain.
	    } 
\email{edsonsampaio@mat.ufc.br/esampaio@bcamath.org}

\thanks{The first author is supported by ERCEA 615655 NMST Consolidator Grant, MINECO by the project 
reference MTM2016-76868-C2-1-P (UCM), by the Basque Government through the BERC 2018-2021 program and Gobierno Vasco Grant IT1094-16, by the Spanish Ministry of Science, Innovation and Universities: BCAM Severo Ochoa accreditation SEV-2017-0718 and by Bolsa Pesquisador Visitante Especial (PVE) - Ciencias sem Fronteiras/CNPq Project number:  401947/2013-0. The second author is supported by a La Caixa Ph.D. grant associated to the MINECO project SEV-2011-0087, by ERCEA 615655 NMST Consolidator Grant, MINECO by the project reference MTM2016-76868-C2-1-P (UCM), by the Basque Government through the BERC 2018-2021 program and Gobierno Vasco Grant IT1094-16, by the Spanish Ministry of Science, Innovation and Universities: BCAM Severo Ochoa accreditation SEV-2017-0718, and by Bolsa Pesquisador Visitante Especial (PVE) - Ciencias sem Fronteiras/CNPq Project number:  401947/2013-0. 
The third author is supported by MINECO by the project reference MTM 2017-89420 and MTM2016-76868-C2-1-P (UCM).
The fourth author is supported by the ERCEA 615655 NMST Consolidator Grant and also by the Basque Government through the BERC 2018-2021 program and Gobierno Vasco Grant IT1094-16, by the Spanish Ministry of Science, Innovation and Universities: BCAM Severo Ochoa accreditation SEV-2017-0718.}

\begin{abstract}
We introduce a new metric homology theory, which we call Moderately Discontinuous Homology, designed to capture Lipschitz properties of metric singular subanalytic germs. The main novelty of our approach is to allow ``moderately discontinuous'' chains, which are specially advantageous for capturing the subtleties of the outer metric phenomena. Our invariant is a finitely generated graded abelian group $MDH^b_\bullet$ for any $b\in [1,\infty]$ and homomorphisms $MDH^b_\bullet\to MDH^{b'}_\bullet$ for any $b\geq b'$. Here $b$ is a ``discontinuity rate''. 
The homology groups of a subanalytic germ with the inner or outer metric are proved to be finitely generated and that only finitely many homomorphisms $MDH^b_\bullet\to MDH^{b'}_\bullet$ are essential. For $b=1$ it recovers the homology of the tangent cone for the outer metric and of the Gromov tangent cone for the inner one. In general, for $b=\infty$ the $MD$- homology recovers the homology of the punctured germ. Hence, our invariant can be seen as an algebraic invariant interpolating from the germ to its tangent cone. Our homology theory is a bi-Lipschitz subanalitic invariant, is invariant by suitable metric homotopies, and satisfies versions of the relative and Mayer-Vietoris long exact sequences. Moreover, fixed a discontinuity rate $b$ we show that it is functorial for a class of discontinuous Lipschitz maps, whose discontinuities are $b$-moderated; this makes the theory quite flexible. In the complex analytic setting we introduce an enhancement called Framed MD Homology, which takes into account information from fundamental classes. As applications we prove that Moderately Discontinuous Homology characterizes smooth germs among all complex analytic germs, recovers the number of irreducible components of complex analytic germs and the embedded topological type of plane branches. Framed MD Homology recovers the topological type of any plane curve singularity and relative multiplicities of complex analytic germs.


\end{abstract}


\maketitle

\tableofcontents

\section{Introduction}

Despite the intense recent research on Lipschitz geometry of real and complex analytic singularities, and more generally, subanalytic sets, still there is a need for general theories providing algebraic and numerical invariants which capture Lipschitz phenomena, and which have sufficiently many computational tools to be used in practice. In this paper we introduce a new metric homology theory, which we call Moderately Discontinuous Homology (MD-Homology, for short) for metric germs of subanalytic sets. Our theory is different in nature that the previous ones by Birbrair-Brasselet~\cite{BirbrairBrasselet:2000},~\cite{BirbrairBrasselet:2002}, and by Valette~\cite{Valette:2010}, in the sense that it is based in chains with controlled discontinuities rather than using allowable chains based on their size. This gives some advantages, specially when dealing with the outer metric. At the end of the introduction we comment further on this. 

In the spirit of any algebro-topological invariant, MD-Homology is a functor from a category of geometric nature, to a category of algebraic nature. The geometric category has  as objects pairs $(X,Y,x_0,d_X)$ of subanalytic germs in $\RR^n$ endowed with a metric, which includes the cases of inner metric (induced by the euclidean metric in $\RR^n$ and given by the minimum of the lengths of rectifiable paths in $X$ joining two points) and the outer metric (restriction of the euclidean metric in $\RR^n$). 
A morphism $f:(X,Y,x_0,d_X)\to (X',Y',x_0,d_{X'})$ in our category is a Lipschitz subanalytic morphism from $X$ to $X'$, sending $Y$ to $Y'$, with the so called linearly vertex approaching condition, which means that there exists a constant $K\geq 1$ such that $1/K||x-x_0||\leq ||f(x)-x'_0||\leq K||x-x_0||$, where $||x||$ denotes the usual norm of a vector in $\RR^n$. In particular MD-Homology is a bi-Lipschitz subanalytic invariant. Since the inner and outer metrics of real or complex analytic germs only depend on the analytic structure we deduce that MD-Homology for the inner or outer metric is a real or complex analytic invariant of real or complex analytic germs.

MD-Homology mimics the definition of Singular Homology, but since it applies to germs, which have conical structure, the \emph{simplices} for this theory are cones of usual simplices over the vertex of the germ. In other words and more concretely, a \emph{simplex} is a subanalytic family $\sigma_t:\Delta_n\to (X,x_0,d_X)$ of singular simplexes, parametrized by $t\in (0,1)$, such that the minimal and maximal distance of the image of $\sigma_t$ to $x_0$ approach $0$ linearly in $t$. We can form a complex by introducing the usual boundary operator.
For technical reasons and to make the theory easier,  we consider the complex obtained quotienting  the complex of simplices  by the equivalence by subdivision.  Moreover, for every $b\in (0,+\infty]$, we get a complex whose homology will be  $b$ MD-Homology, quotienting the previous one by the $b$-proximity relation which goes as follows:  two simplexes $\sigma_t$ and $\sigma'_t$ as above are $b$-proximous if the maximum over $x\in\Delta_n$ of $d_X(\sigma_t(x),\sigma'_t(x))$ decreases at order strictly larger than $b$ with respect to $t$. For $b=\infty$ one imposes no equivalence relation. 
The complete definition of MD-Homology takes the whole Section~\ref{sec:defi}.

The name of ``Moderately Discontinuous Homology'' comes from the fact that the $b$-proximity relation allows ``discontinuities bounded by order $b$'' in the closed chains that represent homology classes. As it will be clear later such discontinuities are specially useful in order to capture phenomena related with the outer metric. In order to let the reader develop an intuition, let us point that when two ``leaves'' of a germ approach each other at speed higher than $b$ towards the origin, then $b$-moderately discontinuous chains are allowed to jump from one leave to the other. Furthermore, the functoriality for the group $MDH^b_\bullet(\star;A)$ is enhanced thanks to the fact that MD closed chains for parameter $b$ may have ``discontinuities bounded by order $b$''. We introduce the concept of $b$-maps, and the functoriality with respect to $b$-maps gives a lot of flexibility to the theory. A $b$-map or a $b$-moderately discontinuous map from the metric germ $(X,x_0,d_X)$ to the metric germ  $(Y,y_0,d_Y)$ is given by  a collection of subanalytic Lipschitz linearly vertex approaching maps $f_i:C_i\to Y$ where the $\{C_i\}_{i\in I}$ is a subanalytic cover of $(X,x_0)$ and that, roughly speaking, match at the intersections of the $C_i$'s only up to order strictly larger than $b$ (Definition~\ref{def:bmaps}). Therefore they do not need to glue to a continuous map. 


Given an abelian coefficient group $A$, the algebraic category where MD-homology takes values is a diagram of graded abelian groups indexed by $b\in (0,\infty]$, where $MDH^b_\bullet(X,Y,x_0,d_X;A)$ is a graded abelian group, called the $b$-MD Homology group, and for any $b\geq b'$ there is a homomorphism of graded abelian groups $MDH^b_\bullet(X,Y,x_0,d_X;A)\to MDH^{b'}_\bullet(X,Y,x_0,d_X;A).$
These homomorphisms are a very important part of the invariant, and sometimes contain the most interesting information. 

One of the main theorems of the paper is that the graded abelian groups are proved to be finitely generated over $A$ for any germ with the inner or outer metric (Theorem~\ref{theo:finiteness}). There are finitely many jumping rates, that is, finitely many $b\in (0,+\infty]$ for which the homomorphism $MDH^{b+\epsilon}_\bullet(X,x)\to MDH^{b-\epsilon}_\bullet(X,x)$ is not an isomorphism (Theorem~\ref{theo:finitejumps}), and that the set jumps are rational (see Theorem~\ref{theo:rat}, whose proof was kindly provided by A. Parusinski). As a side remark, let us notice that the countability of the rational numbers together with finite generation of the groups shows that, as long as the coefficient group is reasonable, like $\ZZ$ or $\QQ$, there are only countably many diagrams of abelian groups that can be Moderately Discontinuous Homologies of metric subanalytic germs with the inner or outer metric (this agrees with the conjecture of L. Siebenmann and D. Sullivan that there are countably many Lipschitz types of analytic sets~\cite{SS}). 

Versions of the usual properties of homology theories hold for our theory, some follow easily from the definition of the theory, and others require new ideas and constructions: relative long exact sequence (Proposition~\ref{prop:relativehomologysequence}), from which the spectral sequence for a filtration follows; invariance under Lipschitz homotopies and $b$-homotopies (Theorems~\ref{theo:homotopyinvariance} and~Theorem~\ref{theo:homotopyinvariance2}), where a $b$-homotopy is a metric homotopy with $b$-moderated discontinuities; computation of the homology of a ``point'' (the right notion of point in our category is the germ $([0,1),0)$ with the euclidean metric, see~Proposition~\ref{prop:point}). The Mayer-Vietoris sequence is quite subtle, since there is no notion of cover for which it holds for all the parameters $b\in (0,\infty]$ at the same time. Fixed a $b\in (0,\infty]$ we define a notion of $b$-cover of $(X,x_0)$ for which Mayer-Vietoris long exact sequence holds. An Excision Theorem is deduced and the \v{C}ech spectral sequence associated with a $b$-cover is also proved. This parallelism with the shape of usual homology extends to the relation between homology and homotopy: Moderately Discontinuous Homotopy has been defined in the thesis of the second author (much of this work is also contained there), and several of the usual properties of homotopy groups, including Hurewicz homomorphism have been proved in our framework; this is the subject of a forthcoming paper.  

After developing the computational tools described above we turn from Section 9 on to the applications of our theory to the detection of Lipschitz phenomena, which include some useful comparison theorems: 
\begin{enumerate}
\item 
as one can expect, for $b=\infty$, the MD-Homology of $(X,x_0,d_X)$ recovers the usual homology of the link of $X$ (see Theorem \ref{theo:link}). 
\item for any $b$ the group $MDH^b_0(X,x_0,d_X;A)$, is isomorphic to the set of $b$-connected components of $X\setminus\{x_0\}$ (Proposition~\ref{prop:0-homology}, see Definition \ref{def:b-eqcomponents})

\item for $b=1$ and for the outer metric,  MD-Homology of $(X,x_0,d_{out})$ recovers the usual homology of the link of the tangent cone at $X$ (Theorem~\ref{th:speed1}). 
\item for $b=1$, and for the inner metric it recovers the homology of the puntured Gromov tangent cone~(Theorem~\ref{th:speed1inner})
\item 
in Theorem~\ref{th:characsmooth} we prove that if the MD-Homology of a complex analytic germ $(X,x_0,d_{out})$ coincides with that of a smooth germ, then $X$ is smooth. This implies that MD-Homology is strong enough to recover the subanalytic version of the fourth author's theorem in~\cite{BirbrairFLS:2016},~\cite{Sampaio:2016}. 
\item 
we confirm an expectation formulated by L. Birbrair: the $b$-MD homology for the outer metric of a germ $(X,x)$ coincides with the ordinary homology of a suitable punctured $b$-horn neighborhood of $(X,x)$ (Corollary~\ref{cor:birbrair}). 
\item 
we also define an enrichment of the invariant, called Framed MD Homology, which takes into account distinguished basis of fundamental classes in the top homology groups and which allows to prove in Proposition~\ref{cor:mixed_multiplicities} the following: if $(X,x_0)$ is a complex analytic germ, then Framed MD-homology of $(X,x_0,d_{out})$ recovers the number of connected components of the tangent cone and its relative multiplicities (this is a quite rich set of information).
\item We fully compute MD Homology for plane curve singularities with the outer metric. 
\end{enumerate}


Item $(6)$ is quite central since we prove several important results like finitely generation and finiteness of the jumping rates as an applcation of it. In the computation $(8)$ the result is expressed in terms of the Eggers-Wall tree and can be found in Theorem~\ref{th:homology of curves}. As it turns out MD-Homology recovers most of the Eggers-Wall tree: it recovers the set of Puiseux exponents of the branches, and the set of contact exponents, but it does not tell how to distribute these exponents and contacts between the different branches (we show an explicit example that shows that MD Homology does not determine the whole outer Lipschitz geometry of reducible curves). For irreducible branches MD-Homology recovers the whole Lipschitz geometry, since it recovers all the Puiseux exponents.  It is worth to note that the most relevant information is contained in the homomorphisms $MDH^b_\bullet(C,0,d_{out};A)\to MDH^{b'}_\bullet(C,0,d_{out};A)$, rather than on the groups. We prove that Framed MD Homology does determine the whole outer Lipschitz geometry of reducible curves (Corollary~\ref{cor:fremeeggers}).



Let us comment further on the previous metric homology theories (see~\cite{BirbrairBrasselet:2000}~\cite{BirbrairBrasselet:2002}, and~\cite{Valette:2010}). The idea of these theories is, for a given parameter, to determine a set of allowable chains, where a chain is again a family of chains degenerating to the vertex, and a chain is allowable if its size degenerates fast enough with respect to the parameter. In the Birbrair-Brasselet version the size is measured in terms of volume, and in the Valette version the degeneration rate is the speed at which the chain collapses into a smaller dimensional set. In Birbrair-Brasselet theory finite generation is still an open problem and in Valette's one is already proved in his original paper. Our theory has a different flavor, since, instead of restricting the admissible chains by measuring their size, what we do is to allow moderately discontinuous chains (this is achieved technically defining an equivalence relation in the group of chains). This is specially advantageous when one studies subanalytic germs with the outer metric. The possibility of the $b$-cycles to change from leave to leave mentioned above is not present in the previous theories, and has very interesting applications that have already been mentioned: we show that if the MD-Homology of a complex analytic germ with the {\em outer metric} coincides with that of a smooth germ, then $X$ is smooth. We also compute the MD-Homology for complex plane  curves with the outer metric and show that it recovers all Puiseux exponents and most of the information of the embedded topological type. These results do not seem easy to obtain in Birbrair-Brasselet or in Valette theories. An important property of a metric germ $(X,x_0)\subset\RR^n$ is whether it is Lipschitz Normally Embedded, that is, if the inner and outer metric are Lipschitz equivalent. Observe that the identity morphism $Id_X:(X,x_0,d_{inn})\to (X,x_0,d_{out})$ is a morphism in our geometric category. Therefore non-isomorphism of MD-Homology detects non-Lipschitz Normally Embedding, as happens, for example, in the case of plane curves.

Let us add a final remark on generality: although this paper is formulated in the language of subanalytic geometry, all the results, except the rationality of the jumping numbers, work, by a word by word adaptation of the proofs, for any polynomially bounded O-minimal structure over the real numbers. Its is quite likely that there exists adaptations to arbitrary O-minimal structures, and this is subject of further investigation.   

Finally we observe that this work lays the ground for possible future work in different directions. For example, it is interesting to explore how strong our invariants are in obstructing Lipschitz equisingularity in higher dimension (see \cite{Mostowski}). Its relation with Zariski equisingularity is also worth exploring because of the work done in \cite{Neumann:2012} and \cite{Parusinski:2019}. In \cite{Kurdyka:1992}, \cite{Parusinski:1994} and \cite{Valette:2005} subanalytic spaces are decomposed into pieces which are simple from the outer Lipschitz viewpoint. It would be interesting to study the relation of such decompositions for subanalytic germs with our invariants.

We would like to thank Birbrair, Coste, Fernandes, Parusinski and Teissier for useful comments at different stages of this project.

\section{Pairs of metric subanalytic germs}\label{section:defdomaincat}

As usual in algebraic topology, our invariant will be a functor from a category of geometric nature to a category of an algebraic nature. We start defining precisely the geometric category. 

\subsection{Setting}
\label{sec:setting}
Along this paper we work in the language of subanalytic geometry. We will always work with bounded subanalytic subsets, which in particular are globally subanalytic (see \cite{Dries:1986}). Recall that the collection of all globally subanalytic sets forms an O-minimal structure (see \cite{Dries:1986}). The results and properties of the globally subanalytic O-minimal structure that we use are satisfied by any polynomially bounded O-minimal structure over the real numbers. In fact the theory developed in this paper works without any change for any of this O-minimal structures.

We use \cite{Dries:1998} and \cite{Coste:1999} as basic references in O-minimal geometry. A result that is repeatedly used in this paper is Hardt's Triviality Theorem, which is Trivialization Theorem 1.7 in \cite{Dries:1998} (see also Theorem 5.22 in \cite{Coste:1999}).

In fact in this paper we work with subanalytic germs and maps between them. Taking bounded representatives of them we can see them as objects and maps of the globally subanalytic O-minimal category.

Subanalytic triangulations are also essential for us. Given a simplicial complex $K$ we denote by $|K|$ the geometric realization of $K$. We call the subsets of $|K|$, that correspond to a simplex in $K$, the {\em faces of $|K|$}. Let $Z$ be a closed subanalytic set. A {\em subanalytic triangulation} is a finite simplicial complex $K$ of closed simplices and a subanalytic homeomorphism $\alpha:|K|\to Z$. 

\begin{remark}\label{rem:triang} Given a finite family $\calS$ of closed subanalytic subsets of  $Z$,  there exists a subanalytic triangulation $\alpha:|K|\to Z$  compatible with $\calS$, that is, such that every subset of $\calS$ is a union of images of simplices of $|K|$. See for example Theorem 4.4. in  \cite{Coste:1999} or Theorem II.2.1. in \cite{Shiota:1997}.

\end{remark}

By a subanalytic triangulation of a subanalytic germ $(X,x_0)$ we mean a subanalytic triangulation of a representative of it, which is compatible with the vertex.

Given two subanalytic triangulations $\alpha:|K|\to Z$ and $\alpha':|K'|\to Z$, we say that $\alpha'$ \emph{refines} $\alpha$ if the image by $\alpha$ of any simplex of $|K|$ is the  union of images by $\alpha'$ of simplices of $|K'|$. 
 
Given a subanalytic triangulation $\alpha:|K|\to Z$, a simplex of $|K|$ is called {\em maximal} if it is not strictly contained in another simplex. We consider the collection $\calT:=\{T_i\}_{i\in I}$ of subsets of $Z$ that are images of the maximal simplices of $|K|$. We call it the {\em collection of maximal triangles}.

Given two simplicial complexes $K$ and $K'$, a homeomorphism $f:|K|\to |K'|$ \emph{preserves  the simplicial structure}  if it the image of each simplex of $|K|$ is a simplex of $|K'|$.

A result about triangulations that will be important is the subanalytic Hauptvermutung (Chapter II, Theorem II in \cite{Shiota:1997})), which states that two subanalytic triangulations have a common refinement.

\subsection{The category of pairs of metric subanalytic germs}

\begin{definition}\label{def:germ}
A subanalytic germ $(X,x_0)$ is  a germ $(X,x_0)$  of a  subanalytic set $X\subset\RR^m$ such that $x_0\in \overline{X}$ (where $\overline{X}$ denotes the closure of $X$ in $\RR^m$). We say that $x_0$ is the \emph{vertex of $(X,x_0)$}. 

A \emph{metric subanalytic set} $(X,d_X)$  is a subanalytic set $X$ in some $\RR^m$, together with a subanalytic metric $d_X$ that induces the same topology on $X$ as the restriction of the standard topology on $\RR^m$. 

A \emph{metric subanalytic germ} $(X,x_0,d_X)$  is a subanalytic germ $(X,x_0)$ where $(X,d_X)$ is a metric subanalytic set. 
We omit $x_0$ and $d_X$ in the notation when it is clear from the context. 

A {\em metric subanalytic subgerm} of a metric subanalytic germ $(X,x_0,d_X)$ is a metric subanalytic germ $(Y,x_0,d_Y)$ with $Y\subseteq X$ and $d_Y$ equal to the restriction $d_X|_Y$ of the metric $d_X$ to $Y$, that is, the restriction to $Y\times Y$ of $d_X:X\times X\to \RR$. 

A {\em pair of metric subanalytic germs} $(X,Y,x_0,d_X)$ is the metric subanalytic germ $(X,x_0,d_X)$ together with the subgerm $(Y,x_0,d_X|_Y)$.


Given two  germs $(X,x_0)$ and $(Y,y_0)$, a \emph{subanalytic map germ} $f:(X,x_0)\to (Y,y_0)$ is a subanalytic continuous map $f:X\to Y$ that admits a continuous and subanalytic extension to a map germ  $\overline{f}:(X\cup\{x_0\},x_0)\to (Y\cup\{y_0\},y_0)$. 

\end{definition}

\begin{remark}\label{rem:germ0}
Notice that in our definition, for a subanalytic germ $(X,x_0)$ it is possible that $x_0\notin X$. These sets play an important role (see for example Definition \ref{def:semialgcover} or \ref{def:b-saturated} ).
\end{remark}

\begin{example}
\label{ex:inout}
A subanalytic germ $(X,x_0)\subset(\RR^m,x_0)$ with the outer metric (the metric induced by restriction of the euclidean metric in $\RR^m$) is a  metric subanalytic germ. We denote the associated metric subanalytic germ by $(X,x_0,d_{out})$.

We denote by $(X,x_0,d_{in})$ the metric subanalytic germ with the inner metric (defined to be the infimum of the lengths of the rectifiable paths between two points). This distance is not known to be subanalytic. However, according to \cite{KurdykaO:1997}  there is a subanalytic distance $d'$ on $X$ such that the identity $Id: (X,x_0,d_{in})\to (X,x_0,d')$ is bi-Lipschitz. This allows us to apply the theory to the germ $(X,x_0,d_{inn})$ in the following way: our homology can be calculated for $(X,x_0,d')$. Moreover if $d''$ is a different choice of subanalytic metric with the same property than $d'$, then the identity map is a subanalytic bi-Lipschitz homeomorhism between $(X,x_0,d')$ and $(X,x_0,d'')$. Hence the invariant calculated to each of the two subanaytic metric germs is the same.
See Remark \ref{rem:Cat_inner}) for an extension of this idea.

\end{example}

Some basic examples are the following: 

\begin{definition}[Standard $b$-cones and straight cones]
\label{def:stdcones}
Let $L\subset \RR^k$ be a subanalytic set and $b\in\cap (0,+\infty )$. Consider the subanalytic set
$$
C_L^b=\{(t^bx,t)\in \RR^k\times \RR;\, x\in L\mbox{ and }t\in [0,+\infty )\}.
$$
The {\em outer (respectively inner) standard $b$-cone over} $L$ is the triple $(C_L^b, (\underline{0},0), d_{out})$ (respectively $(C_L^b, (\underline{0},0), d_{in})$), where $d_{out}$ denotes the outer metric and $d_{in}$ denotes the inner metric. 

When $b=1$, we say $C_L^1$ is a \emph{straight cone over $L$} and we denote it by $(C(L),d_{out}):=(C_L^1,(\underline{0},0),d_{out})$ and $(C(L),d_{in}):=(C_L^1,(\underline{0},0),d_{in})$. 

By $C_L^b$ or $C(L)$ we always mean the germ $(C_L^b,(\underline{0},0))$ or $(C(L),(\underline{0},0))$.

If $b\in\QQ$ the standard cones are subanalytic, and otherwise they are at least definable in the $O$-minimal structure $\RR_{an}^{\RR}$ (see \cite{Miller}).
\end{definition}

\begin{definition}
\label{def:b-horn}
Let $(X,x_0,d_X)$ be a metric subanalytic germ and $Y\subset X$ a subanalytic subgerm. Let $b \in (0, \infty)$. The {\em $b$-horn neighborhood of amplitude $\eta$ of $Y$ in $X$} is the subset 
$$
\calH_{b,\eta}(Y; X):=\bigcup_{y\in Y}B(y,\eta \|y - x_0\|^b),
$$
where $B(y,\eta \|y - x_0\|^b):=\{x\in X:d_X(x,y) < \eta \|y - x_0\|^b\}$ denotes the ball in $X$ centered in $y$ of radius $\eta d_X(y,x_0)^b$. The $\infty$-horn neighborhood $\calH_{b,\eta}(Y; X)$ is defined to be $Y$.  
\end{definition}

\begin{definition}[Spherical Horn Neighborhood]
\label{def:sphhorn}
Let $X$ be a subanalytic germ embedded in $\RR^n$. We assume the vertex of the cone to be the origin in $\RR^n$. Let $b\in\RR^{+}$. The {\em spherical $b$-cone neighborhood} of amplitude $\eta$ of $X$ in $\RR^n$ is the union 
$$S\calH_{b,\eta}(X):=\bigcup_{x\in X}B(x,\eta \|x\|^b)\cap \SSS_{||x||},
$$
where $\SSS_{||x||}$ denotes the open sphere of radius $||x||$ centered at the origin.
\end{definition}

\begin{remark}\label{poly_bounded}
If $b\in\QQ$ then the $b$-horns and $b$-spherical horns are subanalytic. For an arbitrary $b\in\RR^+$,  $S\calH_{b,\eta}(X)$ they are definable in the polynomially bounded $O$-minimal structure $\RR_{an}^{\RR}$ (see \cite{Miller}).
\end{remark}


\begin{remark}\label{rem:germ1} We recall that the link of a subanalytic germ is well defined as a topological space as the intersection of $X$ with a small enough sphere centered at $x_0$; we denote it  by $Link(X,x_0)$ or simply $L_X$. Moreover, the conical structure theorem says, given a subanalytic germ $(X,x_0)$ and a family of subanalytic subgerms $(Z_1,0)$,...,$(Z_k,0)\subseteq(X,0)$, that there exists a subanalytic homeomorphism $h: C(L_X)\to (X,x_0)$ such that $||x_0-h(tx,t)||=t$ and such that $h(C(L_{Z_i}))=Z_i$ with $L_{Z_i}$ in $L_X$ (see Theorem 4.10, 5.22, 5.23 in \cite{Coste:1999} in combination with the fact that globally subanalytic sets form an O-minimal structure). We say that the conical structure $h$ is compatible with the family $\{Z_i \}$. The conical structure is why we say that $x_0$ is the vertex of $(X,x_0)$.   
\end{remark}

\begin{notation}
Let $(X, x_0, d_X)$ be a metric subanalytic germ. When we need to specify the radius, we denote by $L_{X,\epsilon}$ the link  $\{x \in X : ||x-x_0||= \epsilon \}$.
\end{notation}

Let us add that in Proposition 1 of~\cite{CosteKurdyka}, when $X$ is semialgebraic it is proved that the link is well defined up to  semialgebraic homeomorphisms. However, this fact is not used along this paper.

\begin{definition}
\label{def:lva}

A map germ  $f: (X, x_0) \rightarrow (Y, y_0)$ is said to be \emph{linearly vertex approaching} (l.v.a. for brevity) if there exists $K\geq 1$ such that $$\frac{1}{K}||x-x_0||\leq ||f(x)-y_0||\leq K||x-x_0||$$ for every $x$ in some representative of $(X,x_0)$. {The constant $K$ is called the l.v.a constant for $f$}. 
\end{definition}

\begin{remark} 
\label{rem:lvareduction}
Let $(X,x_0)$ be a subanalytic germ with compact link. Consider any subanalytic map germ $f: (X, x_0) \rightarrow (Y, y_0)$ that is a homeomorphism onto its image. Let $\{Z_j\}_{j\in J}$ be a finite collection of closed subanalytic subsets of $X$. There is a subanalytic homeomorphism germ $\phi:(X,x_0)\to (X,x_0)$ such that $\phi(Z_j)=Z_j$ for all $j\in J$ and such that $||f\comp \phi(x)-y_0||=||x-x_0||$, which is stronger than l.v.a.
\end{remark}
\proof


%
Let $h:C(L_X)\to (X,x_0)$ be a subanalytic homemomorphism defining the conical structure compatible with the $Z_i$ (which means that $h(C(L_{Z_i}))=Z_i$) and such that $||h(tx,t)-x_0||=t$ (see Remark \ref{rem:germ1}). 

Consider the mapping $g: C(L_X)\to C(L_X)$ that sends $(xt,t)\mapsto (x\cdot ||f\circ h(xt,t)-x_0||, ||f\circ h(xt,t)-x_0||)$.
It is clearly subanalytic in the coordinates $y=xt$ and $t$ for $t\neq 0$ and therefore it extends continuously and subanalytically to the closure $C(L_X)$. Note that $g$ is a  homeomorphism.

%
%
To finish, it is clear that $\phi:=h\circ{g}^{-1}\circ h^{-1}$ satisfies the statement. 
\endproof

Remark~\ref{rem:lvareduction} can also be shown adapting the following result of Shiota's:
\begin{customcor}{2 of~\cite{Shiota:1989}}
Let $f_1$ and $f_2$ be subanalytic functions on $X$ with
$$
f_1^{-1}(0) = f_2^{-1}(0), \quad \{f_1 < 0 \} = \{f_2 < 0 \}, \quad
\{f_1 > 0 \} = \{f_2 > 0 \}.
$$ 
Then there exists a subanalytic homeomorphism $\phi$ of $X$ such that
$$
f_1 \comp \phi = f_2
$$
on a neighborhood of $f_1^{-1}(0)$.
\end{customcor}

We assume $x_0 = 0$. If the family $\{Z_j\}_{j\in J}$ is empty, we simply apply Corollary 2 of~\cite{Shiota:1989} to the functions $||x||$ and $||f(x)||$. The proof of Corollary 2~\cite{Shiota:1989} only uses the  subanalytic triangulation Theorem (Theorem 1 of~\cite{Shiota:1989}) together with Lemmata 10 and 11 in the same paper, which are stated in the presence of the family $\{Z_j\}_{j\in J}$. Notice that the subanalytic triangulation Theorem (Theorem 1 of~\cite{Shiota:1989}) is valid when the family $\{Z_j\}_{j\in J}$ is non-empty (this is Theorem~II of Chapter~II of~\cite{Shiota:1997}). So Remark~\ref{rem:lvareduction} is true without the assumption that $f$ is a homeomorphism onto its image.

We have preferred to give a simple proof of the case that is used in this paper for the sake of completeness, i.e. including the assumption that $f$ is a homeomorphism onto its image.



\begin{definition}
\label{def:metricmaps}
Let $(X, x_0, d_1)$ and $(Y, y_0, d_2)$ be two metric subanalytic germs. A \emph{Lipschitz linearly vertex approaching subanalytic map germ} (\emph{Lipschitz l.v.a. subanalytic map} for short) 
\[
f: (X, x_0,d_X) \rightarrow (Y, y_0,d_Y)
\]
is a l.v.a subanalytic map germ such that there exists $K \geq 1$ and a representative $X$ of the germ such that 
\[
d_Y (f(x), f(\tilde x)) \leq K d_X (x, \tilde x) \ \forall x, \tilde x \in X.
\]

Any such $K$ {which also serves as a l.v.a. constant for $f$} will be called a \emph{Lipschitz l.v.a. constant for $f$}. 

A {\em Lipschitz l.v.a. subanalytic map of pairs} is a  map germ of pairs 
$$f:(X,Y,x_0,d_X)\to (X',Y',x'_0,d_{X'})$$
such that $f:(X,x_0,d_X)\to (X',x'_0,d_{X'})$ is a Lipschitz l.v.a. subanalytic map.  
\end{definition}

Given a subanalytic subgerm $(Y,x_0,d_X|_Y)\subset (X,x_0,d_X)$, the inclusion is an example of a Lipschitz l.v.a. map.

\begin{definition}\label{def:pairssubangerms}
The \emph{category of pairs of subanalytic metric subanalytic germs} has pairs of metric subanalytic germs as objects and Lipschitz l.v.a. subanalytic maps of pairs as morphisms.
\end{definition}

\begin{remark}\label{rem:Cat_inner}Equivalently, we can work with the bigger category of subanalytic germs $(X,x_0,d_X)$ which are endowed with a metric $d_X$ that induce the same topology as the euclidean metric, and such that  there exists a subanalytic metric $d'$ that is bi-Lipschitz equivalent to $d_X$, which means that the identity $(X,x_0,d_X)\to (X,x_0,d')$ is bi-Lipschitz. Hence we are not asking $d_X$ to be a subanalytic metric. Then, a subanalytic germ $(X,x_0)$ with the inner metric belongs to this category, see Example~\ref{ex:inout}.
\end{remark}

\section{Definition of the Moderately Discontinuous Homology}
\label{sec:defi}

The Moderately Discontinuous Homology (Moderately Discontinuous Homology, or MD-Homology, for short) is a functor from the category of pairs of metric subanalytic germs to an algebraic category whose objects are diagrams of groups. Its definition needs a series of steps.

\subsection{The pre-chain group \texorpdfstring{$MDC^{\mathrm{pre},\infty}_\bullet((X,x_0);A)$}{MDC\{pre,infinity\}((X,x0);A)}}

\begin{notation}\label{not:simplice}
For any $n\in \mathbb{N}_0$, we denote by $\Delta_n \subset \mathbb{R}^{n+1}$ the standard $n$-simplex
\[
\Delta_n := \{ (p_0, ..., p_n) \in (\mathbb{R}_{\geq 0})^{n+1} : \sum_{i=0}^n p_i = 1 \}
\]
oriented as follows: the standard orientation on $\mathbb{R}^{n+1}$ orients the convex hull of $\Delta_n \cup \underline{0}$, where $\underline{0}$ denotes the origin, which in turn induces an orientation on $\Delta_n$, viewed as part of the boundary of the convex hull. We denote by $i_{n}^k:\Delta_{n-1}\to \Delta_n$ the map sending $p_0,...,p_{n-1}$ to $p_0,...,p_{k-1},0,p_k,...,p_{n-1}$. The image of $i_{n}^k$ is the $k$-th facet of $\Delta_n$.

Consider the {oriented germ} of the cone over $\Delta_n$ and denote it as $$\hat{\Delta}_n:=(\{(tx,t)\in \RR^{n+1}\times\RR: x\in\Delta_n, t\in [0,1)\},$$
and let $j^k_n:\hat{\Delta}_{n-1}\to \hat{\Delta}_{n}$ be the map sending $(tx,t)\mapsto (ti^k_n(x),t)$. The {\em $k$-th facet} of $\hat{\Delta}_{n}$ is the image of $j^k_n$. More generally, a face of $\hat{\Delta}_n$ is the cone over a face of $\Delta_n$.

We will usually use $\hat{\Delta}_n$ to denote the germ $(\hat{\Delta}_n,(\underline{0},0))$. 
\end{notation}



 
Coherently to  Definition \ref{def:lva} we have the notion of {\em linearly vertex approaching (subanalytic) n-simplex}, which is a continuous subanalytic map germ $\sigma:\hat{\Delta}_n\to ({X},x_0)$ 
such that there is a $K \geq 1$ such that 
\[
\frac{1}{K} t \leq ||\sigma(xt,t)-x_0|| \leq K t 
\]
for any $x \in \Delta_n$ and any small enough $t$. We will say simply a \emph{l.v.a. simplex}. 

Similarly, a map $\nu:\hat{\Delta}_n\to \hat{\Delta}_n$, expressed as $\nu(xt,y)=(\nu_1(xt,t)\nu_2(xt,t),\nu_2(xt,t))$ in the coordinates $(xt,t)$ of $\hat{\Delta}_n$, is linearly vertex approaching (according to Definition \ref{def:lva}) if there is a $K \geq 1$ such that $\frac{1}{K} t \leq \nu_2(xt,t) \leq K t$.

\begin{definition}
\label{def:preinfhomology}
Given a subanalytic germ $(X,x_0)$ and an abelian group $A$, a {\em linearly vertex approaching $n$-chain in $(X,x_0)$} (l.v.a. $n$-chain, for brevity) is a finite formal sum $\sum_i a_i \sigma_i$, where 
$a_i\in A$ and $\sigma_i$ is a l.v.a subanalytic $n$-simplex in $(X,x_0)$. We define $MDC^{\mathrm{pre},\infty}_n((X,x_0);A)$ to be the abelian group of $n$-chains. 

We define the boundary of $\sigma$ to be the formal sum 
$$
\partial \sigma =\sum_{k = 0}^n (-1)^k  \sigma\comp j_n^k.
$$

The boundary extends linearly to $n$-chains and defines a complex $MDC^{\mathrm{pre},\infty}_\bullet((X,x_0);A)$ whose components are the groups $MDC^{\mathrm{pre},\infty}_n ((X,x_0);A)$ for $n\geq 0$. 

Often, when it is clear from the context we will skip the coefficients group $A$ and/or the vertex in the notation.
\end{definition}

\subsection{The homological subdivision equivalence relation in \texorpdfstring{$MDC^{\mathrm{pre},\infty}_\bullet((X,x_0);A)$}{MDC\{pre,infinity\}((X,x0);A)}} \label{subsection:subdivision}

As in Singular Homology Theory, in order to prove Excision and Mayer-Vietoris we will need to subdivide simplices. In Singular Homology, the standard procedure is to divide a simplex into a chain of smaller simplices by taking barycentric subdivisions. The existence of a Lebesgue number in that context guarantees that iterating that procedure enough times yields a chain for which all of its simplices are contained in one of the open sets of the cover. In our theory, the role of open subgerms are taken by subgerms whose representatives are open and whose closure contains the vertex, but that do not contain the vertex themselves. Observe that those are the complements of subgerms whose representatives are closed and contain the vertex. Therefore an open cover of germs does not cover the image of a l.v.a. simplex: the image of the vertex is remains outside of all the sets of the cover. As a result we do not get a Lebesgue number. That is why we do not adapt the procedure used in Singular Homology, but build a chain complex that incorporates the subdivisions from the beginning.


In the next definition we will need a representative of the germ $\hat{\Delta}_n$. By abuse of notation we denote it also by $\hat{\Delta}_n$, and consider the representative
$$\{(tx,t)\in \RR^{n+1}\times\RR: x\in\Delta_n, t\in [0,1/2]\}.$$

Recall the facts of subanalytic triangulations stated in Section~\ref{sec:setting}.

\begin{definition}
\label{def:subdiv}
 A {\em homological subdivision} of $\hat{\Delta}_n$ is a finite  family $\{\rho_i\}_{i\in I}$ of injective l.v.a. subanalytic map germs $\rho_i:\hat{\Delta}_n\to \hat{\Delta}_n$ for which there is a subanalytic triangulation $\alpha:|K|\to\hat{\Delta}_n$ with the following properties: 
\begin{itemize}
\item the triangulation $\alpha$ is compatible with the collection of all faces of $\hat{\Delta}_n$;
\item all maximal triangles of $\alpha$ meet the vertex of $\hat{\Delta}_n$;
\item the collection $\{ T_i \}$ of maximal triangles of $\alpha$ is also indexed by $I$;
\item for any $i \in I$, the image of $\rho_i$ is $T_i$ and $\alpha^{-1}|_{T_i}\comp \rho_i$ is a homeomorphism onto its image that takes faces of $\hat{\Delta}_n$ to faces of $|K|$.
\end{itemize} 

For a homological subdivision $\{\rho_i\}_{i\in I}$, the sign of $\rho_i$ for any $i \in I$ is defined to be $1$, if $\rho_i$ is orientation preserving, and $-1$, if it is orientation reversing. We denote it by $sgn(\rho_i)$.
\end{definition}

Note that this implies that, whenever the sum $\sum_{i\in I}sgn(\rho_i)\rho_i$ is a cycle in the singular homology
$H_{n+1}(\hat{\Delta}_n,\partial\hat{\Delta}_n\cup\hat{\Delta}_n^{\geq\epsilon} ;\ZZ)$ for $\epsilon=min \{\epsilon_i\}$ where $\hat{\Delta}_n^{\geq\epsilon}$ denotes the set $\{(tx,t)\in \RR^{n+1}\times\RR: x\in , t\in [\epsilon,1)\}$, then $\sum_{i\in I}sgn(\rho_i)\rho_i$ represents the fundamental class. However, in general $\sum_{i\in I}sgn(\rho_i)\rho_i$ does not have to represent a cycle in $H_{n+1}(\hat{\Delta}_n,\partial\hat{\Delta}_n\cup\hat{\Delta}_n^{\geq\epsilon} ;\ZZ)$.

\begin{definition}[Immediate equivalences]
\label{def:Inm}
Two chains $$\sum_{j\in J}a_j\sigma_j, \sum_{k\in K}b_k\tau_k\in MDC_n^{\mathrm{pre},\infty}(X;A)$$ are called {\em immediately equivalent} (and we denote it by $\sum_{j\in J}a_j\sigma_j\to_{\infty} \sum_{k\in K}b_k\tau_k$), if for any $j\in J$ there are homological subdivisions $\{\rho_{ji}\}_{i\in I_j}$ such that we have the equality
$$\sum_{j\in J}\sum_{i\in I_j} sgn(\rho_{ji}) a_j\sigma_j\comp\rho_{ji}=\sum_{k\in K} b_k\tau_k$$
in $MDC_n^{\mathrm{pre},\infty}(X;A)$.
\end{definition}

\begin{remark}
\label{rem:immediatesimplex}
The immediate equivalences can be defined as well by imposing $\sigma\to_\infty \sum_{i\in I} sgn(\rho_i) \sigma\comp\rho_i$ for any l.v.a $n$-simplex $\sigma$ and any subdivision $\{\rho_i\}_{i\in I}$, and extending the immediate equivalences by linearity. 
\end{remark}

\begin{remark}\label{rem:repar} Any l.v.a. subanalytic homeomorphism $\mu:(\hat{\Delta}_n,0) \to (\hat{\Delta}_n,0) $ which preserves the simplicial structure is a homological subdivision of $\hat{\Delta}_n$ for which the index set $I$ has just one element. As a consequence, for any $n$-simplex $\sigma$, we have $\sigma \to_\infty \sigma \comp \mu$, if $\mu$ is orientation preserving, and $\sigma \to_\infty - \sigma \comp \mu$, if $\mu$ is orientation reversing.
\end{remark}

\begin{definition}[The homological subdivision equivalence relation]\label{def:Seq}
The {\em subdivision equivalence relation} in $MDC_n^{\mathrm{pre},\infty}(X;A)$ (denoted by $\sim_{S,\infty}$) is the equivalence relation generated by immediate equivalences. That is $z\sim_{S,\infty} z'$ if there exists a sequence $w_1,...,w_k$ such that $z=w_1$, $z'=w_k$ and for any $1\leq i<k$ we have either the immediate equivalence $w_i\to_\infty w_{i+1}$ or $w_{i+1}\to_\infty w_i$. 
\end{definition}

\begin{lema}
\label{lem:Seqinversion}
Given any three chains $w_1,w_2,w_3\in MDC_n^{\mathrm{pre},\infty}(X;A)$, and immediate equivalences $w_3\to_{\infty} w_1$ and $w_3\to_{\infty} w_2$ there exists an element 
$w_4\in MDC_n^{\mathrm{pre},\infty}(X;A)$ and two immediate equivalences $w_1\to_{\infty} w_4$ and $w_2\to_{\infty} w_4$. 
\end{lema}
\proof
 Since the immediate equivalences are compatible with linear combinations (see Remark~\ref{rem:immediatesimplex}) we may assume that $w_3$ is equal to an $n$-simplex $\sigma$. Then there exist two subdivisions $\{\rho_i\}_{i\in I}$ and $\{\rho'_{i'}\}_{i'\in I'}$ of $\hat{\Delta}_n$ such that we have the equalities
\begin{equation}
\label{eq:updown1}
w_1=\sum_{i\in I}sgn(\rho_i)\sigma\comp\rho_i,\quad\quad
w_2=\sum_{i'\in I'}sgn(\rho'_{i'})\sigma\comp\rho'_{i'}.
\end{equation}

Let $\alpha:|K|\to \hat{\Delta}_n$ and $\alpha':|K'|\to \hat{\Delta}_n$ be the subanalytic triangulations associated with the homological subdivisions $\{\rho_i\}_{i\in I}$ and $\{\rho'_{i'}\}_{i'\in I'}$. By the subanalytic Hauptvermutung (Chapter II, Theorem II in \cite{Shiota:1997})) there is a subanalytic triangulation $\beta:|L|\to\hat{\Delta}_n$ refining $\alpha$ and $\alpha'$. Let $\{T_j\}_{j\in J}$ be the collection of maximal triangles of the triangulation $\beta$. Let $\{\nu_j\}_{j\in J}$ be a collection of  orientation preserving l.v.a. subanalytic homeomorphisms $\nu_j:\hat{\Delta}_n\to T_j$ preserving the simplicial structure.

Consider the splitting $J=\coprod_{i\in I}J_i$, where $j\in J_i$ if and only if $T_j$ is included in the image of $\rho_i$.

Then the collection $\{\rho_i^{-1}\comp\nu_j\}_{j\in J_i}$ is a homological subdivision of $\hat{\Delta}_n$ and we have the immediate equivalence
\begin{equation}
\label{eq:updown2}
\sigma\comp\rho_i\to_{\infty}\sum_{j\in J_i}\sigma\comp\rho_i\comp\rho_i^{-1}\comp\nu_j=\sum_{j\in J_i}\sigma\comp\nu_j.
\end{equation}

The splitting $J=\coprod_{i'\in I'}J'_{i'}$ is defined considering the analogous interaction between the triangulations $\alpha'$ and $\beta$. By the same kind of arguments we have the immediate equivalence
\begin{equation}
\label{eq:updown3}
\sigma\comp\rho'_i\to_{\infty}\sum_{j\in J'_i}\sigma\comp\nu_j.
\end{equation}

Defining $w_4:=\sum_{j\in J}\sigma\comp\nu_j$ and using Equations~(\ref{eq:updown1}), ~(\ref{eq:updown2}) and ~(\ref{eq:updown3}) we complete the proof.  
\endproof

\begin{cor}
\label{cor:downup}
We have the equivalence $w\sim_{S,\infty} z$ if and only if there exist sequences of immediate equivalences $z=z_0\to_{\infty} z_1\to_{\infty} ...\to_{\infty} z_l$ and $w=w_0\to_{\infty} w_1\to_{\infty} ...\to_{\infty} w_m=z_l$.
\end{cor}
\proof
The sequence $z = x_1,...,x_k = w$ predicted in Definition~\ref{def:Seq} is {\em monotonous} at the $i$-th position if we have either $x_{i-1}\to_\infty x_i\to_\infty x_{i+1}$ or $x_{i+1}\to_\infty x_i\to_\infty x_{i-1}$. The sequence $x_1,...,x_k$ {\em has a roof} at the $i$-th position if we have $ x_{i}\to_\infty x_{i-1}$ and $x_i\to_\infty x_{i+1}$. The sequence $x_1,...,x_k$ {\em has a valley} at the $i$-th position if we have $ x_{i-1}\to_\infty x_{i}$ and $ x_{i+1}\to_\infty x_{i}$. Repeated applications of the previous lemma  allow to replace every roof by a valley.
\endproof

\subsection{\texorpdfstring{$\infty$}{infinity}-Moderately discontinuous homology}

\begin{lema}
\label{lem:compatpartialsum}
 The homological subdivision equivalence relation is compatible with the boundary $\partial$ in $MDC_n^{\mathrm{pre},\infty}(X;A)$ in the following sense: given a simplex $\sigma\in MDC_n^{\mathrm{pre},\infty}(X;A)$  and $\{\rho_i\}_{i\in I}$ a homological subdivision of $\hat{\Delta}_n$, then the chains
$\partial\sigma$ and $\sum_{i\in I}\partial( sgn(\rho_i)\sigma\comp\rho_i)$ are $\sim_{S,\infty}$-equivalent. 
\end{lema}
\proof

A homological subdivision $\{\rho_i\}_{i \in I}$ of $\hat{\Delta}_n$ induces a homological subdivision $\{\rho_i^k\}_{i \in I_k}$ of the $k$-th facet of $\hat{\Delta}_n$. So 
$
\partial(\sum_{i\in I} sgn(\rho_i)\sigma\circ\rho_i)
$
splits as the sum, with appropriate signs, of the facets of $\sigma$ (expressed after the corresponding homological subdivision) and the sum of the interior facets of all $\sigma\circ\rho_i$ which cancels in pairs. 
\endproof

\begin{definition}[$\infty$-MD Homology]
\label{def:infty}
We define the {\em $\infty$-moderately discontinuous chain complex of} $(X,x_0,d_X)$ with coefficients in $A$ ($\infty$-MD complex for short) to be the quotient of $MDC^{\mathrm{pre},\infty}_\bullet((X,x_0,d_X);A)$ by the homological subdivision equivalence relation. We denote it by $MDC^\infty_\bullet((X,x_0,d_X);A)$. Its homology is called the {\em $\infty$-moderately discontinuous homology with coefficients in $A$} and is denoted by $MDH^\infty_\bullet((X,x_0,d_X);A)$.
\end{definition}

Note that this homology does not depend on a metric. As we will see in Section \ref{sec:link}, the $\infty$-moderately discontinuous homology coincides with the homology of the link of the germ $(X,x_0)$. 


\subsection{\texorpdfstring{$b$}{b}-Moderately discontinuous homology}\label{subsection:b-equrel}

Given a metric subanalytic germ $(X,x_0,d_X)$, for each $b\in (0,+\infty)$, we define the following equivalence relation in the set of l.v.a. subanalytic $n$-simplices:

\begin{definition}
\label{def:bequivalent} Let $b\in (0,\infty)$.
Let $\sigma_1,\sigma_2$ be $n$-simplices in $MDC^{\mathrm{pre},\infty}_\bullet(X,x_0,d_X)$. We say that $\sigma_1$ and $\sigma_2$ are {\em $b$-equivalent} (we write $\sigma_1\sim_b\sigma_2$) if 
$$
\lim\limits_{t\to 0^+}\frac{\max\limits\{d_X(\sigma_1(tx,t),\sigma_2(tx,t));x\in \Delta_n\}}{t^b}=0.
$$

We extend the relation to $MDC_n^{\mathrm{pre},\infty}(X;A)$ by linearity.
\end{definition}

\begin{remark}Note that in the case of a subanalytic germ with the outer metric, the $b$-equivalence can be interpreted in terms of the contact order of the arcs $\sigma_i(tx,x)$ for $x\in \Delta_n$.
\end{remark}

The importance of $b$-(spherical) horn neighbourhoods in our theory is due to the following:
\begin{remark}\label{rem: b-equ b-horn}
Any l.v.a. simplex that is $b$-equivalent to a l.v.a. simplex whose image is contained in $Y$ is contained in any $b$-(spherical) horn neighborhood  $Y$ in $X$.
\end{remark}

\begin{remark}\label{rem:splitindexset}
The quotient of the free group $MDC^{\mathrm{pre},\infty}_\bullet((X,x_0,d_X);A)$ by the {$\sim_b$-equivalence} relation is the free group generated by the $\sim_b$-equivalence classes of simplices with coefficients in $A$. As a consequence we have the following: let $w=\sum_{j\in J}b_{j}\tau_{j}$ and $w'=\sum_{j\in J'}b'_{j}\tau'_{j}$ be chains in $MDC^{\mathrm{pre},\infty}_\bullet((X,x_0,d_X);A)$. Split the index sets $J=\coprod_{k\in K}J_k$ and $J'=\coprod_{k\in K}J'_k$ in the unique way that satisfies the following properties:
\begin{itemize}
 \item any two $j_1,j_2\in J$ belong to the same $J_k$ if and only if we have $\tau_{j_1}\sim_b\tau_{j_2}$,
 \item any two $j'_1,j'_2\in J$ belong to the same $J'_k$ if and only if we have $\tau'_{j'_1}\sim_b\tau'_{j'_2}$,
 \item for any $k\in K$ and $j\in J_k$ and $j'\in J'_k$ we have $\tau_{j}\sim_b\tau'_{j'}$.
\end{itemize}
Then $w\sim_b w'$ if and only if for any $k\in K$ we have the equality
\begin{equation}
\label{eq:splitident}
\sum_{j\in J_k}b_j=\sum_{j'\in J'_k}b_{j'}.
\end{equation}

\end{remark}

The following arc interpretation of the $b$-equivalence relation will be useful later. 

\begin{lema}
\label{lem:arccharacterization}
Let $\sigma_1,\sigma_2$ be $n$-simplices in $MDC^{\mathrm{pre},\infty}_\bullet(X,x_0,d_X)$. Then we have that the following statements are equivalent:
\begin{itemize}
 \item [(i)] $\sigma_1\sim_b\sigma_2$;
 \item [(ii)] for any  continuous subanalytic arc $\gamma:(0,\epsilon)\to \hat{\Delta}_n$ such that $\gamma(0)$ is equal to the vertex and $\gamma(t)$ is different to the vertex for $t\neq 0$ we have the equality
\begin{equation}
\label{eq:arclimit_2}
\lim\limits_{t\to 0^+}\frac{d(\sigma_1(\gamma(t)),\sigma_2(\gamma(t)))}{\gamma_2(t)^b}=0,
\end{equation}
where $\gamma(t)=(\gamma_2(t)\gamma_1(t),\gamma_2(t))$ is the expression of the arc in the coordinates $(tx,t)$ of $\hat{\Delta}_n$;
 \item [(iii)] for any  l.v.a continuous subanalytic arc $\gamma:(0,\epsilon)\to \hat{\Delta}_n$ we have the equality
\begin{equation}
\label{eq:arclimit}
\lim\limits_{t\to 0^+}\frac{d(\sigma_1(\gamma(t)),\sigma_2(\gamma(t)))}{t^b}=0.
\end{equation}
\end{itemize}
\end{lema}
\proof
If we have the equivalence $\sigma_1\sim_b\sigma_2$ it is obvious that the limit vanishes for any arc as in the statement of (ii). Let $\gamma(t)=(\gamma_2(t)\gamma_1(t),\gamma_2(t))$ be any subanalytic l.v.a continuous arc. Then the limit $\lim\limits_{t\to 0^+}\frac{\gamma_2(t)}{t}$ is finite, and therefore condition $(ii)$ implies condition $(iii)$.


So, to finish the proof, we only need to prove that (iii) $\Rightarrow$ (i). Assume that the condition on arcs in (iii) is satisfied.
The function $$t\mapsto \max\limits\{d(\sigma_1(tx,t),\sigma_2(tx,t));x\in \Delta_n\}$$ is subanalytic. Therefore it admits an expansion of the form  $$\max\limits\{d(\sigma_1(tx,t),\sigma_2(tx,t));x\in \Delta_n\}=Ct^{b'}+o(t^{b'})$$ 
for a certain $b'\in\QQ$ and $C>0$. Then the subset 
$$Z:=\{(tx,t)\in\hat{\Delta}_n:d(\sigma_1(tx,t),\sigma_2(tx,t))\geq (C/2)t^{b'}\}$$
is subanalytic and contains sequences converging to the vertex of $\hat{\Delta}_n$. Therefore, by the subanalytic Curve Selection Lemma there exists a  continuous subanalytic arc $\gamma:(0,\epsilon)\to Z$ such that $\gamma(0)$ is equal to $x_0$ and $\gamma(t)$ is different to the vertex for $t\neq 0$. By Remark \ref{rem:lvareduction} we can assume that $||\gamma(t)-x_0||=t$. Thus, we have the following inequality 
$$
\lim\limits_{t\to 0^+}\frac{d(\sigma_1(\gamma(t)),\sigma_2(\gamma(t)))}{t^{b'}}\geq C/2.
$$
The equivalence $\sigma_1\sim_b\sigma_2$ holds if and only if we have the strict inequality $b'>b$. The previous inequality implies that if $b'\leq b$ then the  continuous subanalytic arc $\gamma$ contradicts the arc condition in the statement of (iii).
\endproof


\begin{lema}
\label{lem:repar}
Let $\sigma, \sigma'$ be   simplices in $MDC_n^{\mathrm{pre},\infty}(X;A)$ such that $\sigma\sim_b\sigma'$. If $\{\rho_i\}_{i\in I}$ is a homological subdivision of $\hat{\Delta}_n$, then we have 
$$\sum_{i\in I} sgn(\rho_i)\sigma\comp\rho_i\sim_b \sum_{i\in I} sgn(\rho_i)\sigma'\comp\rho_i.$$
\end{lema}

\proof
Let $\gamma$ be any subanalytic l.v.a continuous arc in $\hat{\Delta}_n$. Since $\rho_i$ is l.v.a then  $\rho_i\comp\gamma$ is also a subanalytic l.v.a continuous arc. Since we have the equivalence $\sigma\sim_b\sigma'$, Lemma~\ref{lem:arccharacterization} implies the vanishing of the limit $$\lim\limits_{t\to 0^+}\frac{d(\sigma(\rho_i(\gamma(t))),\sigma'(\rho_i(\gamma(t))))}{t^b}=0.$$
Again by Lemma~\ref{lem:arccharacterization} this implies the equivalence $\sigma\comp\rho_i\sim_b\sigma'\comp\rho_i$.
\endproof

In order to define the complex of $b$-moderately discontinuous chains we introduce the $b$-subdivision equivalence relation.  

\begin{definition}[The $b$-subdivision equivalence relation]
\label{def:Seqb}
Two chains 
$$\sum_{j\in J}a_j\sigma_j, \sum_{k\in K}b_k\tau_k\in MDC_n^{\mathrm{pre},\infty}(X;A)$$
are called {\em $b$-immediately equivalent} (and we denote it by $\sum_{j\in J}a_j\sigma_j\to_{b} \sum_{k\in K}b_k\tau_k$), if for any $j\in J$ there is a homological subdivision $\{\rho_i\}_{i\in I_j}$ such that we have the $b$-equivalence

$$\sum_{j\in J}\sum_{i\in I_j} sgn(\rho_i) a_j\sigma_j\comp\rho_i\sim_b\sum_{k\in K} b_k\tau_k$$
in $MDC_n^{\mathrm{pre},\infty}(X;A).$

The {\em $b$-subdivision equivalence relation} in $MDC_n^{\mathrm{pre},\infty}(X;A)$ is the equivalence relation generated by the $b$-immediate equivalences, and is denoted by $\sim_{S,b}$. The equivalence classes are called $b$-moderately discontinuous chains or $b$-chains.

We denote by $MDC^b_n(X;A)$ the quotient group   of $MDC_n^{\mathrm{pre},\infty}(X;A)$ by the $\sim_{S,b}$- equivalence relation. It is the group of $b$-moderately discontinuous chains. 
\end{definition}

\begin{prop}
\label{prop:downupb}
$w\sim_{S,b} z$ if and only if there exist sequences of $b$-immediate equivalences $w=w_0\to_{b} w_1\to_{b} ...\to_{b} w_l$ and $z=z_0\to_{b} z_1\to_{b} ...\to_{b} z_m=w_l$.
\end{prop}
\proof  
%
The proof is an adaptation of the proofs of Lemma~\ref{lem:Seqinversion} and Corollary \ref{cor:downup}, taking into account Lemma~\ref{lem:repar}.
\endproof

\begin{remark}
\label{rem:checkindependence}
Often we will need to define homomorphisms 
$$h:MDC_\bullet^b((X,x_0,d_X);A)\to G,$$ 
where $G$ is an abelian group. The usual procedure is to define first a homomorphism $\bar{h}:MDC^{\mathrm{pre},\infty}_\bullet((X,x_0,d_X);A)\to G$, and check that it descends to a well defined $h$. It is convenient to record that $\bar{h}$ descends if and only if the following two conditions hold:
\begin{itemize}
 \item For any $\sigma, z\in MDC^{\mathrm{pre},\infty}_\bullet((X,x_0,d_X);A)$, where $\sigma$ is a simplex and $z$ is a chain such that we have the immediate equivalence $\sigma\to_\infty z$, we have the equality $h(\sigma)=h(z)$.  
 \item For any two simplices $\sigma, \sigma'\in MDC^{\mathrm{pre},\infty}_\bullet((X,x_0,d_X);A)$ such that we have the equivalence $\sigma\sim_b\sigma'$, we have the equality $h(\sigma)=h(\sigma')$. 
\end{itemize}
\end{remark}

\begin{lema}
\label{lem:compatpartialsumb}
The boundary operator $\partial$ in $MDC_\bullet^{\mathrm{pre},\infty}(X;A)$ descends to a well defined boundary operator in $MDC_\bullet^b(X;A)$.
\end{lema}
\proof
We have to check the conditions of Remark~\ref{rem:checkindependence}. The first condition is exactly Lemma~\ref{lem:compatpartialsum}. The second condition is similar to the proof of Lemma~\ref{lem:repar}.
\endproof

\begin{definition}[$b$-Moderately discontinuous homology]
We define the {\em $b$-moderately discontinuous chain complex of} $(X,x_0,d_X)$ with coefficients in $A$ (the $b$-MD complex for short) to be the complex  $MDC^{b}_\bullet((X,x_0,d_X);A)$ with the boundary operator defined in the previous lemma. Its homology is called the {\em $b$-moderately discontinuous homology with coefficients in $A$} ({\em $b$-MD homology}, for short) and is denoted by $MDH^b_\bullet((X,x_0,d_X);A)$.

\end{definition}

To have a quick grasp of the meaning of $b$-MD Homology see Figure~\ref{fig:toro2puntos}.

\begin{figure}\includegraphics[scale=0.30]{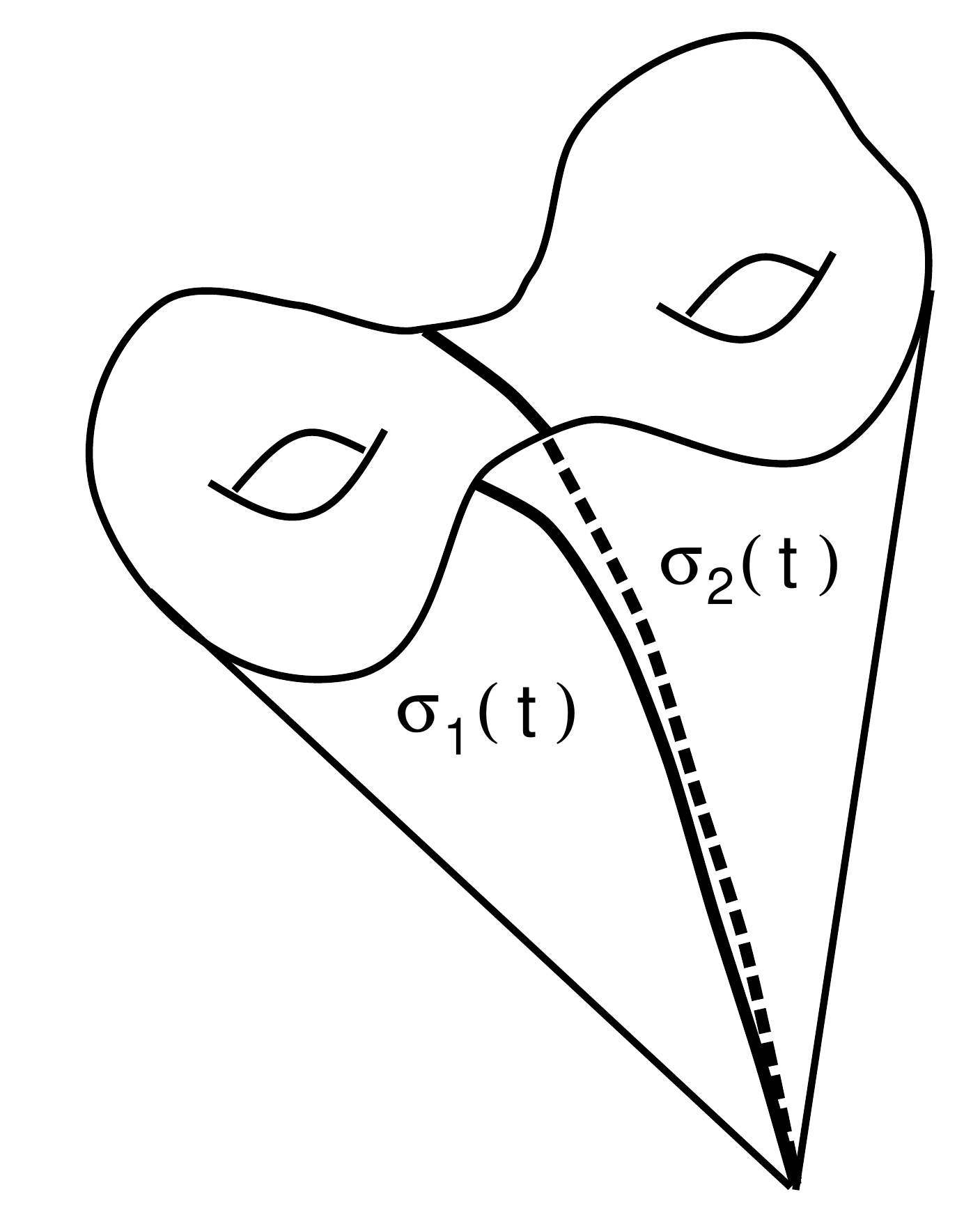}
\includegraphics[scale=0.60]{es}
\includegraphics[scale=0.30]{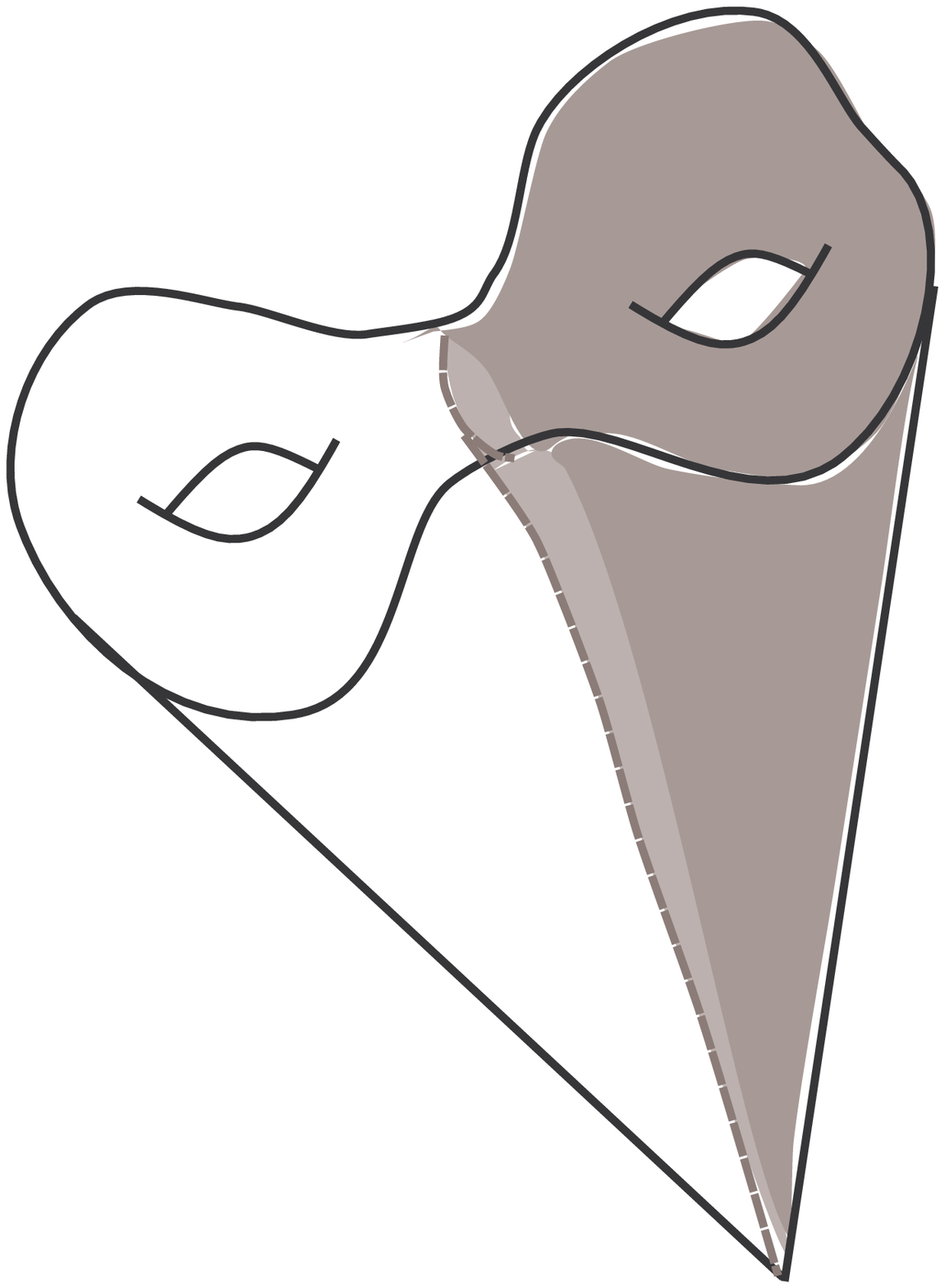}
\caption{Consider the figure with the outer metric where the thiner part collapses with tangency of order 4 while the thicker part collapses linearly. On the left, the two 0-simplexes $\sigma_1$ and $\sigma_2$ are 2-equivalent. On the right, we have shadowed what can be the image of a non-zero 2-cycle in the 3-MD Homology. Note that in singular homology it doesn't even represent a cycle because it has non-trivial boundary.}
\label{fig:toro2puntos}
\end{figure}

The following consequence of Proposition~\ref{prop:downupb} will be used repeatedly:

\begin{lema}
\label{lema:banulacion}
Given an element $z\in MDC^{\mathrm{pre},\infty}_\bullet((X,x_0,d_X);A)$, the class $[z]$ in $MDC^b_\bullet((X,x_0,d_X);A)$ vanishes if and only if there exists a sequence of immediate equivalences $z=z_0\to_{\infty}z_1\to_{\infty}...\to_\infty z_r$ such that $z_r \sim_b 0$. Notice that, by Remark~\ref{rem:splitindexset}, the chain $z_r = \sum_{i\in I}a_i\sigma_i$ is as follows: consider the subdivision of the index set $I=\coprod_{j\in J}I_j$ so that $i,i'$ belong to the same $I_j$ if and only if $\sigma_i\sim_b \sigma_{i'}$. Then for any $j\in J$ we have the equality
\begin{equation}
\label{eq:banulacion}
\sum_{i\in I_j}a_i=0.
\end{equation}
\end{lema}
\proof
If there exists a sequence $z=z_0\to_{\infty}z_1\to_{\infty}...\to_\infty z_r \sim_b 0$, then the class $[z]\in MDC^b_\bullet((X,x_0,d_X);A)$ vanishes obviously. Let us prove the converse.
Suppose that $z\in MDC^{\mathrm{pre},\infty}_\bullet((X,x_0,d_X);A)$ is such that its class $[z]$ vanishes in $MDC^b_\bullet((X,x_0,d_X);A)$. By Proposition~\ref{prop:downupb} there exists a sequence of $b$-immediate equivalences
\begin{equation}\label{eq:seqequivalences}
z=z_0\to_{b} z_1 \to_{b}...\to_{b} z_r=0,
\end{equation}
We proceed by induction over $r$. If $r=1$, there is nothing to show.

For the induction step we prove the following: if $z_1'$, $z_1$ and $z_2$ are chains in the complex $MDC^{\mathrm{pre},\infty}_\bullet((X,x_0,d_X);A)$ such that $z_1' \sim_b z_1$ and $z_1 \to_\infty z_2$, 
then there is a $z_2' \in MDC^{\mathrm{pre},\infty}_\bullet((X,x_0,d_X);A)$ such that $z_1' \to_\infty z_2'$ and $z_2' \sim_b z_2$. To show that, write $z_1 = \sum_{i\in I}a_i\sigma_i$ and let $\{\rho_{i,l}\}_{l \in L_i}$ be homological subdivisions for which 
$$
z_2 = \sum_{i\in I}\sum_{l \in L_i} a_i sgn(\rho_{i,l})\sigma_i \comp \rho_{i,l}.
$$
Let $z_1' = \sum_{j\in J}a'_j\sigma'_j$. Let $I = \coprod_{k \in K} V_k$ and $J = \coprod_{k \in K} V_k'$ be the splitting in accordance with Remark~\ref{rem:splitindexset} applied to $z_1$ and $z_1'$. For any $j \in J$, choose a fixed $k_j \in  V_k'$. Then it is
$$
z_1' = \sum_{k \in K} \sum_{j \in V_k'}a_j' \comp \sigma_{j_k}' + \tilde{z}_1 = \sum_{k \in K} \sum_{i \in V_k}  a_i \comp \sigma_{j_k}' + \tilde{z}_1
$$
where $\tilde{z}_1 \sim_b 0$. Set
$$
z_2' := \sum_{k \in K} \sum_{i \in V_k} \sum_{l \in L_i} sgn(\rho_{i,l})a_i \comp \sigma_{j_k}' \comp \rho_{i,l} + \tilde{z}_1.
$$
By Lemma~\ref{lem:repar}, it is $z_2' \sim_b z_2$.

Now suppose $r>1$. By what we have just shown and the induction hypothesis, sequence (\ref{eq:seqequivalences}) can be transformed into
$$
z=z_0\to_\infty z_0' \to_\infty z_1' \to_\infty...\to_\infty z_r' \sim_b z_r \sim_b 0.
$$
\endproof

\subsection{Relative \texorpdfstring{$b$}{b}-Moderately Discontinuous Homology}\label{subsection:relhomology}

In our setting relative homology exists in two different levels of generality. Let us start with the less general one, which is analogue to the classical Singular Homology Theory (See Subsection \ref{sec:relX} for the other one, which is essential for the formulation of the relative Mayer-Vietoris Theorem in our theory).

Consider $b\in (0,+\infty]$. Given  a subanalytic subgerm $(Y,x_0,d_{X_{|{Y}}} )\hookrightarrow (X,x_0,d_X)$, we denote by $K_\bullet$ the minimal subcomplex of  $MDC^b_\bullet(X,x_0,d_{X})$ which contains the classes $[\sigma]$, where $\sigma$ is a l.v.a. simplex in $Y$. An easy application of Lemma~\ref{lema:banulacion} shows that the obvious epimorphism of complexes $MDC^b_\bullet(Y,x_0,d_{X_{|Y}})\to K_\bullet$ is an isomorphism. Therefore we have an inclusion of complexes 
\begin{equation}
\label{eq:inclusionrelative}
MDC^b_\bullet(Y,x_0,d_{X_{|Y}})\hookrightarrow MDC^b_\bullet(X,x_0,d_{X}).
\end{equation}

\begin{definition}\label{def:rel} Consider $b\in (0,+\infty]$. Given  a subanalytic subgerm $(Y,x_0,d_{X_{|{Y}}} )\hookrightarrow (X,x_0,d_X)$, we define the {\em complex of relative $b$-moderately discontinuous chains with coefficients in $A$}, denoted by  
$MDC^b_\bullet((X,Y,x_0,d_X);A)$ as the following quotient: 
$$MDC^b_\bullet((X,x_0,d_X);A)\biggm/ MDC^b_\bullet((Y,x_0,d_X|_Y);A),$$
which makes sense by inclusion~(\ref{eq:inclusionrelative}).


The {\em $b$-moderately discontinuous homology $MDH^b_{*}((X,Y,x_0,d_X);A)$ with coefficients in $A$} is the homology of the complex $MDC^b_\bullet((X,Y,x_0,d_X);A)$. 

We abbreviate calling these complexes and graded abelian groups the {\em $b$-MD complex} and {\em $b$-MD homology} of  the pair $(X,Y,x_0,d_X)$. When it is clear from the context we will denote it  simply by $MDC^\infty_\bullet(X,Y;A)$ and similarly for homology.
\end{definition}

\begin{notation}
Denote by $Kom(Ab)^-$ the category of complexes of abelian groups bounded from the right.  
Denote by $\calD(Ab)^-$ the bounded above derived category of abelian groups. It is the localization at quasi-isomorphisms of the category whose objects are complexes bounded from the right and whose morphisms are homotopy classes of morphisms of complexes. 
Since we will deal with homology we will index the complexes as $...\to C_k\to C_{k-1}\to...$.  
There is a functor denoted by $H_*$ from $Kom(Ab)^-$ to the category $GrAb$ of graded abelian groups, which consists in taking the homology of a complex.
\end{notation}

At this point we check functoriality for the first time:

\begin{prop}
\label{prop:funct1} For every $b\in (0,+\infty]$, the assignments $$(X,Y,x_0,d_X)\mapsto MDC^b_\bullet((X,Y,x_0,d_X);A)\ \ and$$ 
$$(X,Y,x_0,d_X)\mapsto MDH^b_\bullet((X,Y,x_0,d_X);A)$$ are functors from the category of pairs of  metric subanalytic germs to $Kom(Ab)^-$ resp. $GrAb$.
\end{prop}
\proof A Lipschitz l.v.a. subanalytic map $f:(X,x_0,d_X)\to (X',x'_0,d_{X'})$ induces morphisms $MDC_\bullet^{pre,\infty}(X,x_0.d_x)\to MDC_\bullet^b(X',x'_0,d_{X'})$ for every $b\in (0,+\infty]$, by taking $\sigma\mapsto f\circ\sigma$ for every l.v.a. simplex $\sigma$ and extending by linearity. One can check that it descends to a well defined  morphism from $MDC^{b}_\bullet((X,x_0,d_X);A)$ to $MDC^{b}_\bullet((X',x'_0,d_{X'});A)$ because it satisfies the two conditions of Remark~\ref{rem:checkindependence}, which are straightforward. 

If $f$ takes a subanalytic subgerm $Y$ into a subanalytic subgerm $Y'$, then the homomorphism defined above transforms the subcomplex $MDC^{b}_\bullet((Y,x_0,d_X|_Y);A)$ into $MDC^{b}_\bullet((Y',x'_0,d_{X'}|_{Y'});A)$, and hence descends to the relative homology groups. 
\endproof

\begin{notation}\label{not:f*C}
Given a Lipschitz l.v.a. subanalytic map $$f:(X,Y,x_0,d_X)\to (X',Y',x'_0,d_{X'})$$ we denote by $f_*$ the induced map at the level of $b$-MD chains for every $b\in (0,+\infty]$.
\end{notation}

\subsection{The final definition of Moderately Discontinuous Homology}\label{subsection:finaldef}

In this section we introduce the complete definition of Moderately Discontinuous Chain Complexes/Homology as a functor from the category of pairs of metric subanalytic germs, to a category of diagrams of complexes/groups. 

The starting observation is the following: for $b_1\geq b_2$ with $b_i\in (0,+\infty]$ there are natural epimorphisms (see Section~\ref{sec:BB} for the associated long exact sequence):
\begin{equation}\label{eq:Cb_1-b_2}
h^{b_1, b_2}: MDC^{b_1}_\bullet((X,Y,x_0,d_X);A) \rightarrow MDC^{b_2}_\bullet((X,Y,x_0,d_X);A)
\end{equation} 
which induces a map in homology: 
\begin{equation}\label{eq:Hb_1-b_2}
h^{b_1, b_2}_*: MDH^{b_1}_\bullet((X,Y,x_0,d_X);A) \rightarrow MDH^{b_2}_\bullet((X,Y,x_0,d_X);A).
\end{equation} 

\begin{notation} 
We define the category $\BB$, where the set of objects is $(0,\infty]$ and there is a unique morphism from $b$ to $b'$ if and only if $b\geq b'$.
\end{notation}

\begin{definition}[Categories of $\BB$-complexes and $\BB$-graded abelian groups]\label{def:B-categories}
The \em{category $\BB-Kom(Ab)^-$ of $\BB$-complexes} is the category whose objects are functors from $\BB$ to $Kom(Ab)^-$ and the morphisms are natural transformations of functors. The {\em category}  $\BB-\calD(Ab)^-$ is the category whose objects are functors from $\BB$ to $\calD(Ab)^-$ and the morphisms are natural transformations of functors. The {\em category $\BB-GrAb$ of $\BB$-graded abelian groups} is the category whose objects are functors from $\BB$ to the category $GrAb$ and the morphisms are natural transformations of functors. Concatenation of objects in $\BB-Kom(Ab)^-$ with the homology functor $H_*$ yields a functor $\BB-H_*:\BB-Kom(Ab)^-\to \BB-GrAb$ which factorizes through $\BB-\calD:\BB-\calD(Ab)^-\to \BB-GrAb$.
\end{definition}

\begin{prop}
\label{prop:functoriality} 
The assignments $(X,Y,x_0,d_X)\mapsto MDC^{\star}_\bullet((X,Y,x_0,d_X);A)$ and $(X,Y,x_0,d_X)\mapsto MDH^{\star}_*((X,Y,x_0,d_X);A)$ are functors from the category of pairs of metric subanalytic germs to $\BB-Kom(Ab)^-$ and $\BB-GrAb$ respectively.
\end{prop}
\proof
One only needs to check that the functoriality for Lipschitz l.v.a subanalytic maps for each $b\in\BB$ commutes with the epimorphisms (\ref{eq:Cb_1-b_2}), (\ref{eq:Hb_1-b_2}) which is clear.
\endproof

It is interesting to record that in the case of complex (resp. real) analytic germs our homology theory gives complex (resp. real) analytic invariants.

\begin{cor} Given a complex (resp. real) analytic germ $(X,x_0)$, the $\BB$-moderately discontinuous homology $MDH^b_\bullet(X,x_0,d_{out})$ and $MDH^b_\bullet(X,x_0,d_{inn})$  for the outer and inner metrics are complex (resp. real) analytic invariants.
\end{cor}
\proof

A real or complex analytic diffeomorphism is well known to be bi-Lipschitz both for the inner and outer metric  and it is clearly l.v.a.
\endproof
\subsection{Bi-Lipschitz invariance of \texorpdfstring{$b$}{b}-MD homology with respect to the inner distance}
We check that a subanalytic homeomorphism between two germs $(X,x_0)$
and $(Y,y_0)$ that is bi-Lipschitz for the inner metric is l.v.a..  Then we conclude  that the MD Homology for $d_{inn}$ is a bi-Lipschitz invariant. 

\begin{prop}\label{top_inner} Let $(X,x_0)$ and $(Y,y_0)$ be two germs of subanalytic sets. Let $d_{X,inn}$ (resp. $d_{Y,inn}$) be the inner distance of $X$ (resp. $Y$). Then we have the following:
\begin{itemize}
\item [(a)] $d_{X,inn}$ (resp. $d_{Y,inn}$) induces the same topology on $X$ (resp. $Y$) as the topology induced by the standard topology on $\RR^m$;
\item [(b)] If there exists an inner bi-Lipschitz homeomorphism $h\colon (X,x_0)\to (Y,y_0)$ then there exists $K>0$ satisfying the inequalities
$$
\frac{1}{K}\|x-x_0\| \leq \|h(x)-y_0\| \leq K \|x-x_0\|.
$$
\end{itemize}
\end{prop}
In order to prove Proposition \ref{top_inner}, we recall the following result.
\begin{prop}[Proposition 3 in \cite{KurdykaO:1997}]\label{lne_decomp}
Let $X\subset\RR^m$ be a subanalytic set and $\varepsilon>0$. Then there exists a finite decomposition 
$X = \bigcup_{j =1}^k \Gamma_j$  such that:
\begin{enumerate}
\item each $ \Gamma_j$ is a subanalytic connected analytic submanifold of  
$\RR^m$,
\item each $ \overline \Gamma_j$ satisfies $d_{\overline \Gamma_\nu, inn}(p,q)\leq(1+\varepsilon)\|p-q\|$ for any $p,q\in\overline \Gamma_j$. 
\end{enumerate}
\end{prop}
\proof[Proof of Proposition~\ref{top_inner}]
Let us consider $X = \bigcup_{j =1}^k \Gamma_j$ as in Proposition \ref{lne_decomp} and $\varepsilon =1$. Thus, if $x\in X$, there exists a $j$ such that $x\in \Gamma_j$ and, moreover, we get
\begin{equation}\label{eq:inner_one}
\frac{1}{2}\|x-x_0\| \leq d_{X,inn} (x, x_0)\leq d_{\Gamma_j, inn} (x, x_0) \leq 2 \|x-x_0\|.
\end{equation}
Since $\|x-y\| \leq d_{X,inn} (x, y)$ for any $x,y\in X$, to prove item (a) it is enough to prove that for any $x\in X$ and any ball $B_{inn,\eta}(x)$ with respect to the inner distance, we can find a ball $B_{\delta}(x)$ with respect to the outer distance such that $B_{\delta}(x)\subset B_{inn,\eta}(x)$. But to do this, we just apply Proposition \ref{lne_decomp} to $(X,x)$ and $\varepsilon =1$, and we get that $B_{\eta/2}(x)\subset B_{inn,\eta}(x)$.

Obviously we have the same result for $Y$ and, in particular, we have
\begin{equation}\label{eq:inner_two}
\frac{1}{2}\|y-y_0\| \leq d_{Y,inn} (y, y_0) \leq 2 \|y-y_0\|.
\end{equation}
In order to get item (b), we just need to apply the Lipschitz properties of $h$ and Eq. (\ref{eq:inner_one}) in Eq. (\ref{eq:inner_two}).
\endproof

Thus, by considering Remark~\ref{rem:Cat_inner}, the following is immediate.
\begin{cor}\label{func_inner} Let $(X,x_0)$ and $(Y,y_0)$ be two subanalytic germs. If there exists a subanalytic bi-Lipschitz homeomorphism $h\colon (X,x_0,d_{X,inn})\to (Y,y_0,d_{Y,inn})$, then $(X,x_0,d_{X,inn})$ and $(Y,y_0,d_{Y,inn})$ have the same MD homology. In particular, $h$ induces isomorphisms $$h_n\colon MDH^b_n(X,x_0,d_{X,inn})\to MDH^b_n(Y,y_0,d_{Y,inn})$$ for all $b\in (0,+\infty]$ and $n\in\mathbb{N}$.
\end{cor}

\section{Basic properties of MD-Homology}

In this section we prove properties of MD-Homology in analogy with usual homology theories (relative exact sequence, its value at a ``point'' and sufficiency of chains which are small with respect to a cover). The analogues of homotopy invariance, Mayer-Vietoris and Excision are more subtle and are treated later in the paper. We introduce also a long exact sequence measuring the relation of the $b$-MD homologies for different $b$.

\subsection{The relative MD-Homology sequence}

The relative homology sequence comes quite easily from the definition. 

\begin{prop}\label{prop:relativehomologysequence}
Let $(X,x_0,d_X)$ be a metric subanalytic germ. Let $Z\subset Y\subset X$ be subanalytic subgerms. For any $b\in\BB$ there is a long exact sequence 
\begin{align}\label{long exact sequence: subspace relative homology}
\begin{split}
... &\rightarrow MDH_n^{b}(Y,Z;A) \rightarrow MDH_n^{b}(X,Z;A) \rightarrow MDH_n^{b}(X,Y;A) \\
& \rightarrow MDH_{n-1}^{b}(Y,Z;A) \rightarrow MDH_{n-1}^{b}(X,Z;A) \rightarrow MDH_{n-1}^{b}(X,Y;A) \rightarrow ...
\end{split}
\end{align}
This exact sequence is functorial in $Z\subset Y\subset X$ and in $b\in\BB$.
\end{prop}
\proof
The proof is obvious from the definitions.
\endproof

Similarly we obtain the spectral sequence of a filtration of pairs of metric subanalytic germs:

\begin{prop}
 Let $Z_0\subset Z_1\subset ...\subset Z_r=X$ be a filtration by closed subanalytic subgerms of $(X,x_0)$. Let $Y$ be another closed subanalytic subgerm of $(X,x_0,d_X)$. For each $b$, the induced filtration in $MDC_{\bullet}^b(X,Y;A)$ yields a spectral sequence abutting to $MDH^b_{p+q}(X,Y;A)$ with $E^1$ page equal to $$E[b]^1_{p,q}=MDH^b_{p+q}(Z_p\cup Y,Z_{p-1}\cup Y;A).$$
 The spectral sequence is functorial in $b\in\BB$.  
 \end{prop}

\subsection{The Moderately Discontinuous Homology of a ``point''}\label{subsection:point}

Like in any homology theory the point plays a special role. In the next definition we clarify the notion of point in our category.

\begin{definition}
\label{def:point}
A point in the category of metric subanalytic germs is a metric subanalytic germ isomorphic to $((0,\epsilon), 0, d)$, where $d$ is the Euclidean metric.
 \end{definition}

\begin{prop}
\label{prop:point}
For any $b\in [1,\infty)$ the complex $MDC^b_{\bullet}((0,\epsilon);A)$ is quasi-isomorphic to the complex $A[0]$, i.e. $MDH_0^b((0,\epsilon); A)=A$ and $MDH_n^b((0,\epsilon); A)=0$ for all $n>0$. 
\end{prop}
\proof
We show that the augmented chain complex of $C^b_{\bullet}((0,\epsilon);A)$ by $A$ in degree $-1$ has trivial homology by constructing a chain homotopy $H$ from the identity to the $0$-map: denote by $\sigma_0$ the identity map on $(0,\epsilon)$. On degree $-1$, we define $H(a) = a \sigma_0$. For $n \in \mathbb{N}_0$, given $\sigma:\hat{\Delta}_n\to (0,\epsilon)$ in $MDC^{\mathrm{pre},\infty}_{n}((0,\epsilon);A)$      define $H(\sigma)\in MDC^{\mathrm{pre},\infty}_{n+1}((0,\epsilon);A)$ to be the suspension of $\sigma$ by $\sigma_0$ given by the formula 
\[
H(\sigma)(ts_0,...,ts_{n+1},t):=(-1)^{n+1}(S\sigma(\frac{ts_0}{S},...,\frac{ts_n}{S},t)+s_{n+1}(\sigma_0(t)))
\]
where $(s_0,...,s_{n+1})$ are barycentric coordinates in $\Delta_{n+1}$ and $S:= s_0+...+s_n$. If $S = 0$, define $H(\sigma)(ts_0,...,ts_{n+1},t):=(-1)^{n+1}\sigma_0(t)$. Observe that for an $n$-simplex $\sigma$ with $n \geq 1$ it is $H(\sigma \circ j_n^k) = - H(\sigma) \circ j_{n+1}^k$ for $k\leq n$ and $H(\sigma) \circ j_{n+1}^{n+1} = (-1)^{n+1}\sigma$. This defines the chain homotopy in the augmentation of $MDC^{\mathrm{pre},\infty}_{\bullet}((0,\epsilon);A)$. 

In order to finish the proof, we use Remark~\ref{rem:checkindependence} in order to show that the chain homotopy descends to a chain homotopy defined in the augmentation of $MDC^b_{\bullet}((0,\epsilon);A)$. 

Let $\{\rho_i\}_{i\in I}$ be a homological subdivision of $\hat{\Delta}_n$ associated with a triangulation $\alpha:|K|\to\hat{\Delta}_n$. Notice that $\Delta_{n+1}$ is the cone over $\Delta_n$, with vertex 
$p=(0,...,0,1)$; this allows us to see $\hat{\Delta}_{n+1}$ as the cone over $\hat{\Delta}_n$. Let $C(K)$ be the cone over the simplicial complex $K$ and let $\beta:|C(K)|\to \hat{\Delta}_{n+1}$ be the triangulation obtained by taking the cone over the triangulation $\alpha$. Define $\rho'_i:\hat{\Delta}_{n+1}\to \hat{\Delta}_{n+1}$ to be the cone over the mapping $\rho_i$. Then the collection $\{\rho'_i\}_{i\in I}$ is a homological subdivision of $\hat{\Delta}_{n+1}$ associated with the triangulation $\beta$, such that for any $i\in I$ we have the equality
$$H(\sigma\comp\rho_i)=H(\sigma)\comp\rho'_i.$$
This shows that the homotopy descends to  $MDC^{\infty}_{\bullet}((0,\epsilon);A)$.

In order to prove that it descends to $MDC^{b}_{\bullet}((0,\epsilon);A)$ it only remains to show that it preserves the $b$-equivalence relation. Let $\sigma_1$ and $\sigma_2$ be $b$-equivalent l.v.a. $n$-simplices. Then $H(\sigma_1)$ and $H(\sigma_2)$ are $b$-equivalent, since we have the inequality
\begin{align*}
&|(S\sigma_1(\frac{ts_0}{S},...,\frac{ts_n}{S},t) - (S\sigma_2(\frac{ts_0}{S},...,\frac{ts_n}{S},t) |  \\
\leq &S \max\{ | \sigma_1(u_0t,\dots u_n t, t) - \sigma_2(u_0t,\dots u_n t, t)| : (u_0,\dots, u_n) \in \Delta_n \}
\end{align*}
for every $(s_0,\dots, s_n) \in \Delta_n$ and $S \leq 1$.
\endproof

\subsection{Relative homology with respect to \texorpdfstring{$b\in [0,+\infty )$}{b in [0,+infty)}}
\label{sec:BB}
In our theory we have also a notion of relative homology with respect to $b\in[0,+\infty )$. 

\begin{definition}
\label{def: b-relative homology}
Let $(X,Y)$ be a pair of metric subanalytic germs, $A$ an abelian group and $b_1\geq b_2\in Obj(\BB)$. 
We define the chain complex $MDC_{\bullet}^{b_1, b_2}(X,Y;A)$ to be the kernel of the epimorphism 
\[
h^{b_1, b_2}:MDC_\bullet^{b_1}(X,Y;A) \rightarrow MDC_\bullet^{b_2}(X,Y;A)
\]
The \emph{$n$-th $(b_1, b_2)$-moderately discontinuous homology} is defined to be the homology of $MDC_{\bullet}^{b_1, b_2}(X,Y;A)$.
\end{definition}

\begin{prop}
The following long exact sequence is an immediate consequence of the last definition:
\begin{align}\label{long exact sequence: r-relative homology}
\begin{split}
... &\rightarrow MDH_n^{b_1, b_2}(X,Y;A)\rightarrow MDH_n^{b_1}(X,Y;A) \rightarrow  MDH_n^{b_2}(X,Y;A)\rightarrow \\
& \rightarrow MDH_{n-1}^{b_1, b_2}(X,Y;A) \rightarrow MDH_{n-1}^{b_1}(X,Y;A) \rightarrow MDH_{n-1}^{b_2}(X,Y;A) \rightarrow ...
\end{split}
\end{align}
Its association to $(X,Y)$ is functorial.
\end{prop}

\subsection{MD-chains which are small with respect to a subanalytic cover and to a dense subgerm}
\label{sec:denseqis}

We will need to use chains which are small with respect to covers as in the classical development of singular homology (see for example ~\cite{GH}, Ch 15) as a technical tool. Moreover, our theory becomes easier if we can assume that the chains have support in a dense subanalytic subset of $X$. A consequence of this, for example, is a proof, under a mild condition, that the $b$-MD Homology of a subanalytic subgerm 
$(X,x_0,d)$ of $\RR^n$ coincides with with the $b$-MD Homology of its closure, provided that the closure has a subanalytic metric restricting to $d$.  

\begin{definition}
\label{def:smallchain}
Let $(X,x_0)$ be a subanalytic  germ. Let $\calD=\{D_i\}$ be a collection of subanalytic subsets of $(X,x_0)$. 
A chain $\sum_{j\in J} a_j\sigma_j\in MDC^{\mathrm{pre},\infty}_{\bullet}(X,x_0;A)$ is called \emph{small with respect to $\calD$}, if for any $j$ the image of $\sigma_j$ is contained in some  $D_i$. We denote by $MDC^{\mathrm{pre},\infty,\calD}_{\bullet}(X,x_0;A)$ the subcomplex of $MDC^{\mathrm{pre},\infty}_{\bullet}(X,x_0;A)$ formed by the chains which are small with respect to the $\calD$.

Let $(X,x_0,d_X)$ be a metric subanalytic germ. 
The complexes $MDC^{\infty,\calD}_{\bullet}(X,x_0;A)$ and $MDC^{b,\calD}_{\bullet}(X,x_0,d;A)$ are defined by restricting the equivalence relations $\sim_{S,\infty}$ and $\sim_{S,b}$ to $MDC^{\mathrm{pre},\infty,\calD}_{\bullet}(X,x_0;A)$.

Given a subanalytic subgerm $Y\subset X$, we define the complexes $MDC^{\infty,\calD}_{\bullet}(X,Y;A)$ and $MDC^{b,\calD}_{\bullet}(X,Y,d;A)$  as the quotients of $MDC^{\infty,\calD}_{\bullet}(X,x_0;A)$ and $MDC^{b,\calD}_{\bullet}(X,x_0,d;A)$ by $MDC^{\infty,\calD}_{\bullet}(Y,x_0;A)$ and $MDC^{b,\calD}_{\bullet}(Y,x_0,d|_Y;A)$ respectively (as in Definition~\ref{def:rel}, we may assume that the complexes we quotient by are subcomplexes).

\end{definition}
\begin{definition}
\label{def:semialgcover}
Let $(X,x_0)$ be a subanalytic 
germ. A {\em finite closed subanalytic cover of $X$} is a finite collection of closed subanalytic subsets $\calC:=\{C_i\}_{i\in I}$ of $X$ such that $X=\bigcup_{i\in I}C_i$. 

Given a collection of sets $\calD=\{D_i\} $ and a subset $U$ we denote by $\calD\cap U$ the collection $\{D_i\cap U\}$. 
\end{definition}

The following proposition is the main result of this section:

\begin{prop}\label{prop:smallqis} 
Let $(X,x_0,d)$ be a metric subanalytic germ, $Y$ a subgerm, and $\calC$ be a finite closed subanalytic cover of $X$.
The natural morphism of complexes $$g:MDC^{b,\calC}_{\bullet}(X,Y;A)\to MDC^{b}_{\bullet}(X,Y;A)$$ is an isomorphism.

Assume furthermore that the metric $d$ extends to a subanalytic metric $\overline{d}$ in the closure $\overline{X}$ of $X$ in $\RR^n$. Let $U$ be a dense subanalytic subset of $X$ such that $U\cap Y$ is dense in $Y$.
 For any $b<\infty$ the natural morphism of complexes 
$$g:MDC^{b,\calC\cap U}_{\bullet}(X,Y;A)\to MDC^{b}_{\bullet}(X,Y;A)$$
is an isomorphism.
\end{prop}

Before proving the proposition let us extract two consequences.

\begin{cor}
\label{cor:closure}
Let $(X,x_0,d)$ be a metric subanalytic germ such that the metric $d$ extends to a subanalytic metric $\overline{d}$ in the closure $\overline{X}$ of $X$ in $\RR^n$. Then for any $b<\infty$ we have an isomorphism
$$MDH_\bullet^b(X,x_0,d;A)\cong MDH_\bullet^b(\overline{X},x_0,\overline{d};A).$$
\end{cor}

Proposition~\ref{prop:smallqis} allows us to improve Remark~\ref{rem:checkindependence} in the following manner:

\begin{remark}
\label{rem:checkindependence small chains}
In order to define a homomorphism
$$MDC^b_\bullet((X,x_0,d);A)\to G,$$ 
where $G$ is an abelian group, we will often proceed as follows: we take a finite closed subanalytic cover $\calC=\{C_i\}$ and a dense subanalytic subset $U$ of $X$ (if $b=\infty$ we need to impose $U=X$), define a homomorphism $\bar{h}:MDC^{\mathrm{pre},\infty,\calC\cap U}_\bullet((C_i,x_0,d);A)\to G$, check that the two conditions of Remark~\ref{rem:checkindependence} hold and compose with $g^{-1}$ on the right, where $g$ is the isomorphism of Proposition~\ref{prop:smallqis}. 
\end{remark}

Before proving Proposition~\ref{prop:smallqis} in the general case where $U\neq X$ we need some preparation, (see Figure \ref{fig:SQ_Estratif} for an example).  

\begin{lema}
\label{lem:retracciondiscontinua}
Suppose that $S\supset Q$ are compact subanalytic subsets in $\RR^n$. Let $d$ be a subanalytic metric in $S$. There exists a partition of $S$ into finitely many disjoint subanalytic subsets $\{S_i\}_{i\in I}$, such that there exists continuous subanalytic maps $f_i:S_i\to Q$ with the property that for any $z\in S_i$ we have the equality
\begin{equation}
\label{eq:closest}
 d(z,f_i(z))=d(z,Q).
\end{equation}
In particular $f_i(z)=z$ for any $z\in Q$. 

Moreover if $S\setminus Q$ is dense in $S$ then there exists a subanalytic stratification of $Q$ by smooth manifolds such that the union of maximal strata of the stratification by the closure relation is included in $\cup f_i(S_i\setminus Q)$. In particular $\cup f_i(S_i\setminus Q)$ is dense in $S$.
\end{lema}

\begin{figure}
\includegraphics[scale=0.2]{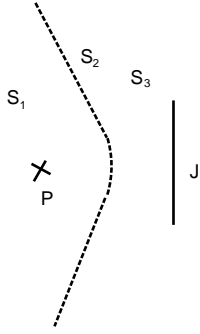}
\caption{Taking $Q=\{P\}\cup J$ the union of the point $P$ and the segment $J$, we stratify $\RR^2$ following Lemma \ref{lem:retracciondiscontinua} in three pieces where $S_2$ is the dotted curved line and $S_1$ and $S_3$ are the components of its complement.  }
\label{fig:SQ_Estratif}
\end{figure}
\proof
The function $\rho:S\to\RR$ defined by $\rho(z):=d(z,Q)$ is subanalytic. Therefore the subset
$$A:=\{(z,x)\in S\times Q: d(z,x)=\rho(z)\}$$ 
is also subanalytic. Denote by $\pi:A\to S$ the restriction of the projection to $S$. By Hardt's Trivialization Theorem there exists a partition of $S$ into finitely many disjoint subanalytic subsets $\{S_i\}_{i\in I}$ such that $\pi|_{\pi^{-1}(S_i)}:\pi^{-1}(S_i)\to S_i$ is trivial for any $i$. By compactness, $\pi^{-1}(S_i)$ is not empty for any $i\in I$. Choose a continuous subanalytic section $\tau_i:S_i\to \pi^{-1}(S_i)$, and define $f_i$ to be the composition of $\tau_i$ with the projection to $Q$. 

Let $\calV'$ be a stratification of $Q$ by smooth manifolds and let $q\in Q$ be a point in a maximal stratum. Then the distance from $q$ to the union of all the other strata is positive. Choose a ball $B$ in $\RR^n$ centered in $q$ of very small radius. Then $B\cap Q$ is a manifold and any point in $b\in B$ has a unique closest point in $Q$ which belongs to $B\cap Q$. Since $S\setminus Q$ is dense in $S$ there exists a sequence $\{s_n\}$ of points in $S\cap B\setminus Q$ converging to $q$. Let $q_n$ be the unique closest point to $s_n$ in $Q$. Then the sequence $\{q_n\}$ also converges to $q$. Since any $q_n$ is at the image of $f_i$ for a certain $i$ we have that the union $\cup f_i(S_i\setminus Q)$ is dense in $S$.
Any stratification $\calV$ of $Q$ by smooth manifolds stratification refining $\calV'$ and the union $\cup f_i(S_i\setminus Q)$ satisfies the property we need. 
\endproof

\begin{lema}
\label{lem:defopenstratum}
Let $(S,O,d)$ be a closed metric subanalytic germ (we stress that the metric $d$ is required to be subanalytic by definition) and $(Q,O)$ be a closed subanalytic subgerm of $(S,O)$ such that $S\setminus Q$ is dense in $S$. Fix any positive $b>0$. There exists a dense subanalytic subgerm $(U,O)\subset (Q,O)$ and a partition $\{U_j\}_{j\in J}$ of $(U,O)$ into finitely many disjoint subanalytic subgerms such that for any $j\in J$ there exists a continuous subanalytic map $g_j:U_j\to S\setminus Q$ satisfying
\begin{equation}
\label{eq:goodapprox}
d(g_j(x),x)<||x||^b
\end{equation}
for any $x\in U_j$.
\end{lema}
\proof
Along this proof we will denote by $Q$ and $S$ compact representatives of $Q$ and $S$ in a small enough ball. 

Let $\{S_i\}_{i\in I}$ and $\{f_i\}_{i\in I}$ be the partition and subanalytic maps predicted in Lemma~\ref{lem:retracciondiscontinua}. Let $V$ be the union of maximal strata of a stratification of $Q$ with the property predicted in Lemma~\ref{lem:retracciondiscontinua}. Choose a partition $V=\coprod_{k\in K}V_k$ into finitely many disjoint subanalytic subsets such that for any $k$ there exists an index $i(k)\in I$ satisfying that $V_k\subset S_{i(k)}$, and that, if we define $W_k:=f_i(k)^{-1}(V_k)\setminus Q$, then the closure of $W_k$ contains $V_k$.

For any $k\in K$ consider the subanalytic map
$$h_k:W_k\to V_k\times [0,\infty)$$
defined by $h_k(z):=(f_{i(k)}(z),d(z,X))$. By Hardt Triviality Theorem we can split $V_k\times [0,\infty)=\coprod_{j\in J_k}G_j$ into finitely many subanalytic subsets such that $h_k$ is trivial over each $G_j$. Let $s_j:G_j\to h_k^{-1}(G_j)$ be a continuous subanalytic section of $h_k|_{G_j}$.  

Let $\pi_1:V_k\times [0,\infty)$ be the first projection. An application of Hardt Triviality Theorem to each of the restrictions $\pi_1|_{G_j}$, and further subdivision allows us to assume that for each $j\in J_k$ the mapping $\pi_1|_{G_j}:G_j\to \pi(G_j)$ is trivial with connected fibre (so the fibre is either an interval or a point). Define $U_j$ to be the interior of $\overline{G_j}\cap (V_k\times \{0\})$
and $G'_j:=G_j\cap\pi_1^{-1}(U_j)$. The collection $\{U_j\}_{j\in J_k}$ is formed by mutually disjoint subanalytic subsets of $Q$ and its union is dense in $V_k$.

Define $J:=\cup_{k\in K}J_k$. Then the finite collection $\{U_j\}_{j\in J}$ is formed by mutually disjoint subanalytic subsets of $Q$ and its union $U:=\cup_{j\in J}U_j$ is dense in $Q$. The germ at the origin of $U$ and each of the $U_j$ provide the subgerms claimed in the statement of the lemma. It remains to construct the subanalytic maps $g_j$.

For any $j\in J_k$ the projection $\pi_1|_{G'_j}:G'_j\to U_j$ is trivial. Consequently for any $u\in U_j$ the fibre $\pi_1|_{G'_j}$ is the interval $(0,\alpha_j(u))$, where $\alpha_j:U_j\to (0,\infty)$ is a continuous subanalytic function. Define the continuous subanalytic function $\beta_j:U_j\to (0,\infty)$ by the formula $\beta_j(u):=\min\{||u||^b,\alpha_j(u)\}/2$. By construction, the collection of mappings $\{g_j\}_{j\in J}$, defined by $g_j(u):=s_j(u,\beta_j(u))$, have the property claimed in the lemma.
\endproof

\proof[Proof of Proposition~\ref{prop:smallqis}]
By the $5$-Lemma it is enough to prove the proposition for absolute homology, that is to prove the isomorphism $MDC^{b,\calC\cap U}_{\bullet}(X;A)\to MDC^{b}_{\bullet}(X;A)$ for any metric subanalytic germ $(X,x_0,d_X)$ (when $b=\infty$ we have $U=X$).

Injectivity is an immediate consequence of Lemma~\ref{lema:banulacion}.

For the surjectivity, we do first the case where $U=X$. Let $\sigma:\hat{\Delta}_n\to (X,x_0)$ be an $n$-simplex. We consider the collection $\calD$ of closed subanalytic subsets of $\hat{\Delta}_n$ given by the preimages by $\sigma$ of $Z$, of the subsets of $\calC$ and the collection of all the faces of $\hat{\Delta}_n$. Let $\alpha:|K|\to \hat{\Delta}_n$ be a triangulation of a representative of $\hat{\Delta}_n$ compatible with $\calD$ (see Remark~\ref{rem:triang}). Let $\{T_i\}_{i\in I}$ be the collection of maximal triangles of $\alpha$. By restricting the representative of $\hat{\Delta}_n$ we may assume that each maximal triangle $T_i$ contains the vertex. For each $i\in I$ choose a subanalytic homeomorphism $\rho_i:\hat{\Delta}_n\to T_i$ sending the vertex to the vertex, and which preserves the simplicial structure. By Remark~\ref{rem:lvareduction} we may assume $\rho_i$ to be l.v.a. Then the collection $\{\rho_i\}_{i\in I}$ is a homological subdivision of $\hat{\Delta}_n$, and we have the equivalence $\sigma\sim_{S,b}\sum_{i\in I}sgn(\rho_i)\sigma\comp\rho_i$. Since the chain on the right hand side is small with respect to $\calC$ surjectivity is proven in case $U=X$.

If $U\neq X$, first notice that we can assume $U$ to be open in $X$: otherwise we can replace it by its interior in $X$, and the surjectivity for the interior implies the surjectivity for $U$. So, from now on we assume $Z:=X\setminus U$ to be closed in $X$.  Now we proceed as follows. Since in this case $b<\infty$, everything boils down to prove the following assertion:

\textsc{Claim}: let $\sigma:\hat{\Delta}_n\to X$ be any l.v.a. simplex. Then there exists a homological subdivision $\{\rho_i\}_{i\in I}$ of $\hat{\Delta}_n$, and a $n$-simplex $\sigma'_i:\hat{\Delta}_n\to X$ for any $i$ whose image is contained in $U$ and such that $\sigma\comp\rho_i\sim_b\sigma'_i$. 

Indeed, if the assertion holds we can replace $\sigma$ by $\sum_{i\in I}\sigma_i'$ in $MDC^{b}_{\bullet}(X;A)$, and after apply the subdivision procedure explained in the paragraph above to each $\sigma'_i$ in order to obtain a sum of simplexes which is small both with respect to $\calC$ and $U$.

Now we prove the claim. First we reduce to the case in which $X$ is closed in $\RR^n$. Let $\overline{X}$ be the closure of $X$. By hypothesis it admits a subanalytic metric extending $d$. The interior $V$ of $U$ in $X$ is an open dense subanalytic subset of the closure $\overline{X}$. The claim for $\overline{X}$ and $V$ implies the claim for $X$ and $U$. 

Assume that $X$ is closed in $\RR^n$ and $U$ open in $X$. Define $Z:=X\setminus U$. The proof goes by induction on $\dim(X)$, where $\dim(X)$ is the maximal dimension of the irreducible components of $X$. 

Let $\sigma:\hat{\Delta}_n\to X$ be any l.v.a. simplex. By the subdivision procedure given in the third paragraph of this proof we can reduce to one of the following two cases:
\begin{enumerate}
 \item the image by $\sigma$ of the interior of $\hat{\Delta}_n$ is contained in $U$.  
 \item The image of $\sigma$ is fully contained in $Z$
\end{enumerate}

Let us start by case (1) (see Figure \ref{fig:lemaDense}). Let $p\in\Delta_n$ be the barycenter. Let $r:\Delta_n\setminus\{p\}\to\partial\Delta_n$ be the retraction along the straight lines connecting the boundary with the barycenter. Consider any subanalytic homeomorphism $g:\partial\Delta_n\times [0,1)\to\Delta_n\setminus\{p\}$ which sends
$\{x\}\times [0,1)$ into the retraction line meeting $x$. Fix $b'>b$ and $\eta>0$. Define $\Omega\subset\hat{\Delta}_n$ by 
$$\Omega:=\{(tx,t):x\in \Delta_n\setminus\{p\}, d_X(\sigma_i(tx,t),\sigma_i(tr(x),t)\leq \eta t^{b'}\}.$$
Obviously $\Omega$ is a subanalytic open subset containing $\partial\hat{\Delta}_n$. Now we construct a subanalytic mapping $h: \hat{\Delta}_n\to \hat{\Delta}_n$ that is a homeomorphism onto its image and such that $h(\partial\hat{\Delta}_n)\subseteq \Omega$.  

Choose a continuous subanalytic function $\theta:\partial\hat{\Delta}_n\to [0,1)$ such that for any $x\in \partial \Delta_n$ 
we have that $g((tx,t),s)$ belongs to $U$ for any $s\leq \theta(tx,t)$, and so that $\theta$ is strictly positive outside the vertex. The subanalytic mapping 
$f:\partial\Delta_n\times [0,1)\to \partial\Delta_n\times [0,1)$ defined by the formula $f((tx,t),s)=((tx,t),\mathrm{max}\{s,\theta(tx,t)\})$
induces a subanalytic mapping $h:\hat{\Delta}_n\to\hat{\Delta}_n$ in the following way: if $y$ is the barycenter of $\Delta_n$ then $h(ty,t)=(ty,t)$ for all $t$; if $y$ is different from the barycenter let $g^{-1}(y)=(x,s)$ and define $h(ty,t):=g(f(tx,t),s))$.

\begin{figure}
\includegraphics[scale=0.4]{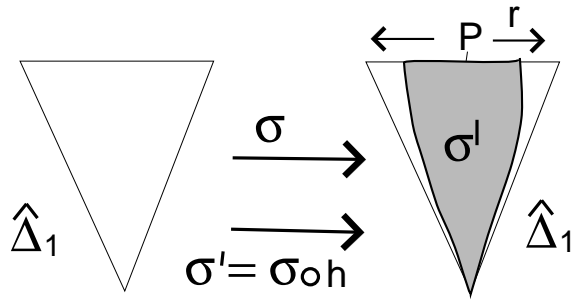}
\caption{We can see the 1-simplex $\sigma'$, constructed in the proof, that is 4-close to $\sigma$  but without touching the boundary of $\hat{\Delta}_1$.}
\label{fig:lemaDense}
\end{figure}
Define $\sigma':=\sigma\comp h$. Since $\theta$ is strictly positive outside the vertex and $\sigma$ verifies that the image of the interior of $\hat{\Delta}_n$ lies completely into $U$, we obtain that the image by $\sigma'$ of $\hat{\Delta}_n$ minus its vertex is contained in $U$. On the other hand we have $\sigma\sim_b\sigma'$ by the way that $\Omega$ and $\theta$ are defined. 

Notice that in the proof of Case (1) the induction procedure was not necessary. 

Now we treat case (2). Here the induction procedure is needed. Start by noticing that the pair of subsets $X\supset Z$ meet the hypothesis of Lemma~\ref{lem:defopenstratum} for $X=S$ and $Z=Q$. Fix $b'>b$ Let $\{U_j\}_{j\in J}$ and $\{g_j\}_{j\in J}$ be the partition of $Z$ and the functions $g_j:U_j\to X\setminus Z$ predicted in Lemma~\ref{lem:defopenstratum} for the parameter $b'$. For any $j\in J$ let $V_j$ denote the interior of $U_j$ in $Z$. Then the union $\cup_{j\in J}V_j$ is open and dense in $Z$. Since $\sigma$ has image in $Z$ and $\dim(Z)<\dim(X)$, by induction, up to $\sim_b$-equivalence we can replace $\sigma$ by a sum of simplexes $\sigma_i$ such that each of them has image contained in a $U_j$. By inequality~(\ref{eq:goodapprox}) for parameter $b'$, we have the equivalence $\sigma_i\sim_b g_j\comp\sigma_i$, and $g_j\comp\sigma_i$ has image in $U$. 
\endproof

\subsection{Triviality for \texorpdfstring{$b<1$}{b<1}}

\begin{prop}
\label{prop:b<1trivial}
Let $(X,x_0,d)$ be a metric subanalytic subgerm of $\RR^n$ with the inner or the outer metric. Then $MDH^b_0(X,x_0,d;A)=A$ and $MDH^b_k(X,x_0,d;A)=0$ for $k>0$. 
\end{prop}
\proof
By Theorem 2.1 of \cite{reembedding} there exists a subanalytic subgerm  $(X',x'_0)\subset \RR^m$ such that $(X,x_0,d)$ bi-Lipschitz homeomorphic to $(X',x'_0)$ both with the inner and outer metric. This jointly with Corollary \ref{func_inner} reduce the problem to the outer metric. By the l.v.a. condition and the triangle inequality any two $n$-simplexes are $\sim_b$ equivalent for $b<1$. 
\endproof

\section{Moderately discontinuous functoriality}
\label{section:MDfunctoriality}

In this section we improve functoriality properties of the $b$-MD homology for a fixed $b$ by allowing a certain class of non-continuous maps, which we call $b$-maps. This makes our theory quite flexible. The discontinuities that we allow are {\em moderated} in a Lipschitz sense. This may be seen as a motivation for the name of our homology.


\subsection{Functoriality for \texorpdfstring{$b$}{b}-maps}

\begin{definition}\label{def:point2}
Let $(X, x_0, d_X)$ be a metric subanalytic germ. In line with Definition~\ref{def:point}, we define a {\em point in $X$} to be a continuous l.v.a. subanalytic map germ $p:(0,\epsilon) \to X$. For any subanalytic $Y\subseteq X$, we say that {\em $p$ is contained in Y}, if Im$(p)\subseteq Y$. Observe that a point in $X$ is the same as a  l.v.a. $0$-simplex of $X$. 

Two points $p$ and $q$ are called {\em $b$-equivalent}, for $b\in (0, +\infty)$, and we write $p \sim_b q$, if 
$$
\lim_{t \to 0} \frac{d_X(p(t), q(t))}{t^b} = 0.
$$
\end{definition}

We can restate the equivalence of $(i)$ and $(iii)$ of  Lemma~\ref{lem:arccharacterization} as follows:.
\begin{remark}
\label{rem:point characterization of b-equivalence}
Let $\sigma_1,\sigma_2$ be $n$-simplices in $MDC^{\mathrm{pre},\infty}_\bullet(X,x_0,d_X)$. It follows from Lemma \ref{lem:arccharacterization} that we have the equivalence $\sigma_1\sim_b\sigma_2$ if and only if for any point $p$ in $\hat{\Delta}_n$, $\sigma_1 \circ p$ and $\sigma_2 \circ p$ are $b$-equivalent.
\end{remark}

\begin{definition}
\label{def:bmaps}
Let $(X, x_0, d_X)$ and $(Y, y_0, d_Y)$ be metric subanalytic germs, $b \in (0, \infty)$. A {\em $b$-moderately discontinuous subanalytic map ($b$-map, for abbreviation)} from $(X, x_0, d_X)$ to  $(Y, y_0, d_Y)$ is a finite collection $\{(C_i,f_i)\}_{i\in I}$, where $\{C_i\}_{i\in I}$ is a finite closed subanalytic  cover of $X$  and $f_i: C_i \to Y$ is a l.v.a. subanalytic map satisfying the following: for any $b$-equivalent pair of points $p$ and $q$ contained in $C_i$ and $C_j$ respectively ($i$ and $j$ may be equal), the points $f_i\circ p$ and $f_j\circ q $ are b-equivalent in $Y$.

Two $b$-maps $\{(C_i,f_i)\}_{i\in I}$ and $\{(C'_i,f'_i)\}_{i\in I'}$ are called {\em $b$-equivalent} if for any $b$-equivalent pair of points $p$, $q$ with Im$(p) \subseteq C_i$ and Im$(q) \subseteq C'_{i'}$, the points $f_i\circ p$ and $f'_{i'}\circ q$ are b-equivalent in $Y$.

We make an abuse of language and we also say that a $b$-map from $(X,x_0,d_X)$ to $(Y,y_0,d_Y)$ is an equivalence class as above. 

For $b = \infty$, a $b$-map from $X$ to $Y$ is a l.v.a. subanalytic map from $X$ to $Y$.
\end{definition}

\begin{prop}[Definition of composition of $b$-maps]
\label{prop:compbmaps}
Let $\{(C_i,f_i)\}_{i\in I}$ be a $b$-map from $X$ to $Y$ and let $\{(D_j,g_j)\}_{j\in J}$ be a $b$-map from $Y$ to $Z$.  Then the composition of the two b-maps is well defined by $\{ (f_i^{-1}(D_j) \cap C_i, g_j \circ f_{i| f_i^{-1}(D_j) \cap C_i})\}_{(i,j) \in I \times J}$. 

\end{prop}
\begin{proof} 
Any pair of $b$-equivalent points $p$ and $q$ that are contained in $f_{i_1}^{-1}(D_{j_1}) \cap C_{i_1}$ resp. $f_{i_2}^{-1}(D_{j_2}) \cap C_{i_2}$ are sent by $f_{i_1}$ resp. $f_{i_2}$ to $b$-equivalent points in $Y$ contained in $D_{j_1}$ resp. $D_{j_2}$. Those are sent by $g_{j_1}$ resp. $g_{j_2}$ to $b$-equivalent points in $Z$.

Let $\{(\hat{C}_i, \hat{f}_i)\}_{i \in \hat I}$ and $\{(\hat{D}_j, \hat{g}_j)\}_{j \in \hat{J}}$ be $b$-equivalent to $\{(C_i,f_i)\}_{i\in I}$ and $\{(D_j,g_j)\}_{j\in J}$ respectively. Let $p$ and $q$ be $b$-equivalent points contained in $f_i^{-1}(D_j) \cap C_i$ and $\hat {f}_{\hat {i}}^{-1}(\hat {D}_{\hat {j}}) \cap \hat {C}_{\hat {i}}$ respectively. By the exact same reasoning $p$ and $q$ are sent to $b$-equivalent points in $Z$ by $g_j\circ f_i$ and $\hat{g}_{\hat j}\circ \hat{f}_{\hat{i}}$ respectively.
\end{proof}

\begin{cor}
The category of metric subanalytic germs with $b$-maps is well defined.
\end{cor}

\begin{figure}
\includegraphics[scale=0.3]{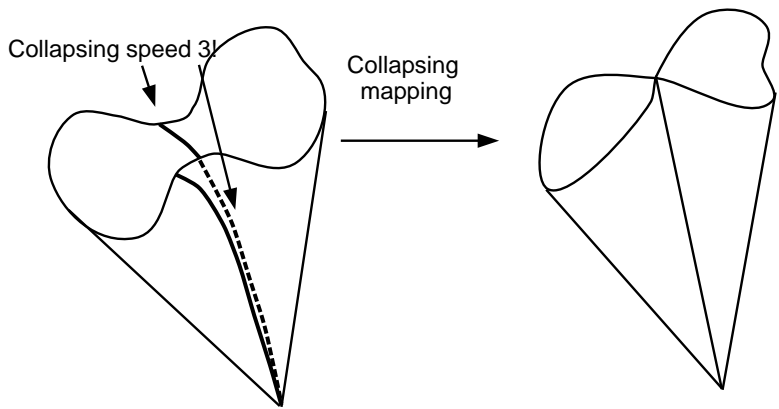}
\caption{Consider the surface on the left with the outer metric: it has a thin part collapsing at rate 3 and a thicker part collapsing linearly. The collapsing mapping in the figure has inverse as a 2-map. Then it induces an isomorphism in the 2-MD homology groups.  }
\label{fig:col}
\end{figure}

\begin{definition}
A $b$-map between pairs of metric subanalytic germs $(X,Y,x_0, d_X)$ and $(\tilde{X}, \tilde{Y},\tilde{x}_0, d_{\tilde{X}})$ is a $b$-map from $X$ to $\tilde{X}$ admitting a representative $\{(C_i, f_i )\}_{i \in I}$  for which the image of $C_i \cap Y$ under $f_i$ is contained in $\tilde{Y}$ for any $i$.
\end{definition}

Let $\phi:=\{(C_i, f_i)\}_{i \in I}$ be a $b$-map between two pairs $(X,Y,x_0, d_X)$ and $(\tilde{X}, \tilde{Y},\tilde{x}_0, d_{\tilde{X}})$. We are going to define a homomorphism
$$
\phi^{b}_\bullet: MDC^{b}_\bullet((X, x_0,d_X);A)\to MDC^{b}_\bullet((\tilde{X}, \tilde{x}_0, d_{\tilde{X}});A)
$$
depending on $\{(C_i, f_i)\}_{i \in I}$ that clearly descends to a homomorphism on the relative chain complexes.
Following Remark~\ref{rem:checkindependence small chains}, to define $\phi^{b}_\bullet$, we define a homomorphism
$$
\phi^{pre,b}_\calC: MDC^{pre, \infty ,\calC}_\bullet((X, x_0,d_X);A)\to MDC^{b}_\bullet((\tilde{X}, \tilde{x}_0, d_{\tilde{X}});A)
$$
where $\calC$ is a finite closed subanalytic refinement of $\{C_i\}_i$ as follows: the image of any $\sigma\in MDC^{\mathrm{pre},\infty,\calC}_\bullet((X, x_0,d_X);A)$ is contained in some $C_i$. We define the image of $\sigma$ under $\phi^{pre,b}_\calC$ to be $f_i \circ \sigma$ and extend this definition linearly.

There are five things to be checked to guarantee that $\phi^{pre,b}_\calC$ and $\phi^{b}_\bullet$ are well-defined:
\begin{enumerate}
\item\label{b-map functoriality item1} If the image of $\sigma$ is also contained in a different $C_j$, $f_j \circ \sigma$ is $b$-subdivision equivalent to $f_i \circ \sigma$;
\item\label{b-map functoriality item2} $\phi^{pre,b}_\calC$ is compatible with the $b$-equivalence relation;
\item\label{b-map functoriality item3} $\phi^{pre,b}_\calC$ is compatible with $\infty$-immediately equivalences;
\item\label{b-map functoriality item4} If $\calC$ and $\tilde{\calC}$ are different refinements of $\{C_i\}_i$, $\phi^{pre,b}_\calC$ and $\phi^{pre,b}_{\tilde{\calC}}$ define the same $\phi$;  
\item\label{b-map functoriality item5} If $\{(\tilde{C}_{i},\tilde{f}_i) \}_{i \in \tilde{I}}$ is $b$-equivalent to $\{(C_i, f_i) \}_{i \in I}$, consider a refinement $\calC$ that refines both $\{C_i\}_{i \in I}$ and $\{\tilde{C}_{i}\}_{i \in \tilde{I}}$. The image of any $\sigma\in MDC^{\mathrm{pre},\infty,\calC}_\bullet((X, Y, x_0,d_X);A)$ is contained both in some $C_i$ and in some $\tilde{C}_j$. The two simplices $f_i \circ \sigma$ and $\tilde{f}_j \circ \sigma$ are $b$-subdivision equivalent.
\end{enumerate}

For (\ref{b-map functoriality item1}), we are going to show that $f_i \circ \sigma$ and $f_j \circ \sigma$ are $b$-equivalent, where $\sigma$ is an $n$-simplex whose image is contained both in $C_i$ and $C_j$. Let $p$ be a point in $\hat{\Delta}_n$. By definition of $b$-map, $f_i \circ \sigma \circ p$ and $f_j \circ \sigma \circ p$ are $b$-equivalent. So the statement follows from Remark~\ref{rem:point characterization of b-equivalence}. For (\ref{b-map functoriality item5}), we can use the exact same argument to show that $f_i \circ \sigma$ and $\tilde{f}_j \circ \sigma$ are $b$-equivalent.
Statement~(\ref{b-map functoriality item3}) is obvious.
 For (\ref{b-map functoriality item2}), 
let $\sigma$ and $\sigma'$ be l.v.a. simplices that are $b$-equivalent whose images are contained in $C_{i_1}$ resp. $C_{i_2}$. 
We have to show that $f_{i_1}\comp\sigma$ and $f_{i_2}\comp\sigma'$ are $b$-equivalent. Suppose they were not. Then there would be a point $p$ in $\hat{\Delta}_n$ for which $f_{i_1}\circ\sigma\circ p$ and $f_{i_2}\circ\sigma'\circ p$ are not $b$-equivalent. So $\sigma\circ p$ and $\sigma'\circ p$ would not be $b$-equivalent.
To show (\ref{b-map functoriality item4}), take a common refinement $\calD$ of $\calC$ and $\tilde{\calC}$. Then, $\calD$ defines the same $\phi^{b}_\bullet$ as $\calC$ and the same as $\tilde{\calC}$.


Then, we have proved the following:

\begin{prop}[Functoriality for $b$-maps]\label{prop:b-discfunctoriality} 
For a fixed $b\in (0, \infty]$, there are well defined functors  
$$(X,Y,x_0,d_X)\mapsto MDC^{b}_*((X,Y,x_0,d_X);A)$$
$$(X,Y,x_0,d_X)\mapsto MDH^{b}_*((X,Y,x_0,d_X);A)$$
from the category of pairs of metric subanalytic germs with $b$-maps to $Kom(Ab)^-$ and $GrAb$ respectively.
\end{prop}

\begin{cor}
The $b$-moderately discontinuous homology is invariant by isomorphisms in the category of $b$-maps.
\end{cor}

\subsection{A sufficient geometric condition}


We have the following geometric condition that is sufficient for a collection $\{(C_i,f_i)\}_{i\in I}$ to define a $b$-map:

\begin{lema}
\label{geometric b-maps}
Let $(X,x_0, d_X)$ and $(Y,y_0, d_Y)$ be  metric subanalytic germs, $b \in (0, \infty)$. Let $\{(C_i,f_i)\}_{i\in I}$ be a finite collection, where $\{C_i\}_{i\in I}$ is a finite closed subanalytic  cover of $X$, the maps $f_i:C_i\to Y$ are  Lipschitz l.v.a. subanalytic and admit an extension $\overline{f}_i:\calH_{b,\eta}(C_i; X)\to Y$ that are Lipschitz (non-necessarily subanalytic) l.v.a. maps for some $\eta \in \mathbb{R}_{>0}$, and the following condition is satisfied for any pair of indices $i,j\in I$:
\begin{equation}\label{eq:b-map}
\lim\limits_{\epsilon\to 0}\frac{\sup\limits\{d_Y(\overline{f}_i(x),\overline{f}_j(x));x\in L_{X,\epsilon}\cap \calH_{b,\eta}(C_i; X) \cap \calH_{b,\eta}(C_{j}; X)\}}{\epsilon^b}=0
\end{equation}
Then, $\{(C_i, f_{i}) \}_{i \in I}$ is a $b$-map.
\end{lema}
\begin{proof}
We suppose $x_0 = 0$. Let $p$ and $q$ be two $b$-equivalent points contained in $C_i$ and $C_j$ respectively. We have to show that $f_i\comp p$ and $f_j\comp q$ are $b$-equivalent. 

Since $p$ and $q$ are $b$-equivalent, the image of $q$ is contained in $\calH_{b,\eta}(C_i; X)$. By the triangle inequality we have
$$
\frac{d_X(f_{i}\circ p(t), f_{j}\circ q(t))}{t^b} 
\leq \frac{d_X(\overline{f}_i\circ p(t), \overline{f}_i\circ q(t))}{t^b} + \frac{d_X(\overline{f}_i\circ q(t), \overline{f}_j\circ q(t))}{t^b}.
$$

Since $\overline{f}_i$ is Lipschitz and $p$ and $q$ are $b$-equivalent, the first summand of the right hand side converges to $0$ as $t$ approaches $0$. The second summand converges to $0$ by the equation~(\ref{eq:b-map}).

\end{proof}

\begin{lema}
Let $\{(C_i,f_i)\}_{i\in I}$ and $\{(C'_i,f'_i)\}_{i\in I'}$ be two collections fulfilling the conditions of Lemma \ref{geometric b-maps}. If for any $i\in I$ and $i'\in I'$, it is 
\begin{equation}\label{eq:b-map equivalence}
\lim\limits_{\epsilon\to 0}\frac{\sup\limits\{d_Y(\bar{f}_i(x),\bar{f}'_{i'}(x));x\in L_{X,\epsilon}\cap \calH_{b,\eta}(C_i; X) \cap \calH_{b,\eta}(C'_{i'}; X)\}}{\epsilon^b}=0,
\end{equation}
the two $b$-maps defined by them are $b$-equivalent.
\end{lema}
\begin{proof}
The proof is analogous to the one of Lemma \ref{geometric b-maps}.
\end{proof}

\subsection{Applications using b-maps}

\begin{definition}
A {\em section of a $b$-map $\varphi:X\to Y$ ($b$-section for short)} is a $b$-map $\psi:Y\to X$ such that $\varphi\comp \psi=Id_Y$ in the category of $b$-maps.
\end{definition}

\begin{remark}
Notice that admitting sections in the category of $b$-maps is much less restrictive than in the category of continuous subanalytic maps, since $b$-maps are only piecewise continuous, and piecewise univalued.
\end{remark}

\begin{theo}
\label{theo:b-map sections}
Let $\varphi:X\to Y$ be a Lipschitz l.v.a. subanalytic map between two metric subanalytic germs so that there exists a finite closed subanalytic  cover $\{Y_i\}_{i\in I}$ of $Y$ so that 
$$
\varphi|_{\varphi^{-1}(Y_i)}:\varphi^{-1}(Y_i)\to Y_i
$$
admits a $b$-section $\{(Y_{i,j},\psi_{i,j})\}_{j\in J_i}$ for any $i\in I$. Suppose that for any two points $p$ and $q$ in $X$ for which $\varphi\circ p$ and $\varphi\circ q$ are $b$-equivalent in $Y$, $p$ and $q$ are $b$-equivalent in $X$. 
Then, $\varphi$ induces an isomorphism  
$$\varphi_*\colon MDC^b_{\bullet}(X; A)\to MDC^b_{\bullet}(Y; A).$$ Consequently $\varphi_*$ induces an isomorphism in $b$-MD homology.
\end{theo}
\begin{proof}
The $b$-sections glue to a global $b$-section $(Y_{i,j},\psi_{i,j})_{i \in I, j\in J_i}$: let $p_1$ and $p_2$ be $b$-equivalent points in $Y_{i_1, j_1}$ and $Y_{i_2,j_2}$ respectively. Then, $\varphi\circ \psi_{i_l, j_l}\circ p_l$ is $b$-equivalent to $p_l$ for $l=1,2$ and therefore $\varphi\circ \psi_{i_l, j_l}\circ p_1$ and $\varphi\circ \psi_{i_l, j_l}\circ p_2$ are $b$-equivalent. Therefore, by hypothesis so are $\psi_{i_l, j_l}\circ p_1$ and $\psi_{i_l, j_l}\circ p_1$.

To show that the global $b$-section is in fact the inverse of $(X,\varphi)$, we have to show that $\{(\varphi^{-1}(Y_{i,j}), \psi_{i,j} \circ \varphi)\}_{i \in I, j\in J_i}$ is $b$-equivalent to $(X, id_X)$. Let $p$ and $q$ be $b$-equivalent points in $\varphi^{-1}(Y_{i,j})$ and $X$ respectively. Then $\varphi\circ\psi_{i,j}\circ\varphi\circ p$ is $b$-equivalent to $\varphi \circ p$, which is $b$-equivalent to $\varphi\circ q$ as $\varphi$ is Lipschitz. Therefore, $\psi_{i,j}\circ\varphi\circ p$ is $b$-equivalent to $q$.
\end{proof}

%

\begin{cor}
\label{cor:b-maps colapsediameter}
Let $\varphi:(X,x_0,d_X)\to (Y,y_0,d_Y)$ be a Lipschitz l.v.a. subanalytic map between two metric subanalytic germs so that there exists a  finite closed subanalytic  cover $\{Y_i\}_{i\in I}$ of $Y$ and open sets $U_i$ containing $Y_i$ for every $i\in I$ such that there is a $b$-horn neighborhood $\calH_{b,\eta}(Y_i;Y)$ contained in $U_i$, and  
$$\varphi|_{\varphi^{-1}(U_i)}:\varphi^{-1}(U_i)\to U_i$$
admits a section $\psi_i$ in the category of Lipschitz l.v.a. subanalytic maps for any $i\in I$. Suppose that 
$$\lim\limits _{t\to 0^+}\frac{\sup\{diam(\varphi^{-1}(y)):y\in L_{Y,t}\}}{t^b}=0,$$
Then, $\varphi$ fulfills the hypothesis of Theorem~\ref{theo:b-map sections} and therefore induces an isomorphism in $Kom(Ab)^-$ and $GrAb$.
\end{cor}

\begin{proof}
Let $p_1$ and $p_2$ be points in $X$ for which $\varphi\circ p_1$ and $\varphi\circ p_2$ are $b$-equivalent and contained in $Y_i$ and $Y_j$ respectively. By Remark~\ref{rem: b-equ b-horn}, $\varphi\circ p_2$ is contained in $\calH_{b,\eta}(Y_i;Y)$. 
As $\psi_i$ is Lipschitz, $\psi_i \circ \varphi \circ p_1$ and $\psi_i \circ \varphi \circ p_2$ are $b$-equivalent. Further, if $K_l$ is a l.v.a. constant for $\varphi\circ p_l$, $l \in \{1,2\}$, we have
$$
\frac{d_X(p_l(t), \psi_i \circ \varphi \circ p_l(t))}{t^b}
\leq K_l^b \frac{\sup \{ diam(\varphi^{-1}(y)): y\in L_{Y, \|\varphi\circ p_l(t)\|} \} }{\|\varphi\circ p_l(t)\|^b}
$$
and therefore $p_l$ and $\psi_i \circ \varphi \circ p_l$ are $b$-equivalent. Using the triangle inequality, we get that $p_1$ and $p_2$ are $b$-equivalent.
\end{proof}

The following corollary is an example of how $b$-maps and Theorem~\ref{theo:b-map sections} can be used concretely.

\begin{cor}\label{cor:colapsdiam point}
Let $X$ be a metric subanalytic germ such that 
$$\lim\limits _{t\to 0^+}\frac{diam(L_{X,t})}{t^b}=0$$
Then $X$ has the $b$-MD homology of a point in the category of metric subanalytic germs (recall Definition~\ref{def:point}).
\end{cor}
\proof
Map $X$ to $[0,1)$ by outer distance to the vertex of $X$ and use the previous corollary. Considering the trivial cover of $X$ by the single open subset $X$, the required section is the parametrization of an arc in $X$ by its distance to the origin. 
\endproof

\section{A useful isomorphism of \texorpdfstring{$b$}{b}-MD Homology} 

In this Section we prove a very useful improvement of Corollary~\ref{cor:b-maps colapsediameter}, which holds in a slightly more special situation. Its proof does not use $b$-maps, and the technique that we use also shows the flexibility of $b$-MD homology.

In the next theorem let $\mathbb{S}_t^{m-1}$ denote the sphere of radius $t$ centered at a point $x_0$ in $\RR^m$.

\begin{theo}
\label{theo:colapsediameter2}
Let $(Y,x_0), (X,x_0)\subset(\RR^m,x_0)$ be closed subanalytic germs, which we endow with the outer metric $d_{out}$ from $\RR^m$. Let $\varphi:(X,x_0)\to (Y,x_0)$ be a l.v.a. (non-necessarily Lipschitz) subanalytic map germ which satisfies 
\begin{equation}
\label{eq:diamcond_one}
 \lim\limits _{t\to 0^+}\frac{sup\{diam(\varphi^{-1}(y)\cup \{y\}):y\in Y_t\}}{t^b}=0,
\end{equation}
where $Y_t=Y\cap \mathbb{S}_t^{m-1}$. 
Then $\varphi_*$ induces an isomorphism  
$$\varphi_*\colon MDC^b_{\bullet}(X,d_{out}; A)\to MDC^b_{\bullet}(Y,d_{out}; A).$$
Consequently $\varphi_*$ induces isomorphism in $b$-MD homology.
\end{theo}
\proof

Let $\sigma_1$, $\sigma_2$ be simplices in $MDC^{\mathrm{pre},\infty}_{\bullet}(X; A)$. 
The condition~(\ref{eq:diamcond_one}) implies the equivalence $\sigma_i\sim_b\varphi\comp\sigma_i$ for $i=1,2$ in $X$. Then we have the equivalence $\varphi\comp\sigma_1\sim_b\varphi\comp\sigma_2$ if and only if we have the equivalence $\sigma_1\sim_b\sigma_2$. 

Then, since the composition with $\varphi$ is compatible with the immediate equivalences, by Remark~\ref{rem:checkindependence} we obtain a well-defined morphism of complexes 
$$\varphi_*\colon MDC^b_{\bullet}(X; A)\to MDC^b_{\bullet}(Y; A).$$

Next, we show that $\varphi$ induces an injective morphism of complexes 
$$\varphi_*\colon MDC^b_{\bullet}(X; A)\to MDC^b_{\bullet}(Y; A).$$

Let $[z]\in MDC^{b}_{\bullet}(X; A)$ be such that $\varphi_*([z])=0$. Express $z=\sum_{i\in I}a_i\sigma_i$.
By Lemma~\ref{lema:banulacion} there exists a sequence of immediate equivalences $w_1\to_\infty ...\to_\infty w_r$ in $MDC^{\mathrm{pre},\infty}_{\bullet}(Y; A)$ such that we have $w_1=\sum_{i\in I}a_i\varphi\comp\sigma_i$ and $w_r\sim_b 0$. Each of the immediate equivalences $w_i\to_\infty w_{i+1}$ involves a family of homological subdivisions. Using the same families of homological subdivisions we produce a sequence $z_1\to_\infty ...\to_\infty z_r$ in $MDC^{\mathrm{pre},\infty}_{\bullet}(Y; A)$ such that we have the equalities $z_1=z=\sum_{i\in I}a_i\sigma_i$ and  $\varphi_*(z_j)=w_j$ for all $j$. 

Express $z_r=\sum_{j\in J}b_j\psi_j$. Then we have the equivalence $w_r=\sum_{j\in J}b_j\varphi\comp\psi_j\sim_b 0$. Since the diameter condition~(\ref{eq:diamcond_one}) implies the equivalence $\psi_i\sim_b\varphi\comp\psi_i$ for all $i\in J$, we have that $z_r\sim_{b}\varphi_*(z_r)=w_r$. Thus, we obtain the equivalence $z_r\sim_b 0$ and this implies the vanishing $[z]=0$ in $MDC^b_{\bullet}(X; A)$ by Lemma~\ref{lema:banulacion}. This finishes the proof of injectivity. 

Now we show surjectivity. By Hardt's Triviality Theorem, there exists a subanalytic stratification $\calY$ of $Y$, such that the restriction of $\varphi$ to the preimage of any stratum is a subanalytic trivial fibration. 

Denote by $\calC$ the collection of closed analytic subsets of $Y$ given by the closures of the strata of the stratification $\calY$. Let $U\subset Y$ be the union of the interiors of the strata in $Y$. Then $U$ is a dense subanalytic subset of $Y$. By Proposition~\ref{prop:smallqis} it is enough to show that any simplex $\sigma$ in $MDC^{\mathrm{pre},\infty,\calC\cap U}_{\bullet}(Y;A)$ has a subanalytic lifting to $X$.

For such a simplex $\sigma:\hat{\Delta}_n\to Y$ the image of $\hat{\Delta}_n\setminus\{(0,0)\}$ is completely contained in an stratum of $\calY$. Then, using the triviality given by Hardt's Triviality Theorem, we construct a subanalytic lifting $\widetilde{\sigma}:\hat{\Delta}_n\setminus\{(0,0)\}\to X$ of $\sigma$ outside the vertex. We define the extension of $\widetilde{\sigma}$ to $\hat{\Delta}_n$ by $\widetilde{\tau}_i(0,0):=x_0$. We denote the extension to $\hat{\Delta}_n$ also by $\widetilde{\sigma}$. The  graph of $\widetilde{\tau}_i$ is the closure of the graph of $\tau_i$, which is subanalytic. Hence $\widetilde{\sigma}$ defines a l.v.a subanalytic lifting whose image by $\varphi_*$ equals $\sigma_i$ in $MDC^b(Y,d_{out})$. 
\endproof

\section{Metric homotopy and \texorpdfstring{$b$}{b}-homotopy invariance}

Now we prove the invariance of MD Homology by different kinds of metric homotopies. Here the theory differs if we consider actual (Lipschitz l.v.a. subanalytic) maps or $b$-maps. For actual maps the notion of metric homotopy is simply a family of Lipschitz l.v.a subanalytic maps   with uniform Lipschitz and l.v.a. constant. For $b$-maps the definition is slightly more elaborated. 

In this section $I$ denotes the unit interval $[0,1]$.

\subsection{Metric homotopy}

\begin{definition}[Metric homotopy]
\label{def:1homotopy}
Let $(X,x_0,d_X)$ and $(Y,y_0,d_Y)$ be metric subanalytic germs. Let $f,g:(X,x_0,d_X)\to (Y,y_0,d_Y)$ be Lipschitz l.v.a. subanalytic maps. A continuous subanalytic map $H:X \times I \to Y$ is called a {\em metric homotopy between $f$ and $g$}, if there is a uniform constant $K\geq 0$ such that for any $s$ the mapping $H_s := H(-,s)$ is  Lipschitz l.v.a. subanalytic  with  Lipschitz l.v.a. constant $K$  and $H_0=f$ and $H_1=g$. 

\end{definition}

\begin{theo}\label{theo:homotopyinvariance}
Let $(X,x_0,d_X)$ and $(Y,y_0,d_Y)$ be metric subanalytic germs. Let $A$ be an abelian group. 
\begin{enumerate}
\item Let $f,g:(X,x_0)\to (Y,y_0)$ be  l.v.a. subanalytic maps such that there exists a continuous subanalytic mapping $H:X\times I\to Y$ with $H_0=f$ and $H_1=g$ satisfying that there exists a uniform constant $K>0$ such that for every $s$, the mapping $H_s$ is l.v.a. for the constant $K$. Then we have that both $f^\infty,g^\infty:MDC^\infty_\bullet(X;A)\to MDC^\infty_\bullet(Y;A)$ are the same in $\calD(Ab)^-$.
 \item\label{metric homotopy invariance} Let $f,g:(X,x_0,d_X)\to (Y,y_0,d_Y)$ be Lipschitz l.v.a. subanalytic maps that are metrically homotopic. Then 
 $$f_{\bullet},g_{\bullet}:MDC^{\star}_\bullet(X;A)\to MDC^{\star}_\bullet(Y;A)$$
 represent the same map in the category $\BB-\calD(Ab)^-$. As a consequence they induce the same homomorphism in MD Homology.
 
\end{enumerate}
\end{theo}
\begin{proof}
Let us prove Assertion (2). The proof of Assertion (1) is completely similar, disregarding metric considerations. 

A common proof for the analogue statement in singular homology uses the inclusions $x \to (x,0)$ and $x \to (x,1)$ from $X$ to $X\times I$ and constructs a chain homotopy between the maps they induce on the singular chain complex; functoriality then yields the desired result. To prove Assertion~(2), we imitate the idea behind that chain homotopy, but as $X\times I$ is not an object of the category of  metric subanalytic germs, we directly construct a chain homotopy $\eta^b$ from $f_{\bullet}^b$ to $g_{\bullet}^b$. Such a chain homotopy will be clearly compatible with the homomorphisms connecting the complexes for different $b$'s.

Let $H$ be a metric homotopy from $f$ to $g$. In order to construct the chain homotopy $\eta^b_\bullet:MDC^b_\bullet(X;A) \to MDC_{\bullet+1}^b(Y;A)$, by Remark~\ref{rem:checkindependence}, it is enough to construct a homomorphism $h_\bullet:MDC^{pre, \infty}_\bullet(X;A) \to MDC_{\bullet+1}^b(Y;A)$ fulfilling the two conditions of the remark. 

Define 
$$
\widehat{\Delta_n\times I}:=\{(t(x,s),t): (x,s)\in\Delta_n\times I, t\in [0,1)\} \subset \mathbb{R}^{n+2}.
$$
The parameter $s$ is the ``homotopy parameter'', and the parameter $t$ measures the proximity to the vertex, as usually along this paper. We have the notion of l.v.a. maps from $\widehat{\Delta_n\times I}$ to a metric germ $(X,x_0,d_X)$, in an analogous way with the case of maps from $\hat{\Delta}_n$. Moreover Definition~\ref{def:subdiv} extends in an obvious way to a notion of homological subdivision of $\widehat{\Delta_n\times I}$.

Let $\sigma:\hat{\Delta}_n\to X$ be a l.v.a. simplex. Define $\hat{h}_n(\sigma):\widehat{\Delta_n\times I}\to Y$ to be the continuous subanalytic extension of the map given by $(t(x,s),t) \mapsto H(\sigma(tx,t),s)$ for $t\neq 0$. The map $\hat{h}_n(\sigma)$ is subanalytic and l.v.a..

Let $\alpha_j:|K|\to \widehat{\Delta_n\times I}$ be triangulations of $\widehat{\Delta_n\times I}$ for $j=1,2$, and let 
$\{\rho_{j,i}\}_{i\in I_j}$ be an orientation preserving homological subdivisions of $\widehat{\Delta_n\times I}$ associated with each of the triangulations. For $j=1,2$ the sum $z_j:=\sum_{i\in I_j}\hat{h}_n(\sigma)\comp\rho_{j,i}$ is an element of $MDC_\bullet^{\mathrm{pre},\infty}(Y,A)$. By choosing a common refinement of the subanalytic triangulations $\alpha_1$ and $\alpha_2$ and arguing like in the proof of Lemma~\ref{lem:Seqinversion}, we show that there exists an element $z_3\in MDC_\bullet^{\mathrm{pre},\infty,}(Y,A)$ and immediate equivalences $z_1\to_\infty z_3$ and  $z_2\to_\infty z_3$.
This shows that the assignment $h_n(\sigma):=z_j$ in  $MDC_\bullet^{b}(Y,A)$ gives, extending by linearity, a well defined homomorphism
$$h_n:MDC^{pre, \infty}_n(X;A) \to MDC_{n+1}^b(Y;A).$$
Now we check that the conditions of Remark~\ref{rem:checkindependence} are satisfied. 

If we have two $b$-equivalent simplices $\sigma\sim_b\sigma'$, in order to prove the equivalence $h_n(\sigma)\sim_b h_n(\sigma')$, using the arc characterization of Lemma~\ref{lem:arccharacterization}, it is enough to prove that for any subanalytic l.v.a continuous arc $\gamma:[0,\epsilon)\to \widehat{\Delta}_n$, with coordinates $\gamma(t)=(\gamma_2(t)\gamma_1(t),\gamma_2(t))$, and for any subanalytic function $\rho:[0,\epsilon)\to I$, we have the vanishing of the limit
$$
\lim\limits_{t\to 0^+}\frac{d(H(\sigma(\gamma(t)),\rho(s)),\sigma'(\gamma(t)),\rho(s)))}{t^b}=0.
$$
Since the numerator is bounded by $Kd(\sigma(\gamma(t),\sigma'(\gamma(t))$, and we have $\sigma\sim_b\sigma'$ and $\gamma$ is l.v.a. the limit vanishes as needed.  

Let $\sigma$ be a $n$-simplex, and $\{\rho_i\}_{i\in I}$ be a homological subdivision of $\hat{\Delta}_n$ associated with a subanalytic triangulation $\alpha:|K|\to \hat{\Delta}_n$. The triangulation $\alpha$ induces a decomposition of $\widehat{\Delta_n\times I}$ that can be refined to a subanalytic triangulation $\beta$ of $\widehat{\Delta_n\times I}$. Let $\{\mu_k\}_{k\in K}$ be a homological subdivision associated to  $\beta$. Then we have that $h_n(\sigma)$, previously defined, coincides with $\sum_{k}sgn(\mu_k)\hat{h}_n(\sigma)\comp\mu_{k}$. 

Thus, we have constructed for every $n$ a well defined map 
$$\eta^b_n:MDC^{b}_n(X;A) \to MDC_{n+1}^b(Y;A).$$
In order to prove that it is a chain homotopy we have to check the equation $\partial\eta^b_{n}+\eta^b_{n-1}\partial=g^b_\bullet-f^b_\bullet$. For this we only need a canceling of interior boundaries very similar to the proof of Lemma~\ref{lem:compatpartialsum}.
\end{proof}

\subsection{\texorpdfstring{$b$}{b}-Homotopies.}

For the definition of $b$-homotopies we need a notion of product of a metric subanalytic germ $(X,0,d_X)$ with the interval $I$, which lives in the category of metric subanalytic germs. Moreover we need the hypothesis $d_{X,out}\leq d_X$ which in particular holds for the inner and the outer metrics.

\begin{definition}

Let $(X, 0, d_X)$ be a metric subanalytic germ. For $x\in X$ we denote by $||x||$ the usual euclidean norm of $x$, which may differ from $d_X(x,0)$. By $X \times_p I$, we denote the following metric subanalytic germ $(\tilde{X}, \tilde{v}, \tilde{d})$:
\begin{align*}
& \tilde{X} := \{ (x, \|x \| s) : x \in X, s \in I \}\subset X\times\RR \\
& \tilde{v} := (0,0)\\
& \tilde{d}((x_1, ||x_1||s_1), (x_2, ||x_2||s_2)) := \sup \{ d_X(x_1, x_2), \text{ } d_{X, \triangledown}((x_1, ||x_1||s_1), (x_2, ||x_2||s_2)) \}
\end{align*}
where $d_{X, \triangledown}$ is defined as follows: let $T$ denote the straight cone over the unit interval:

$$T := \{(d,ds) \in \mathbb{R}^2 : d \in [0,1], s \in [0, 1] \}$$
Let $\dtr$ denote the maximum metric on $T$. We define 
 
$$d_{X, \triangledown}((x_1, |x_1\|s_1), (x_2, \|x_2\|s_2)) := \dtr ( \|x_1\|, \|x_1\|s_1), (\|x_2\|, \|x_2\|s_2))$$



\end{definition}

\begin{lema} Let $d_X$  be a metric on a subanalytic germ $(X,x_0)$ such that $d_{X,out}\leq d_X$.  
 The following inequality holds
\begin{equation}\label{eq:homotopy}
d_{X, \triangledown}((x_1, s),(x_2, s))\leq M\sqrt{2} d_X(x_1, x_2)
\end{equation}
for any $x_1, x_2 \in X, s \in I$, where $M$ is the bi-Lipschitz constant between the maximum and the Euclidean norm on $T$.

Moreover, 
\begin{equation}\label{eq:tilded}
\tilde{d}((x_1, s),(x_2, s))\leq M\sqrt{2} d_X(x_1, x_2).
\end{equation}
\end{lema}
\proof
We have the following easy chain of inequalities:
$$d_{X, \triangledown}((x_1, s),(x_2, s))\leq M \sqrt{1 + s^2} |\|x_1 \| - \|x_2\|| \leq$$
$$\leq M\sqrt{1 + s^2} \|x_1 - x_2 \| \leq M\sqrt{1 + s^2} d_X(x_1, x_2).$$
Notice that $s\leq 1$.
To prove (\ref{eq:tilded}) we just use the previous inequality and the definition to get that $\tilde{d}\leq d_X\cdot max{M\sqrt{1 + s^2},1}$.  
\endproof

\begin{definition}[$b$-homotopy]\label{def:b-homotopy}
Let $(X, x_0, d_X)$ and $(Y, y_0, d_Y)$ be metric subanalytic germs. A {\em b-homotopy} is a $b$-map from $X \times_p I$ to $Y$.
\end{definition}

\begin{theo}\label{theo:homotopyinvariance2} If there is a $b$-homotopy $H$ with $H_0=f$ and $H_1=g$,  then  
 $$f_{\bullet},g_{\bullet}:MDC^{b}_\bullet(X;A)\to MDC^{b}_\bullet(Y;A)$$
 represent the same map in the category $\calD(Ab)^-$. As a consequence they induce the same homomorphism in MD homology.
\end{theo}
\begin{proof}
 For this proof we can follow the classical proof for singular homotopy much more closely: denote by $i_s:X\to X \times_p I$ the inclusion given by $i_s(x):=(x,||x||s)$ (which is a Lipschitz l.v.a.  subanalytic map by (\ref{eq:tilded})). It is enough to prove that $i_0$ and $i_1$ induce chain homotopic homomorphisms from $MDC^b_*(X)$ to $MDC^b_*(X\times_p I)$. 

Given any l.v.a. simplex $\sigma:\hat{\Delta}_n\to X$ we define $\hat{\eta}_n(\sigma):\widehat{\Delta_n\times I}\to X\times_p I$ to be the map $(t(x,s),t) \mapsto (\sigma(tx,t),||\sigma(tx,t)||s)$. 

In order to define the homomorphism $\eta_n:MDC^b_n(X)\to MDC^b_n(X\times_p I)$ we proceed as in the proof of Theorem~\ref{theo:homotopyinvariance}: choose an orientation preserving homological subdivision $\{\rho_k\}_{k\in K}$ of $\widehat{\Delta_n\times I}$ associated with a triangulation and define $\eta_n(\sigma):=\sum_{k\in K}\hat{\eta}_n(\sigma)\comp\rho_i$. Independence of the subdivision and compatibility with immediate equivalences is checked in the same way. Compatibility with $b$-equivalences 
follows by the inequality~(\ref{eq:tilded}).

Checking that the collection of maps $\eta_n$ for $n$ varying is a chain homotopy between the homomorphisms induced by $i_0$ and $i_1$ is like in Theorem~\ref{theo:homotopyinvariance}.

\end{proof}

\begin{definition}
\label{def:conicalretract}
Let $\iota:X\hookrightarrow Y$ be a Lipschitz l.v.a. map of  metric subanalytic germs which on the level of sets is an injection. A {\em$b$-retraction} is a $b$-map $r:Y\to X$ such that 
$r\comp\iota$ is the identity as a $b$-map. A {\em $b$-deformation retraction} is a $b$-retraction such that $\iota\comp r$ is $b$-homotopic to the identity. In those cases $X$ is called a {\em $b$-retract} or {\em $b$-deformation retract of $Y$}, respectively. A metric subanalytic germ is called {\em $b$-contractible} if it admits $[0,\epsilon)$ as a $b$-deformation retract.   
\end{definition}

The usual consequences of the existence of retracts and deformation retracts in topology hold trivially in our theory

\begin{cor}
\label{cor:retracts}
If $\iota:X\hookrightarrow Y$ admits a $b$-retraction the connecting homomorphisms in the long exact sequence of relative $b$-MD homology vanishes. If $\iota:X\hookrightarrow Y$ admits a $b$-deformation retraction, $\iota$ induces a quasi-isomorphism of $b$-MD chain complexes.
If $X$ is $b$-contractible then it has the $b$-MD homology of the metric subanalytic germ $[0,\epsilon)$.
\end{cor}

\section{\texorpdfstring{$b$}{b}-covers, Mayer-Vietoris, Excision and the \texorpdfstring{\v{C}}{C}ech spectral sequence}

\subsection{An extension of relative homology}\label{sec:relX}

For the proof and statement of the relative Mayer-Vietoris exact sequence we need to generalize the concept of relative homology.

\begin{definition}[Category of pairs of metric  subanalytic subgerms]
 
A {\em pair of metric  subanalytic subgerms} $(Y_1,Y_2,x_0,d_X)_{rel\ X}$ is given by two metric subanalytic subgerms $(Y_i,x_0)$ of a certain metric subanalytic germ $(X,x_0,d_X)$. Recall that on each $Y_i$ we consider the restriction metric $d_X|_{Y_i}$.

A {\em Lipschitz l.v.a. subanalytic map between the pairs of subgerms} $(Y_1,Y_2,x_0,d_X)_{rel\ X}$ and $(Y_1',Y_2',x_0,d_{X'})_{rel\ X'}$ is a  Lipschitz l.v.a. subanalytic  map $$(Y_1\cup Y_2,x_0,d_{X|_{Y_1\cup Y_2}})\to (Y_1'\cup Y_2',x_0,d_{X'|_{Y'_1\cup Y'_2}})$$ that carries $Y_i$ into $Y_i'$.

The \emph{category of pairs of metric subanalytic subgerms} has,  as objects,  pairs of metric subanalytic subgerms, and as morphisms,  Lipschitz subanalytic l.v.a. maps between them, as defined above.
\end{definition}

\begin{definition}\label{def:relX} Consider $b\in (0,+\infty]$. Given a pair of subanalytic subgerms \\
$(Y_1,Y_2,x_0,d_X)_{rel\ X}$, we identify  $MDC^b_\bullet(Y_i,x_0,d_{X|_{Y_i}})$ with the subgroup of\\ $MDC^b_\bullet(X,x_0,d_{X})$ generated by all l.v.a. simplices in $X$ that are $b$-equivalent to a representative fully contained in $Y_i$. We define the {\em complex of relative $b$-moderately discontinuous chains of the pair  $(Y_1,Y_2,x_0,d_X)_{rel\ X}$ with coefficients in $A$}, denoting it by $MDC^b_\bullet((Y_1,Y_2,x_0,d_X);A)_{rel X}$, as the quotient 
$$MDC^b_\bullet((Y_1,x_0,d_{X|_{Y_1}});A)+ MDC^b_\bullet((Y_2,x_0,d_{X|_{Y_2}});A)\biggm/ MDC^b_\bullet((Y_2,x_0,d_{X|_{Y_2}});A).$$


The {\em $b$-moderately discontinuous homology of the pair  $(Y_1,Y_2,x_0,d_X)_{rel\ X}$ is   denoted by $MDH^b_{*}((Y_1,Y_2,x_0,d_X);A)_{rel\ X}$} and it is the homology of the complex defined above. 

We abbreviate calling these complexes and graded abelian groups the {\em $b$-MD complex} and {\em $b$-MD homology} of  the pair $(Y_1,Y_2,x_0,d_X)_{rel\ X}$ .
\end{definition}

It is straightforward that a Lipschitz subanalytic l.v.a. map $f$ between pairs of subanalytic subgerms of some $(X,x_0,d_X)$, $(X',x'_0,d_{X'})$ induces morphisms at the level of $b$-MD chains for every $b\in (0,+\infty]$ (we denote by $f_*$ the morphism at the level of $b$-MD chains similarly to Notation \ref{not:f*C}). Moreover,  morphisms (\ref{eq:Cb_1-b_2}) and (\ref{eq:Hb_1-b_2}) also hold.
So, the following proposition is obvious from the definitions:

\begin{prop}\label{prop:b-functoriality} 
The assignments $$(Y_1,Y_2,x_0,d_X)_{rel\ X}\mapsto MDC^{\star}_\bullet((Y_1,Y_2,x_0,d_X);A)_{rel\ X}$$ and $$(Y_1,Y_2,x_0,d_X)_{rel\ X}\mapsto MDH^{\star}_*((Y_1,Y_2,x_0,d_X);A)_{rel\ X}$$ are functors from the category of pairs of metric subanalytic subgerms to $\BB-Kom(Ab)^-$ and $\BB-GrAb$ respectively.
\end{prop}

\[
\]

We have also the obvious generalizations of the definitions of small chain complexes with respect to a finite closed subanalytic covering $\calC$, and a dense subanalytic subset $U$. We denote them by $MDC^{pre, +\infty,\calC\cap U}_{\bullet}(Y_1,Y_2;A)_{rel X}$, $MDC^{b,\calC\cap U}_{\bullet}(Y_1,Y_2;A)_{rel X}$. We also have the analogue to Proposition \ref{prop:smallqis}: 
\begin{equation}\label{eq:smallqisMV}g:MDC^{b,\calC\cap U}_{\bullet}(Y_1,Y_2,d_X;A)_{rel X}\to  MDC^b_{\bullet}(Y_1,Y_2,d_X,A)_{rel X}
\end{equation}
is an isomorphism for every $b\in (0,\infty)$, and an isomorphism when $b=\infty$ and $U$ is the whole space.

\begin{remark}\label{re:equal_rel_X}
Note that when $Y_2\subset Y_1$ then $MDC^{b,\calC\cap U}_{\bullet}(Y_1,Y_2,d_X;A)_{rel X}$ coincides with $MDC^{b,\calC}_{\bullet}(Y_1,Y_2,d_X;A)$.
\end{remark}

\subsection{\texorpdfstring{$b$}{b}-covers}
\label{sec:bcovers}

In our theory we need a special notion of covers in order to prove Mayer-Vietoris, Excision, and computation of homology via a \v{C}ech spectral sequence.  

\begin{definition}
\label{def:b-saturated}
Let $(X,x_0,d_X)$ be a metric  germ and $Y_1$, $Y_2$ subanalytic subgerms, consider $b \in (0, \infty]$. 
A finite collection $\{U_i\}_{i \in I}$ of subanalytic subgerms is called  {\em a $b$-cover of $(Y_1,Y_2)$}, if it is a finite cover of $Y_1\setminus\{x_0\}$ and for any $i$ there is a subanalytic subset $\hat{U}_i\subseteq Y_1$ such that 
\begin{itemize}
\item for any two $b$-equivalent points $p,q:[0,\epsilon)\to ({Y_1},x_0)$, if $p$ has image in $U_i$ then $q$ has image in $\hat{U}_i$.  
\item For any finite $J \subseteq I$ there is a 
 {subanalytic map $r_J:\cap_{i\in J}\hat{U}_i \to \cap_{i \in J}  U_i$ which induces an inverse in the derived category of the  morphism of complexes:} 
$$MDC^b_\bullet((\cap_{i\in J} U_i,Y_2,x_0,d_X);A)_{rel X}\to MDC^b_\bullet((\cap_{i\in J}\hat{U}_i,Y_2,x_0,d_X);A)_{rel X}.$$
\end{itemize}
We call the collection $\{\hat{U}_i\}_{i \in I}$ a {\em $b$-extension of $\{U_i\}_{i \in I}$}.

\end{definition}

Observe that a $b$-deformation retract would be an example of a map $r_J$ in the previous definition, but other subanalytic maps may serve the same purpose.  

Notice that for $b=\infty$ any finite subanalytic cover of $X$ is a $b$-cover.

\begin{figure}%

\includegraphics[scale=0.3]{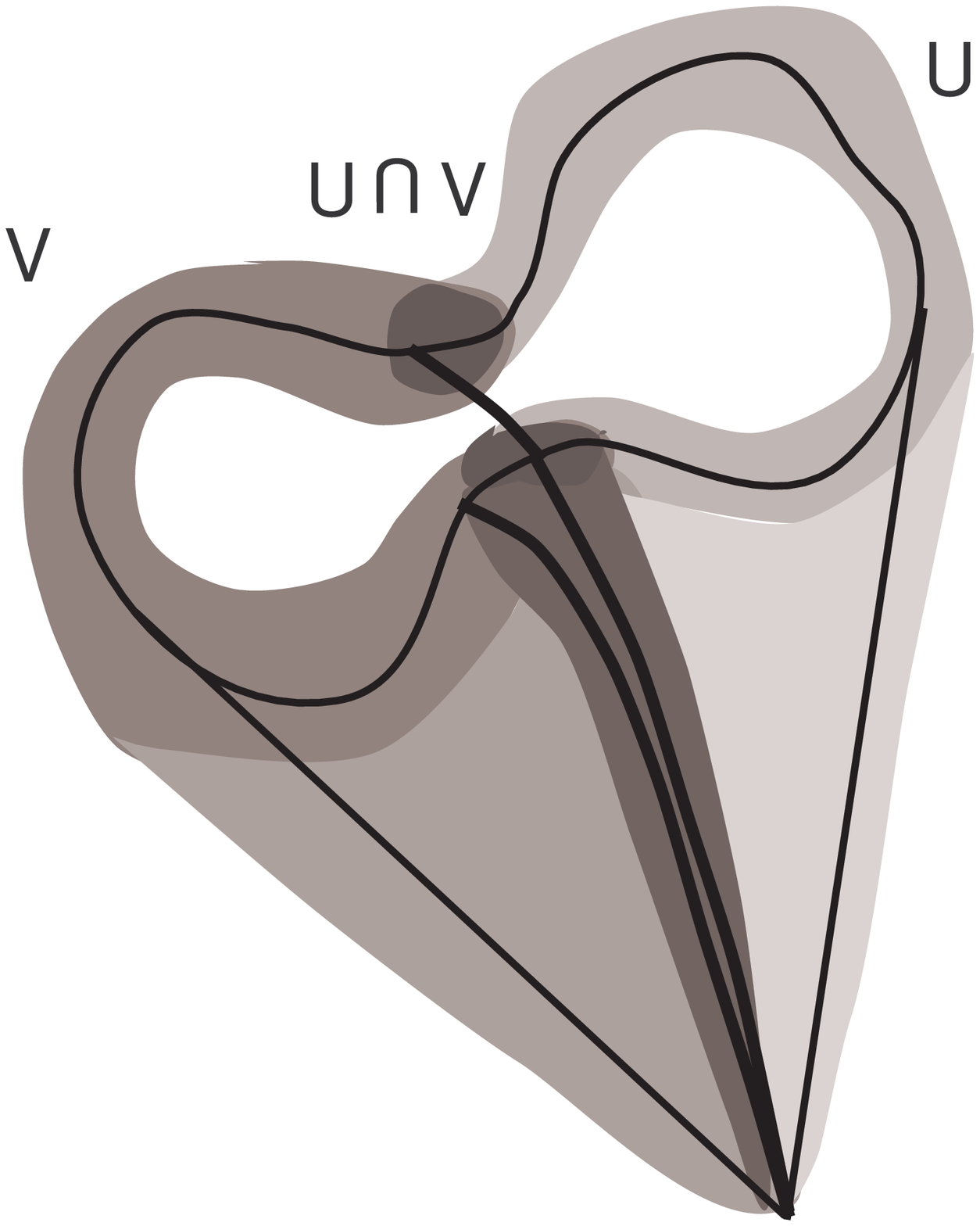}
\includegraphics[scale=0.8]{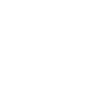}\includegraphics[scale=0.3]{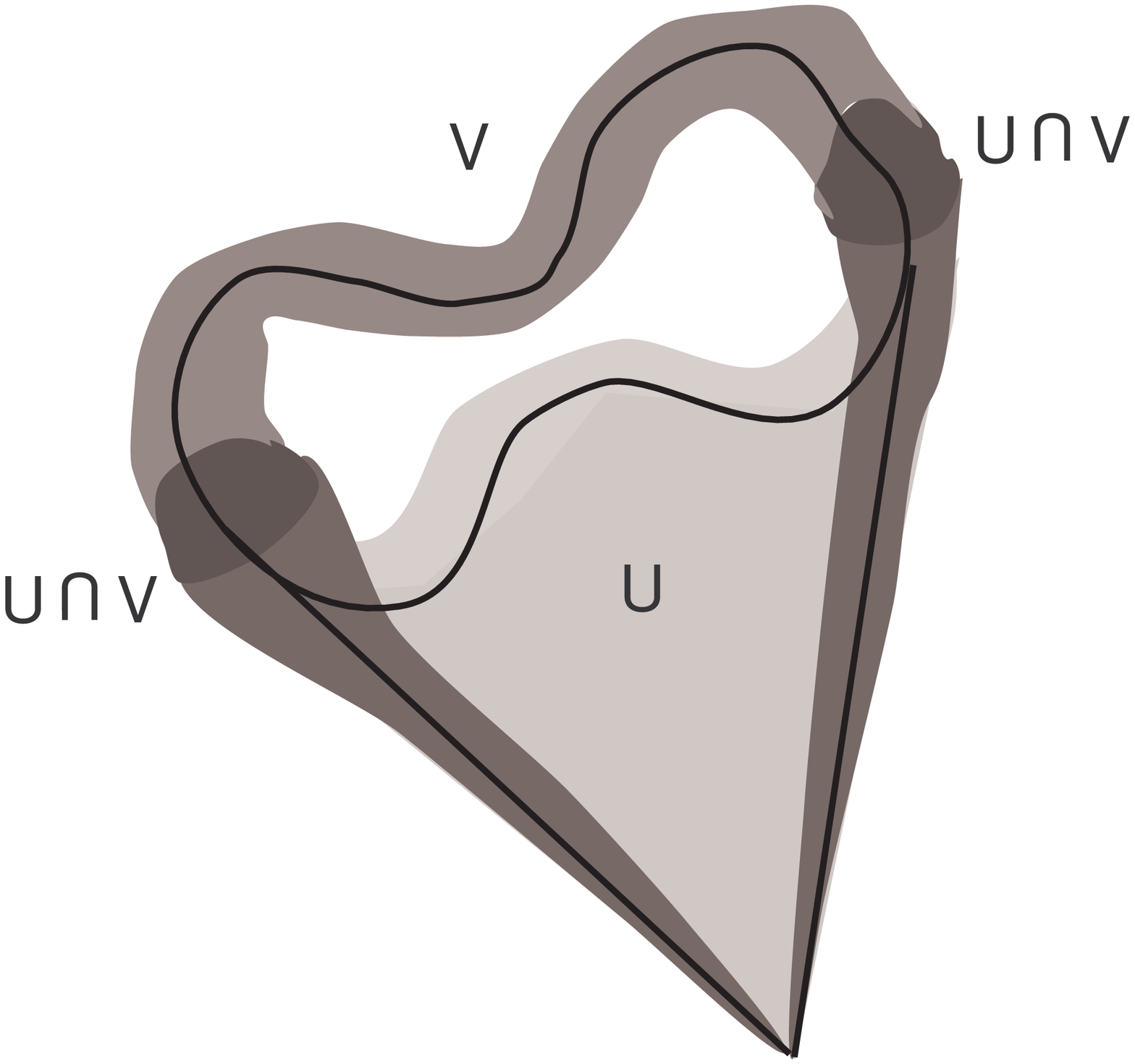}
\caption{In the same example of Figure \ref{fig:col} we can consider a 2-covering $U$, $V$ as in the figure on the left. On the right some covering $U$, $V$ that is not a 2-cover.} 
\end{figure}

The following remark is a consequence of the definition of $b$-horn neighborhood and of Theorem~\ref{theo:homotopyinvariance}.

\begin{remark}
\label{lem:b-saturated2}
In the terminology of the previous definition, when $Y_1=X$ and $Y_2=\emptyset$, the following two conditions imply the two conditions of the previous definition respectively:
\begin{itemize}
\item there is a $b$-horn neighborhood $\calH_{b,\eta}(U_i; X)$ contained in $\hat{U}_i$ for any $i \in I$ (see  Definition \ref{def:b-horn}).
\item For any finite $J \subseteq I$, the intersection $\cap_{i\in J} U_i$ is a $b$-deformation retract of $\cap_{i \in J} \hat{U}_i$ (see \ref{def:conicalretract}).
\end{itemize}
\end{remark}

\begin{lema}
\label{lem:inclusionsimplex}
Let $(X,x_0,d_X)$ be a metric  germ, $b \in (0, \infty]$ and $U\subset\hat{U}\subset X$ be subanalytic subsets such that for any two $b$-equivalent points $p,q:[0,\epsilon)\to (X,x_0)$, if $p$ has image in $U_i$ then $q$ has image in $\hat{U}_i$. If $\sigma_i:\hat{\Delta}_n\to X$ are $b$-equivalent  l.v.a simplices for $i=1,2$ and $\sigma_1$ is a simplex in $U$ then $\sigma_2$ is a simplex in $\hat{U}$. 
\end{lema}
\proof
Assume the contrary. Then $\sigma_2^{-1}(X\setminus\hat{U})$ is a subanalytic subset of $\hat{\Delta}_n$ having the vertex at its closure. By the subanalytic Curve Selection Lemma and Remark~\ref{rem:lvareduction} there exists a l.v.a subanalytic map $\gamma:[0,\epsilon)\to\hat{\Delta}_n$ such that $\gamma(t)$ is in $\sigma_2^{-1}(X\setminus\hat{U})$ for $t>0$. The arcs $p_i:=\sigma_i\comp\gamma$ give a contradiction.
\endproof

\begin{definition}
A $b$-cover $\{U_i\}_{i \in I}$ of $(Y_1,Y_2)$ is said to be a {\em closed (respectively open) $b$-cover} if any of the subsets $U_i$ is closed (respectively open) in $Y_1\setminus\{x_0\}$.
\end{definition}

The following proposition will be useful in reducing statements like Mayer-Vietoris and related results from open $b$-covers to closed $b$-covers.

\begin{prop}\label{prop:closed refinement} Let $\calU=\{U_1,...,U_k\}$ a finite open subanalytic covering of a subanalytic germ $(X,0)$. Then  there exists a subanalytic closed set $C_i$ contained in $U_i$ for every $i$ such that $ \{C_1,..,C_k\}$ is a closed covering.
\end{prop}

The proof is obtained by  repeatedly applying the following lemma:

\begin{lema}\label{lem:opencover}
Let $\calU=\{U_1,...,U_k\}$ be a finite open subanalytic covering of a subanalytic germ $(X,0)$. There exists a closed set $C_1$ contained in $U_1$ such that $ \{U_2,..,U_k,\mathring{C}_1\}$ is also an open covering of $(X,0)$ where $\mathring{C}_1$ is the interior of $C_1$. 
\end{lema}

\begin{proof} Let $L_X$ be the link of $X$. By the conical Structure Theorem (see Remark~\ref{rem:germ1}) we can take a subanalytic homeomorphism $h: C(L_X)\to X$ for a small enough representative for $(X,0)$ compatible with the covering $\calU$. That is, any $U_i$ coincides with $h(L_i)$ for a certain subanalytic subset $L_i$ of $L_X$. 

We prove that given a finite open subanalytic covering  $\mathcal{L}=\{L_1,...,L_k\}$  of $L_X$, there exists a closed set $D_1$ contained in $L_1$ such that $ \{L_2,..,L_k,\mathring{D}_1\}$ is also an open covering of $L_X$ where $\mathring{D}_1$ is the interior of $D_1$. 

To finish the proof we will consider the covering given by $C_i:=h(C(D_i))$. 

Let us prove the statement for a covering of $L_X$. Given $Y\subset L_X$ we denote by 
$\partial_{L_X} Y$ the boundary set $\bar{Y}\setminus\mathring{Y}$ of $Y$ in $L_X$. 

Let $K$ be $\partial_{L_X}L_1\cap (L_2\cup ...\cup L_k)$. Note that, since $L_1$ is open, in fact $\partial_{L_X}L_1$ equals $K$. 

Consequently the subanalytic fuction $\theta: K\to \RR$ defined by $$\theta(x):=d_{out}(x,\partial_{L_X}(L_2\cup...\cup L_k)$$ is strictly positive. Choose another strictly positive subanalytic  function $\eta:K\to \RR$ such that $\eta(x)<\theta(x)$ for every $x\in K$. 

Let $\{K_i\}_{i\in I}$ be a stratification of $K$ by subanalytic submanifolds. 

For every $i\in I$, consider the following  subset in the normal bundle of $K_i$ 
$$W_i:=\{(x,v)\in NK_i: ||v||<\eta(x)\}.$$
Let $V_i$ be a neighbourhood of $K_i$ inside $NK_i$, whose existence follows the Definable Tubular Neighborhood Theorem (see Theorem 6.11 in \cite{Coste:1999}), such that $\pi|_{V_i}$ is a diffeomorphism and such that $\pi(V_i)$ is a subanalytic neighbourhood of $K_i$, where $\pi:NK_i\to X$ is defined by $\pi(x,v)=x+v$.

Define $\calU(K,\eta):=\cup_{i\in I}\pi(V_i\cap W_i)$. This is a globally subanalytic neighbourhood of $K$. By the definition of $\eta$, we have that the closure of $\calU(K,\eta)\cap X$ is contained in  $L_2\cup...\cup L_k$.

We define $C_1$ as $L_1\setminus \calU(K,\eta)$. This is a closed set since it coincides with $\overline{L}_1\setminus \calU(K,\eta)$. Moreover $\{\mathring{C}_1,U_2, ..., U_k\}$ covers $X$. 
\end{proof}

\subsection{The Mayer-Vietoris Exact Sequence}\label{sec:MV}
\begin{theo}\label{theo:MayerVietoris}
Let $(X,x_0, d_X)$ be a  metric  germ, $Y_1$, $Y_2$  subanalytic subgerms and $\{U, V\}$ either a closed or an open a $b$-cover of $(Y_1,Y_2)$. The single complex associated with the Mayer-Vietoris double complex 
$$MDC^b_{\bullet}(U\cap V, Y_2)_{rel X}\to MDC^b_{\bullet}(U, Y_2)_{rel X}\oplus MDC^b_{\bullet}(V, Y_2)_{rel X}$$ 
is quasi-isomorphic to $MDC^b_{\bullet}(Y_1,Y_2)_{rel X}$. As a consequence there is a Mayer-Vietoris long exact sequence as follows:

\begin{align}\label{long exact Mayer-Vietoris sequence}
\begin{split}
... &\rightarrow  MDH_n^b(U \cap V, Y_2)_{rel X}    \rightarrow  MDH_n^b(U,Y_2)_{rel X} \oplus MDH_n^b(V,Y_2)_{rel X}   \rightarrow MDH_n^b(Y_1,Y_2)_{rel X}         \rightarrow \\
& \rightarrow  MDH_{n-1}^b(U \cap V,Y_2)_{rel X}     \rightarrow ...
\end{split}
\end{align}
Note that we have omitted the coefficient group $A$ in the notation for brevity.
\end{theo}
\begin{proof} 
We omit the coefficient group $A$ in the notation for brevity. 


We have the following short exact sequence, where $\alpha(\sigma, \tau) := \sigma - \tau$ is extended linearly:
\begin{equation}\label{auxiliary short exact sequence Mayer-Vietoros}
0\to \text{Ker}(\alpha)\to MDC^{b}_{\bullet}(U, Y_2)_{rel X}    \oplus MDC^{b}_{\bullet}(V, Y_2)_{rel X}    \xrightarrow{\alpha} MDC^{b}_{\bullet}(Y_1,Y_2)_{rel X}    \to 0
\end{equation}

If the $b$-cover is closed, surjectivity follows from the fact that (\ref{eq:smallqisMV}) is an isomorphism. If the cover $\{U, V\}$ is formed by open subsets, use Proposition~\ref{prop:closed refinement} to produce a closed subanalytic cover $\calC$ refining $\{U, V\}$ and use the isomorphism
$$
MDC^{b}_{\bullet}(Y_1,Y_2)_{rel X} \cong MDC^{b, \mathcal{C}}_{\bullet}(Y_1,Y_2)_{rel X}.
$$

As a consequence, the single complex associated with the double complex $$d:\text{Ker}(\alpha)\to MDC^{b}_{\bullet}(U, Y_2)_{rel X}\oplus MDC^{b}_{\bullet}(V, Y_2)_{rel X}$$ is quasi-isomorphic to $MDC^{b}_{\bullet}(Y_1,Y_2)_{rel X}$.

Let $\{ \hat{U}, \hat{V} \}$ be a $b$-extension of $\{U, V\}$. Consider the analogue of short exact sequence to~(\ref{auxiliary short exact sequence Mayer-Vietoros}).
\begin{equation}\label{auxiliary2}
0\to \text{Ker}(\hat\alpha)\to MDC^{b}_{\bullet}(\hat{U}, Y_2)_{rel X}    \oplus MDC^{b}_{\bullet}(\hat{V}, Y_2)_{rel X}    \xrightarrow{\alpha} MDC^{b}_{\bullet}(Y_1,Y_2)_{rel X}    \to 0.
\end{equation}
The inclusions $U \hookrightarrow \hat{U}$ and $V \hookrightarrow \hat{V}$ together induce a morphism 
$$\iota^{U,V}_* MDC^{b}_{\bullet}(U, Y_2)_{rel X}\oplus MDC^{b}_{\bullet}(V, Y_2)_{rel X} \to MDC^{b}_{\bullet}(\hat{U}, Y_2)_{rel X}\oplus MDC^{b}_{\bullet}(\hat{V}, Y_2)_{rel X}$$
that restricts to a morphism $\rho:\text{Ker}(\alpha) \to \text{Ker}(\hat{\alpha})$, and that together with the identity on $MDC^{b}_{\bullet}(Y_1,Y_2)_{rel X}$ produces a morphism from the short exact sequence (\ref{auxiliary short exact sequence Mayer-Vietoros}) to the short exact sequence (\ref{auxiliary2}). Since $\iota^{U,V}_*$ is a quasi-isomorphism by definition of $b$-cover, we conclude that $\rho$ is also a quasi-isomorphism.   . 

The morphism $\rho$ admits the following factorization:
\[
\text{Ker}(\alpha) \xrightarrow{f} MDC^{b}(\hat{U} \cap \hat{V}, Y_2)_{rel X} \xrightarrow{g} \text{Ker}(\hat{\alpha})
\]
where $g(\hat{\sigma}):= (\hat{\sigma}, \hat{\sigma})$ is extended linearly
and $f$ is defined as follows:


Let $([\sum_{i\in I}a_i\sigma_i],[\sum_{j\in J}b_i\psi_j])$ be an element of $MDC^{b}_{\bullet}(U, Y_2)_{rel X}\oplus MDC^{b}_{\bullet}(V, Y_2)_{rel X}$ such that $[\sum_{i\in I}a_i\sigma_i]+[\sum_{j\in J}b_i\psi_j]=0$ in $MDC^{b}_{\bullet}(Y_1,Y_2)_{rel X}$. After replacing the representatives by the ones obtained by sequences of $\to_\infty$-equivalences as in Lemma~\ref{lema:banulacion}, consider splittings $I=I_0\cup I_1\cup...,I_r$, $J=J_0\cup J_1\cup...,J_r$ as above, which satisfy that 

\begin{enumerate}
 \item $[\sigma_i]\in \text{Ker} (MDC_{\bullet}^b(U)_{rel X}\to MDC_{\bullet}^b(U,Y_2)_{rel X})$ for any $i\in I_0$, 
 \item $[\psi_j]\in \text{Ker} (MDC_{\bullet}^b(V)_{rel X}\to MDC_{\bullet}^b(V, Y_2)_{rel X})$ for any $j\in J_0$,
 \item $\sigma_i\sim_b \psi_j$ if $i\in I_k$ and $j\in J_k$ for a given $k\geq 1$,
\end{enumerate}
%

%

and that for any $k\geq 1$ we have 
$$\sum_{i\in I_k}a_i+\sum_{j\in J_k}b_k=0.$$
If $I_k$ and $J_k$ are non-empty, there is a $\tau \in MDC^{b}_{\bullet}(V,Y_2)_{rel X}$ in the same $b$-equivalence class as $\sigma_i$ for any $i \in I_k$. Observe that any l.v.a. simplex $b$-equivalent to a l.v.a. simplex in $V$ is contained in $\hat{V}$ by Lemma~\ref{lem:inclusionsimplex}, so $\sigma_i \in MDC^{b}_{\bullet}(\hat{U} \cap \hat{V},Y_2)_{rel X}$. We define 
\[
f(\sum_{i\in I}a_i[\sigma_i],\sum_{j\in J}b_i[\psi_j]) := \sum_{i \in I \setminus I_0} a_i [\sigma_i] = \sum_{j \in J \setminus J_0} b_j [\psi_j].
\]

We claim that $f$ is an isomorphism in the derived category. Indeed, since we have the factorization $\rho=g\comp f$ and $\rho$ is an quasi-isomorphism, we conclude that $f$ has a left inverse in the derived category. Let 
$$\theta:MDC^{b}_{\bullet}(U\cap V, Y_2)_{rel X}\to \text{ker}(\alpha)$$ be the natural inclusion of complexes. Let $\iota^{U\cap V}$ denote the inclusion $U \cap V \hookrightarrow \hat{U} \cap \hat{V}$. Then in the derived category the morphism
$$\iota^{U\cap V}_*:MDC^{b}(U\cap V,Y_2)_{rel X}\to MDC^{b}(\hat{U}\cap \hat{V},Y_2)_{rel X},$$
which is an isomorphism by definition of $b$-cover, equals the composition $f\comp\theta$. Therefore $f$ admits a right inverse in the derived category and the claim is proved.

Using that $f$ is an isomorphism in the derived category, and the isomorphism $\iota^{U\cap V}_*$ we conclude that $\theta$ is an isomorphism in the derived category. So in the derived category, the single complex associated with the double complex 
$$MDC^b_{\bullet}(U\cap V,Y_2)_{rel X} \to MDC^{b}_{\bullet}(U, Y_2)_{rel X}\oplus MDC^{b}_{\bullet}(V, Y_2)_{rel X}$$
is isomorphic to the double complex associated with $d$ which is isomorphic to the complex $MDC^b_{\bullet}(Y_1,Y_2)_{rel X}$.

\end{proof}

As a consequence we obtain the Excision Theorem.
\begin{cor}
\label{cor:excision}
Let $(X,x_0,d_X)$ be a metric  germ. Let $U\subset X\setminus\{x_0\}$ and $K\setminus \{x_0\}\subset U$ such that $\{U, X \setminus K \}$ is either an open or closed $b$-cover of $(X,U)$. Then the inclusion induces a quasi-isomorphism $MDC_{\bullet}^b(X\setminus K, U ;A)_{rel X} \to  MDC_\bullet^b(X, U;A)$. As a consequence for each $n$ we have an isomorphism
\begin{equation*}
MDH_n^b(X\setminus K, U ;A)_{rel X} \stackrel{\simeq }{\to}  MDH_n^b(X, U;A)
\end{equation*}
\end{cor}
\begin{proof}
Apply Theorem~\ref{theo:MayerVietoris} to the $b$-cover $\{U, X \setminus K \}$. 
\end{proof}

\subsection{The \texorpdfstring{\v{C}}{C}ech homology complexes}\label{sec:cech}
Let $Y_1,Y_2$ be subanalytic germs of a metric subanalytic germ $(X,x_0,d_X)$. 
Let $\calU=\{U_i\}_{i\in \{1,...,r\}}$ be a finite open or closed subanalytic cover of $Y_1$. Denote by $U_{i_1,...,i_r}$ the intersection $U_{i_1}\cap...\cap U_{i_r}$. The \v{C}ech double complex of $b$-MD homology of a pair $(Y_1,Y_2)$ associated with $\calU$ with coefficients in $A$ is defined by $$MDC^b(\calU,Y_1,Y_2;A)_{p,q}:=\bigoplus_{1\leq i_0<...<i_p\leq r}MDC^b_q(U_{i_0,...,i_p},Y_2;A)_{rel X},$$
with vertical differential equal to the $b$-MD differential and horizontal differential the usual \v{C}ech homology differential: 
$$MDC^b_q(U_{i_0,...,i_p},Y_2;A)_{rel X}\to \oplus_{k=0}^p MDC^b_q(U_{i_0,...,\hat{i}_k,...,i_p}, Y_2;A)_{rel X}$$
$$[\sigma]\mapsto \sum_{k=0}^p(-1)^k j^b_{i_k}([\sigma]),$$
where $j^b_{i_k}$ is the $b$-MD chain map associated to the inclusion $U_{i_0,...,i_p}\subset U_{i_0,...,\hat{i}_k,...,i_p}$. 

\begin{theo}
\label{theo:Cech}
Let $Y_1,Y_2$ be subanalytic subgerms of a metric germ $(X,x_0,d_X)$.
{If for any two disjoint finite subsets $I,J\subset \{1,...,r\}$ we have that $\{(\cap_{j\in J}U_j)\cap U_i\}_{i\in I}$ is a $b$-cover of $(\cup_{i\in I} U_i\cap (\cap_{j\in J}U_j),Y_2)$} and $\cup_{i=1}^r U_i=Y_1$, then the single complex associated with the \v{C}ech complex
$MDC^b_{\bullet,\bullet}(\calU,Y_1,Y_2;A)$ is quasi-isomorphic to $MDC^b_{\bullet}(Y_1,Y_2;A)_{rel X}$. Consequently there is a \v{C}ech spectral sequence abutting to $MDH^b_{*}(Y_1,Y_2;A)_{rel X}$ with $E^1$ page 
$$E[b]^1_{p,q}:=\bigoplus_{1\leq i_0<...<i_p\leq r}MDH^b_q(U_{i_0,...,i_p}, Y_2;A)_{rel X}.$$
\end{theo}
\proof
The case of a cover of 2 subsets is exactly Theorem~\ref{theo:MayerVietoris}. The general case runs by induction on the number of subsets, applying Mayer-Vietoris for the decomposition $V \cup U_r$ with $V:= U_1\cup...\cup U_{r-1}$ and the induction step for the decompositions $U_1\cup...\cup U_{r-1}$ and $(U_1\cap U_r)\cup...\cup (U_{r-1}\cap U_r)$: 

Let $\tilde{A}_\bullet$ and $\tilde{B}_\bullet$ denote the single complexes associated with the \v{C}ech complexes
$$
MDC^b_{\bullet,\bullet}(\{U_1 \cap U_r,\dots, U_{r-1}\cap U_r\},V \cap U_r,Y_2;A) \text{ and } MDC^b_{\bullet,\bullet}(\{U_1,\dots, U_{r-1}\},V,Y_2;A),
$$
respectively. By induction hypothesis, we get that $\tilde{A}_\bullet$ and $\tilde{B}_\bullet$ are quasi-isomorphic to 
$$
A_\bullet := MDC_\bullet^b(V \cap U_r,Y_2;A)_{rel X} \text{ and } B_\bullet := MDC_\bullet^b(V,Y_2;A)_{rel X},
$$
respectively. 

Let $S\bullet(D_{\bullet, \bullet})$ denote the single complex associated with a double complex $D_{\bullet, \bullet}$. Since we have quasi-isomorphisms $\tilde{A}_\bullet \to A_\bullet$ and $\tilde{B}_\bullet \to B_\bullet$, the morphism 
$$S_\bullet(\tilde{A}_\bullet \to MDC_\bullet(U_r, Y_2;A)_{rel X} \oplus \tilde{B}_\bullet )\to S_\bullet(A_\bullet \to MDC_\bullet(U_r, Y_2;A)_{rel X} \oplus B_\bullet )$$ 
is a quasi-isomorphism. 

By Mayer Vietoris Theorem (Theorem~\ref{theo:MayerVietoris}) applied to the cover $\{U_r, V\}$ we obtain a quasi-isomorphism 
$$S_\bullet(A_\bullet \to MDC_\bullet(U_r, Y_2;A)_{rel X} \oplus B_\bullet)\to MDC_\bullet(Y_1, Y_2;A)_{rel X}.$$

Now the statement follows from the fact that 
$$S_\bullet(\tilde{A}_\bullet \to MDC_\bullet(U_r, Y_2;A)_{rel X} \oplus \tilde{B}_\bullet )$$
is isomorphic to the single complex of the \v{C}ech complex $MDC^b_{\bullet,\bullet}(\calU,Y_1,Y_2;A)$.

\endproof

\begin{definition}
\label{def:nerve}
The nerve of the cover $\calU=\{U_i\}_{i\in \{1,...,r\}}$ is the simplicial complex which assigns a $p$-simplex to each non-empty intersection $U_{i_0,...,i_p}$, and identifies faces according to the inclusions $U_{i_0,...,i_p}\subset U_{i_0,...,\hat{i}_k,...,i_p}$.
\end{definition}

\begin{cor}
\label{cor:nerve}
In the setting of the last theorem, if $Y_1=X$ and $Y_2=\emptyset$ and for any finite set of indexes $U_{i_0,...,i_p}$ is either empty or has the $b$-MD homology of a point, the $b$-MD homology of $X$ coincides with the ordinary homology of the nerve of the cover with coefficients in $A$. 
\end{cor}
\proof
In the spectral sequence of Theorem~\ref{theo:Cech} we have $E[b]^1_{p,q}=0$ if $q>0$ and $E[b]^1_{p,0}=\bigoplus_{1\leq i_0<...<i_p\leq r}MDH^b_0(U_{i_0,...,i_p},A)$, where $MDH^b_0(U_{i_0,...,i_p};A)\cong A$ if and only if $U_{i_0,...,i_p}$ is not empty.
\endproof

\section{Moderately Discontinuous Homology in degree \texorpdfstring{$0$}{0}}

\begin{definition}
\label{def:b-eqcomponents} Let $(X,x_0)$ be a metric subanalytic germ. Two connected components $X^1$ and $X^2$ of $X\setminus\{x_0\}$ are $b$-equivalent if there exist two l.v.a. $0$-simplices $\sigma_i:\hat{\Delta}_0\to (X_i,x_0)$ which are $b$-equivalent. The equivalence classes are called $b$-connected components of $X$. The $\infty$-connected components are the usual connected components of $X\setminus\{x_0\}$.
\end{definition}

\begin{prop}
\label{prop:0-homology}
The $b$-moderately discontinuous homology $MDH^b_0(X;A)$ at degree $0$ is isomorphic to $A^{r(b,X)}$, where $r(b,X)$ is the number of $b$-connected components of $X$. A basis is given by the choice of a $0$-simplex in each $b$-connected component. For $b_1, b_2 \in (0, \infty]$, $b_1 \geq b_2$, the homomorphism $h_0^{b_1, b_2}$ is the projection that sends a base element $\alpha$ of $A^{r(b_1,X)}$ onto the base element of $A^{r(b_2,X)}$ that represents the $b_2$-connected component $\alpha$ lies in.
\end{prop}
\proof
Let $L:=X\cap\SSS_\epsilon$ be the link of $X$ (where $\epsilon>0$ is small enough). Let $\theta:C^1_L\to X$ be a subanalytic homeomorphism preserving the distance to the origin (this exists by Remark~\ref{rem:germ1}). Let $\tau:\Delta_n\to L$ be a subanalytic map. The {\em straight} $n$-simplex with respect to $\theta$ associated with $\tau$ is defined to be the map germ $\sigma:\hat{\Delta}_n\to X$ given by $\sigma(tz,t):=\theta(\tau(z),t)$. 

Let $x_1,x_2$ be two points in the same connected component of $L$. Then there exists a subanalytic path $\gamma:[0,1]\to L$ joining $x_1$ and $x_2$. The boundary operator ``$\partial$'' applied to the straight simplex associated with $\gamma$ is the difference of the straight simplices associated with $x_i$. So, we conclude that two straight $0$-simplices in the same connected component of $X\setminus\{x_0\}$ are $b$-homologous for any $b$. 

Let $\sigma:\hat{\Delta}_0=[0,1)\to (X,x_0)$ be any $0$-simplex. Up to reparametrization (see Remark~\ref{rem:lvareduction}) we may assume that $||\sigma(t)||=t$. We can express the restriction  $\theta^{-1}\comp\sigma|_{\Delta_0\times (0,1)}$ as a pair $\theta^{-1}\comp\sigma|_{\Delta_0\times (0,1)}(t)=(\gamma(t),t)$, where $\gamma:(0,1)\to L$ is the germ at $0$ of a subanalytic path. We may choose the radius $\epsilon$ defining the link $L$ small enough so that $\epsilon$ is in the domain of definition of the germ $\sigma$, and hence of $\gamma$. The map $\tau:\hat{\Delta}_1 \to (X, x_0)$, where $\tau(ts,t):=\theta(\gamma(\epsilon+s(t-\epsilon)),t)$, defines a $1$-simplex whose boundary shows that $\sigma$ is $b$-homologous to a straight simplex.

We have proven that all $0$-simplices lying in the same connected component of $X\setminus\{x_0\}$ are $b$-homologous for any $b$. 

After this the proof is obvious.
\endproof

\section{The \texorpdfstring{$\infty$}{infinity}-MD Homology: comparison with the homology of the link}
\label{sec:link}

Let $(X,Y,\{x_0\},d_X)$ be a pair of closed metric subanalytic germs in $\mathbb{R}^n$. 
By Remark~\ref{rem:triang} there is a finite subanalytic triangulation $\alpha:|K|\to X\cap B_\epsilon$ of a representative $X\cap B_\epsilon$, which is compatible with $Y$ and $x_0$. By choosing $\epsilon$ sufficiently small and intersecting with $\SSS_\epsilon$ we obtain a subanalytic triangulation $\beta:|L|\to X\cap \SSS_\epsilon$ compatible with $Y\cap\SSS_\epsilon$ such that $(K,\alpha)$ is the cone over $(L,\beta)$. In other words: there exists a pair of simplicial
complexes $(L_1,L_2)$ and a subanalytic homeomorphism
$$h:(C(|L_1|),C(|L_2|))\to (X,Y,\{x_0\})\cap B_\epsilon$$
from the cones of the geometric realizations to the representative $(X,Y,\{x_0\})\cap B_\epsilon$. By the reparametrization trick of Remark~\ref{rem:lvareduction} we may assume that $||h(tx,t)||=t$.

Denote by $C^{\mathrm{Simp}}_{\bullet}(L_1,L_2;A)$ the simplicial homology complex for the pair $(L_1,L_2)$ with coefficients in $A$. The homeomorphism $h$ induces a morphism of complexes 
\begin{equation}
\label{eq:link}
c:C^{\mathrm{Simp}}_{\bullet}(L_1,L_2;A)\to MDC_{\bullet}^\infty(X,Y;A).
\end{equation}

\begin{theo}
\label{theo:link}
The morphism~(\ref{eq:link}) is a quasi-isomorphism. As a consequence we have an isomorphism between the singular homology  $H_{*}(X\setminus\{x_0\},Y\setminus\{x_0\};A)$ and $MDH^{\infty}_{*}(X,Y,x_0;A)$.
\end{theo}
\proof
By using the relative homology sequence and the $5$-lemma we reduce to the absolute case $Y=\emptyset$. The singular homology $H_{*}(X\setminus\{x_0\};A)$ is isomorphic to the singular homology of the link, by homotopy invariance, and the later is isomorphic with the simplicial homology of $L_1$.

A simplex of $L_1$ is called {\em maximal} if it is not strictly contained in another simplex. The collection $\{Z_i\}_{i\in I}$ of maximal simplices forms a closed cover of $|L_1|$ such that any finite intersection is a simplex, and hence, contractible. Then the simplicial homology of $L_1$ coincides with the homology of the nerve of the cover. 

The collection $\{h(C(Z_i))\}_{i\in I}$ is a closed subanalytic cover. Any finite intersection $\cap_{i\in J}h(C(Z_i))$ is of the form $h(C(T))$ where $T$ is a simplex in $L_1$. An immediate application of Assertion~(1) of Theorem~\ref{theo:homotopyinvariance} shows that $h(C(T))$ has the $\infty$-MD homology of a point. Since any closed subanalytic cover is an $\infty$-cover, by Corollary~\ref{cor:nerve} the homology $MDH^{\infty}_{*}(X\setminus\{x_0\};A)$ coincides with the homology with coefficients in $A$ of the nerve of the cover. This concludes the proof. 
\endproof

This Theorem extends to the case of non-closed subanalytic germs:

\begin{theo}
\label{theo:link_no_cerrado}
Let $(X,Y,\{x_0\},d_X)$ be a pair of metric subanalytic germs in $\RR^n$. We have an isomorphism between the singular homology  $H_{*}(X\setminus\{x_0\},Y\setminus\{x_0\};A)$ and $MDH^{\infty}_{*}(X,Y,x_0;A)$.
\end{theo}
\proof
As in the previous proof we reduce to the absolute case $Y=\emptyset$. Let $\partial X:=\overline{X}\setminus X$. Fix a positive $\eta>0$, for any $k\in\NN$ define $X_k:=X\setminus \calH_{k,\eta}(\partial X; X)$ (recall that $\calH_{k,\eta}(\partial X; X)$ denotes the $k$-horn neighborhood of $\partial X$ in $X$. Each $X_k$ is a subanalytic metric subgerm of $X$ which is closed in $\RR^n$. We have the sequence of inclusions $X_1\subset X_2\subset ...\subset X_k\subset ...$, and the equality $X=\cup_{k\in\NN} X_k$. Hence the complex $MDC^{\infty}_{*}(X,x_0;A)$ is the direct limit of the complexes $MDC^{\infty}_{*}(X_k,x_0;A)$, and since homology commutes with direct limit the same happens for the homology. This reduces the result to Theorem~\ref{theo:link}.
\endproof

\section{Finiteness properties}
\label{sec:finitegeneration}

\subsection{Finite generation}
The idea to prove finite generation is to compare $MD$-homology with the usual homology of an increasing sequence of neighborhoods. The inspiration for this owes to conversations of the authors with L. Birbrair and A. Parusinski. 

\begin{theo}
\label{theo:finiteness}
Let $(X,O,d)$ be a metric subanalytic germ such that $d$ is either the inner metric $d_{inn}$ or the outer metric $d_{out}$. Then for any ring $A$, any  $b\in\BB$ and any $k$ the $A$-module $MDH^{b}_k(X,O,d;A)$ is finitely generated over $A$. 
\end{theo}

It is enough to prove the Theorem for the outer metric, since due to (\cite{reembedding}, Theorem 2.1) the germ $(X,O,d_{inn})$ is bi-Lipschitz subanalitically homeomorphic
to a metric subanalytic subgerm with the outer metric.

To prove it for the outer metric some preparation is needed. Let us start introducing an auxiliary complex. 

\begin{definition}
\label{def:bclosechains}
Let $d_e$ denote the euclidean metric in $\RR^n$. Let $(X,O)$ be a metric subanalytic germ embedded in $\RR^n$. We denote by $\calN_bMDC^{pre,\infty}_k(X;A)$ the subgroup of $MDC^{pre,\infty}_k(\RR^n,O;A)$ spanned by the set of simplexes $\sigma:\hat{\Delta}_k\to (\RR^n,O)$ such that 
$$\lim_{t\to 0}\frac{max\{d_e(\sigma(tx,t),X):x\in\Delta_n\}}{t^b}=0.$$

The minimal subcomplex $\calN_{b}MDC^{b}_\bullet(X;A)$ of $MDC^{b}_\bullet(\RR^n,0;A)$ which is spanned by the groups $\calN_bMDH^{pre,\infty}_k(X;A)$ is called the complex of chains $b$-close to $X$.
\end{definition}

Consider the germ with the outer metric $(X,O,d_{out})$. Let us see that there is a well defined morphism of complexes
\begin{equation}\label{eq:rho}\rho:\calN_{b}MDC^{b}_\bullet(X;A)\to MDC^{b}_\bullet(X,d_{out};A).\end{equation}
Note that $d_{out}=d_e|_X$. 

Assume $X$ to be closed. Let $B$ be a small closed ball centered at the origin so that $B\cap X$ is compact. Apply Lemma~\ref{lem:retracciondiscontinua} to the pair $B\supset X$. Let $\calS:=\{S_i\}_{i\in I}$ and $\{f_i\}_{i\in I}$ be the partition and the subanalytic maps predicted in the Lemma. Then $\{\overline{S}_i\}_{i\in I}$ is a closed cover of $B$. Let $U$ be the union of the interiors of the sets $S_i$. Then $U$ is a dense subanalytic subset of $B$. By Proposition~\ref{prop:smallqis}, for any $b<\infty$ we have an isomorphism 
$$MDC^{b,\calS\cap U}_\bullet(B,d_e;A)\to MDC^{b}_\bullet(B,d_e;A).$$
This isomorphism restricts to an isomorphism
$$\calN_bMDC^{b,\calS\cap U}_\bullet(X;A)\to \calN_{b}MDC^{b}_k(X;A),$$
where $\calN_bMDC^{b,\calS\cap U}_\bullet(X;A)$ is the complex defined in analogy with Definition~\ref{def:smallchain}: define $\calN_bMDC^{pre,\infty,\calS\cap U}_\bullet(X;A)$ the complex analogous to $\calN_bMDC^{pre,\infty}_k(X;A)$, but containing only small chains with respect to $U$ and to the cover $\calZ$. According with Remark~\ref{rem:checkindependence small chains}, in order to define a morphism  $\rho$ as in (\ref{eq:rho}) it is enough to define it from $\calN_bMDC^{pre,\infty,\calS\cap U}_\bullet(X;A)$ and prove independence with respect to homological subdivisions and the $b$-equivalence relation. 

Given any simplex $\sigma\in \calN_bMDC^{pre,\infty,\calS\cap U}_\bullet(X;A)$ we choose an $i\in I$ such that the image of $\sigma$ is contained in $S_i$, and define $\rho(\sigma):=f_i\comp\sigma$. Suppose that we have $\sigma\sim_b\sigma'$ with $\sigma'\subset S_j$. Let's check that $f_i\comp(\sigma)\sim_b f_j\comp(\sigma')$. The triangle inequality gives
$$d(f_i(\sigma(tx,t),f_j(\sigma'(tx,t)))\leq d(f_i(\sigma(tx,t)),\sigma(tx,t))+d(\sigma'(tx,t),\sigma(tx,t))+d(f_j(\sigma'(tx,t)),\sigma'(tx,t)).$$
The middle term decreases faster than $t^b$ because $\sigma\sim_b\sigma'$. By Inequality~\ref{eq:closest} the first term is bounded by $d(\sigma(tx,t),X)$, which decreases faster than $t^b$ for being $\sigma$ contained in $\calN_bMDC^{pre,\infty,\calS}_\bullet(X;A)$. The third term is bounded analogously. This shows independence on the $\sim_b$-equivalence relation. The independence on homological subdivision is immediate, and so, the morphism $\rho$ is well defined.

\begin{prop}
\label{prop:rhoiso}
Let $(X,O,d_{out})$ be a closed metric subanalytic germ embedded in $\RR^n$ with the outer metric.
The morphism of complexes 
$$\rho:\calN_{b}MDC^{b}_k(X;A)\to MDC^{b}_\bullet(X,O,d_{out};A)$$
defined above is an isomorphism.
\end{prop}
\proof
The  inclusion  $\iota:MDC^{b}_\bullet(X,d_{out};A)\to \calN_{b}MDC^{b}_k(X;A)$ is a right inverse of $\rho$. Furthermore, by definition of $\calN_bMDC^{pre,\infty,\calZ}_\bullet(X;A)$ we have that $\iota(\rho(\sigma))\sim_b\sigma$; this shows that $\iota$ is also a left inverse. 
\endproof

Now, in order to prove finite generation of $MDH^b_k(X,d_{out};A)$, it is enough to prove finite generation of the homology groups of the complex $\calN_{b}MDC^{b}_k(X;A)$. With this purpose we study this complex further.

For the next lemma recall the definition of spherical horns (Definition~\ref{def:sphhorn}).

\begin{lema}
\label{lem:directlimit1}
For any $b'>b$ the complex $MDC^{pre,\infty}(S\calH_{b',\eta}(X))$ is a subcomplex of $\calN_bMDC^{pre,\infty}_k(X;A)$. Moreover $\calN_bMDC^{pre,\infty}_k(X;A)$ is the union of the subcomplexes $MDH^{\infty,pre}(S\calH_{b',\eta}(X))$ for any $b'>b$.
\end{lema}
\proof
The first part is obvious, since any l.v.a simplex in $S\calH_{b,\eta}(X))$ approaches $X$ at speed at least $b'$. The second part is also obvious: let $\sigma:\hat{\Delta}_n\to\RR^n$ be a l.v.a simplex  approaching  $X$ faster than $t^b$. Then we can write $max\{d_e(\sigma(tx,t):x\in\Delta_n\}=Ct^{b''}+o(t^{b''})$ for a certain $b''>b$. Choosing $b<b'<b''$ we obtain that $\sigma$ belongs to $MDH^{\infty,pre}(S\calH_{b',\eta}(X))$.
\endproof

\begin{lema}
\label{lem:directlimit2}
 For any fixed $\eta>0$, the complex $\calN_{b}MDC^{b}_k(X;A)$ is the direct limit of the complexes $MDC^{b'}(S\calH_{b',\eta}(X),d_{out})$ when $b'>b$.
\end{lema}
\proof
There are obvious morphisms of complexes 
$$MDH^{b'}(S\calH_{b',\eta}(X),d_e)\to \calN_{b}MDC^{b}_k(X;A),$$ 
which induce a morphism from the direct limit to $\calN_{b}MDC^{b}_k(X;A)$. By the previous lemma such morphism is an epimorphism. 

To prove the injectivity, let $z$ be an element in $MDC^{pre,\infty}(S\calH_{b',\eta}(X))$ such that its image in $\calN_{b}MDC^{b}_k(X;A)$ vanishes. By Lemma~\ref{lema:banulacion} there is a sequence of immediate equivalences $z=z_0\to_\infty ...\to_\infty z_r$ such that we have that $z_r$ splits as finite sum of chains of the form $\sum_{j\in J}a_j\sigma_j$ with $\sum_{j\in J}a_j=0$ and $\sigma_j\sim_b\sigma_{j'}$ for any $j,j'\in J$. After this  the following argument suffices: let $\sigma$ and $\sigma'$ be two $n$-simplexes which represent elements in the direct limit and which are $\sim_b$ equivalent. Then $max\{d_e(\sigma(tx,t),\sigma'(tx,t):x\in\Delta_n\}=Ct^{b'}+o(t^{b'})$ for a $b'>b$. Then $\sigma\sim_{b''}\sigma'$ for any $b<b''<b'$ and hence $\sigma$ and $\sigma'$ represent the same element in $MDH^{b''}(S\calH_{b',\eta}(X),d_e)$.
\endproof

Since taking homology commutes with direct limits, combining the previous lemma with Proposition~\ref{prop:rhoiso} we obtain:

\begin{theo}
\label{theo:iso2}
Let $(X,O,d_{out})$ be a closed metric subanalytic germ embedded in $\RR^n$ with the outer metric. The $b$-MD Homology $MDH^{b}_\bullet(X,d_{out};A)$ is isomorphic to the direct limit $$\lim_{\stackrel{\longrightarrow}{b'>b}}MDH^{b'}_\bullet(S\calH_{b',\eta}(X),d_e;A).$$
\end{theo}

\begin{prop}
\label{prop:comparaciondificil}
There is an isomorphism $$MDH^{b'}_\bullet(S\calH_{b',\eta}(X),d_e;A)\cong MDH^{\infty}_\bullet(S\calH_{b',\eta}(X),d_e;A).$$
\end{prop}
\begin{proof}
Let $L_X$ be the link of $X$ and let $h:C(L_X)\to X$  the subanalytic homeomorphism realizing the conical structure (see  Remark~\ref{rem:germ1}). Recall that we can assume that $||h(tx,t)||=t$. Choose $\{x_k\}_{k\in\NN}$ a dense denumerable set of points in $X$. For each $k\in\NN$ define the arc $\gamma_k(t):=h(tx_k,t)$; denote by $\Gamma_k$ the image of the arc. Define the sequence $\eta_l:=(l/l+1)\eta$. For any $k\in\NN$ we consider the subsets $U_{k,l}:=\overline{S\calH_{b',\eta_l}(\Gamma_k)}$. Observe that, by the density of the collection  $\{x_k\}_{k\in\NN}$, we obtain a denumerable cover $\calU=\{U_{k,l}\}_{k,l\in\NN}$ of $S\calH_{b',\eta}(X)$ by closed subsets. 

We need the following Lemma, whose proof we postpone:

\begin{lema}
\label{lem:sphbcover}
Let $b'>1$. For any finite subset of pairs of indexes $\{(k_i,l_i)\}_{i\in I}$ the following are true:
\begin{enumerate}
 \item we have the equalities $MDH^{b'}_k(\cap_{i\in I}U_{k_i,l_i},d_{out};A)=MDH^{\infty}_k(\cap_{i\in I}U_{k_i,l_i};A)=0$ if $k>0$ and $MDH^{b'}_0(\cap_{i\in I}U_{k_i,l_i},d_{out};A)\cong MDH^{\infty}_0(\cap_{i\in I}U_{k_i,l_i};A)\cong A$, where the isomorphism is realized by the natural homomorphism connecting both homologies. 
 \item The collection $\{U_{k_i,l_i}\}_{i\in I}$ is a $b'$-cover of the union $Z_I:=\cup_{i\in I}U_{k_i,l_i}$, where for each $i\in I$ we can choose the subsets $\hat{U}_{k_i,l_i}$ according to  Definition~\ref{def:b-saturated}, equal to 
$\hat{U}_{k_i,l_i+1}\cap Z_I$.
\end{enumerate}
\end{lema}

We continue with the proof of Proposition~\ref{prop:comparaciondificil}. 

Choose a bijection $n\mapsto (k_n,l_n)$ from $\NN$ to $\NN\times\NN$. Define $Z_m:=\sum_{n=1}^mU_{k_n,l_m}$. We have $X=\cup_{m\in\NN}X_m$. Then the natural homomorphism of complexes 
$$MDC^{\infty}_\bullet(S\calH_{b',\eta}(X);A)\to MDC^{b'}_\bullet(S\calH_{b',\eta}(X),d_e;A)$$
is the direct limit, when $m$ increases, of the homomorphisms of complexes 
\begin{equation}
\label{eq:homom}
MDC^{\infty}_\bullet(Z_m;A)\to MDC^{b'}_\bullet(Z_m;A)
\end{equation}
Since taking homology commutes with direct limits it is enough to prove that each of the homomorphisms above is a quasi-isomorphism.

Because of Lemma~\ref{lem:sphbcover}, and Corollary~\ref{cor:nerve}, for any $m$, each of the homology complexes involved in the homomorphism~(\ref{eq:homom}) coincides with the homology of the nerve of the cover $\{U_{k_n,l_n}\}_{n\leq m}$, and the homomorphism is an isomorphism.  
\end{proof}

\proof[Proof of Lemma~\ref{lem:sphbcover}]
Assume the following claim: 

\textsc{Claim}: Let $(Y,O,d_{out})$ be a metric subanalytic germ with the outer metric which satisfies:
\begin{enumerate}
 \item For every sufficiently small $\epsilon>0$ the intersection $Y\cap\SSS_\epsilon$ is geodesically convex.
 \item There exists an arc $\gamma:[0,\epsilon)\to Y$ such that $||\gamma(t)||=t$; let $\Gamma$ be the image of the arc. The set $Y$ is contained in the spherical horn neighborhood $S\calH_{b',\eta'}(\Gamma)$ for a certain $\eta'>0$.
\end{enumerate}
Then $\Gamma$ is a metric deformation retract of $Y$. 

Observe that for any finite index set $I$ as in Lemma~\ref{lem:sphbcover}, the set $\cap_{i\in I}U_{k_i,l_i}$ satisfies the properties of the set $Y$ of the claim above: geodesically convexity follows because a finite intersection of geodesically convex sets is geodesically convex. One can take $\eta'$ for example equal to $3\eta$. The the claim implies part (1) of Lemma~\ref{lem:sphbcover} as a direct application of Theorem~\ref{theo:homotopyinvariance} and Proposition~\ref{prop:point}. 

For part (2) of the lemma let $I$ a finite index set as required. By definition of the sets $U_{k_i,l_i}$ and $\hat{U}_{k_i,l_i}$ the first condition for $b'$-cover is immediately satisfied. For the second condition choose a finite subset $J\subset I$. Then both $\cap_{i\in J}\hat{U}_{k_i,l_i}$ and $\cap_{i\in J}U_{k_i,l_i}$ are subanalytic subsets satisfying the properties of $Y$ in the claim, and moreover the arc $\gamma$ can be chosen to be the same for both of them. Then, by the conclusion of the claim we can choose the map $r_J:\cap_{i\in J}\hat{U}_{k_i,l_i}\to \cap_{i\in J}U_{k_i,l_i}$ needed in the second property required in the definition of $b'$ cover to be the retraction to $\Gamma$. 

Finally we prove the claim. Let $B$ be the tangent half-line of $\Gamma$ at the origin. Define 
$\phi:S\calH_{b',\eta'}(\Gamma)\to S\calH_{b',\eta'}(B)$
in such a way that the restriction $\phi_t:=\phi|_{S\calH_{b',\eta'}\cap \SSS_t}$ is a the restriction of a rotation of the sphere $\SSS_t$, and such that the family of rotations $\phi_t$ with $t\in [0,\epsilon)$ is smooth and converges to the identity when $t$ approaches $0$. Then $\phi$ is a bi-Lipschitz homeomorphism. It is obvious that the spherical neighborhood $S\calH_{b',\eta'}(B)$ admits $B$ as a metric subanalytic deformation retract, and such a deformation retraction can be taken leaving invariant the geodesics in $\SSS_t$ starting at $B\cap\SSS_t$ for any $t$. Conjugating by $\phi$ we obtain that $S\calH_{b',\eta'}(\Gamma)$ admits $\Gamma$ as a metric subanalytic deformation retract, by a deformation retraction which leaves the geodesics in $\SSS_t$ starting at $\Gamma\cap\SSS_t$ invariant. Because $Y$ is contained in  $S\calH_{b',\eta'}(\Gamma)$ and $Y\cap\SSS_t$ is geodesically convex, the metric deformation retraction leaves $Y$ invariant, and shows that $Y$ admits $\Gamma$ as a metric subanalytic deformation retract.  
\endproof

In the proof of the main result of this section we need an straightforward adaptation of the concept of cobordism to the $O$-minimal case. We record it in a definition to avoid misundertandings.

\begin{definition}
\label{def:subancobor}
Fix an $O$-minimal structure defined over $\RR$. A {\em $O$-minimal cobordism} is a triple $(X,A,B)$ of definable closed subsets such that $A$ and $B$ are closed disjoint definable subsets of the frontier $\partial X$. A $O$-minimal cobordism is {\em trivial} if $(X,A,B)$ is $O$-minimal homeomorphic to $(A\times [0,1],A\times\{0\},A\times\{1\})$. An $O$-minimal cobordism $(X,A,B)$ is {\em homologically trivial} if the inclusions $A\hookrightarrow X$ and $B\hookrightarrow X$ induce isomorphisms in homology. Given two $O$-minimal cobordisms $(X,A_1,B_1)$ and $(X',A'_1,B'_1)$, a {\em composition} of $(X,A_1,B_1)$ with $(X',A'_1,B'_1)$ on the right is an $O$-minimal cobordism $(X'',A'',B'')$ together with a decomposition $X''=X''_1\cup X''_2$ in two closed subsets with the following property: there exists two definable continuous mappings 
$$\varphi:(X,A_1)\to (X'',A''),$$
$$\varphi':(X',B'_1)\to (X'',B''),$$
such that $\varphi$ is a homeomorphism onto $X''_1$, $\varphi'$ is a homeomorphism onto $X''_2$, and we have the equalities $X''_1\cap X''_2=\varphi(B_1)=\varphi'(A'_1)$. An $O$-minimal cobordism is {\em invertible} if there are compositions on the right and on the left such that the result is a trivial $O$-minimal cobordism in both cases. 
\end{definition}

\begin{remark}
\label{rem:inverhomtriv}
An invertible cobordism is homologically trivial.
\end{remark}

Now we can prove the Finiteness Theorem.

\proof[Proof of Theorem~\ref{theo:finiteness}]
As we said at the beginning of the section it is enough to prove the Theorem for the outer metric, since due to (\cite{reembedding}, Theorem 2.1) the germ $(X,O,d_{inn})$ is bi-Lipschitz subanalytically homeomorphic
to a metric subanalytic subgerm with the outer metric. By Proposition~\ref{prop:b<1trivial} we can assume $b\geq 1$. By Corollary~\ref{cor:closure} we can assume $X$ to be closed.

For the outer metric, by Corollary~\ref{theo:iso2}, Proposition~\ref{prop:comparaciondificil}, we obtain that the group $MDH^{b}_k(X,O,d_{out};A)$ is  isomorphic to the direct limit 
\begin{equation}
\label{eq:dirlim1}
\lim_{\stackrel{\longrightarrow}{b'>b}}MDH^{\infty}_k(S\calH_{b',\eta}(X);A).
\end{equation}

By Theorem~\ref{theo:link_no_cerrado} we obtain that if $\epsilon_{b'}>0$ is small enough so that $S\calH_{b',\eta}(X)$ is the cone over $S\calH_{b',\eta}(X)\cap\SSS_{\epsilon_{b'}}$ then the group $MDH^{\infty}_k(S\calH_{b',\eta}(X);A)$ is  isomorphic to the singular homology $H_k(S\calH_{b',\eta}(X)\cap B_{\epsilon_{b'}}\setminus\{O\};A)$. 

Fix $\epsilon>0$. By Remark \ref{poly_bounded} the set $Z\subset \RR^n\times (b,\infty)$ defined by the formula $Z:=\{(x,b'):x\in S\calH_{b',\eta}\cap B_\epsilon\}$ is definable in the $O$-minimal structure $\RR^\RR_{an}$. By Hardt's Triviality Theorem there exist $b''>b$ such that the projection $Z|_{(b,b'')}\to (b,b'')$ is a trivial fibration, with a trivialization preserving the origin. 

We claim that for any $b_1,b_2'\in (b,b'')$ we have the isomorphism 
$$H_k(S\calH_{b_1',\eta}(X)\cap B_{\epsilon_{b'_1}}\setminus\{O\};A)\cong H_k(S\calH_{b_2',\eta}(X)\cap B_{\epsilon_{b'_2}}\setminus\{O\};A).$$ 
If the claim is true the proof is finished because the direct limit~(\ref{eq:dirlim1}) stabilizes.

The pairs $(S\calH_{b_i',\eta}(X)\cap B_{\epsilon_{b'_i}},\{O\})$ and $(S\calH_{b_i',\eta}(X)\cap B_{\delta_i},\{O\})$ are $O$-minimal homeomorphic for any $\delta_i\leq \epsilon_{b'_i}$. Hence, in order to prove the claim notice that we can choose $\epsilon_{b'_i}$ smaller than $\epsilon$ for $i=1,2$.

Let $\phi:S\calH_{b_1',\eta}\cap B_\epsilon\to S\calH_{b_2',\eta}\cap B_\epsilon$ be the $O$-minimal homeomorphism predicted by Hardt's Triviality. Choose decreasing sequences $\{\delta_{1,n}\}_{n\in\NN}$ and $\{\delta_{2,n}\}_{n\in\NN}$ converging to $0$ and so that the following inclusions are satisfied 
$$\phi(S\calH_{b_1',\eta}\cap B_{\delta_{1,n}})\subset S\calH_{b_2',\eta}\cap B_{\delta_{2,n}},$$
$$\phi(S\calH_{b_1',\eta}\cap B_{\delta_{1,n-1}})\supset S\calH_{b_2',\eta}\cap B_{\delta_{2,n}}.$$
By the conical structure the cobordism 
$$(\overline{S\calH_{b_i',\eta}\cap B_{\delta_i,{n-1}}\setminus S\calH_{b_i',\eta}\cap B_{\delta_{i,n}}},S\calH_{b_i',\eta}\cap \SSS_{\delta_i,{n-1}},S\calH_{b_i',\eta}\cap \SSS_{\delta_{i,n}})$$
is trivial for any $i=1,2$ and $n\in\NN$. Hence the cobordism 
$$(\overline{S\calH_{b_2',\eta}\cap B_{\delta_i,{n}}\setminus \phi(S\calH_{b_1',\eta}\cap B_{\delta_{i,n}})},S\calH_{b_2',\eta}\cap\SSS_{\delta_i,{n}},\phi(S\calH_{b_1',\eta}\cap\SSS_{\delta_{i,n}})$$ is invertible, and therefore homologically trivial. This implies the claim.
\endproof

\begin{remark}
\label{rem:etadoesnotmatter}
As consequence of the proof above is that the direct limit~(\ref{eq:dirlim1}) is independent of $\eta$. Since this may seem surprising we include another proof below. 
\end{remark}
\proof
Let $\eta_1<\eta_2$. The inclusion $S\calH_{b',\eta_1}(X)\subset S\calH_{b',\eta_2}(X)$ gives a homomorphism 
$$\lim_{\stackrel{\longrightarrow}{b'>b}}MDH^{\infty}_k(S\calH_{b',\eta_1}(X);A)\to \lim_{\stackrel{\longrightarrow}{b'>b}}MDH^{\infty}_k(S\calH_{b',\eta_2}(X);A).$$
On the other hand we have the inclusion $S\calH_{b',\eta_2}(X)\subset S\calH_{(b+b')/2,\eta_2}(X)$, which gives a homomorphism 
$$\lim_{\stackrel{\longrightarrow}{b'>b}}MDH^{\infty}_k(S\calH_{b',\eta_2}(X);A)\to \lim_{\stackrel{\longrightarrow}{b'>b}}MDH^{\infty}_k(S\calH_{b',\eta_1}(X);A).$$
Both homomorphisms are inverse to each other, since, because of the previous proof there is an interval $(b,b'')$ such that the homologies 
$MDH^{\infty}_k(S\calH_{b',\eta_1}(X);A)$ and $MDH^{\infty}_k(S\calH_{b',\eta_2}(X);A)$ are independent of $b'$.
\endproof

\subsection{Birbrair conjecture and finiteness of jumping rates}

L. Birbrair conjectured (oral communication) that if $(X,O,d_{out})$ is a metric subanalytic subset equipped with the outer metric, then the $b$-MD Homology of $(X,O,d_{out})$ coincides with the homology of a certain punctured horn neighborhood. The proof of Theorem~\ref{theo:finiteness} confirms this conjecture if one uses spherical horn neighborhoods.

\begin{cor}
\label{cor:birbrair}
Let $(X,O,d_{out})$ is a closed metric subanalytic subset equipped with the outer metric. Let $b\in\BB$ finite and $\eta>0$. There exists a $b'$ satisfying $b<b'$ such that the $b$-MD homology $MDH^{b}_\bullet(X,O,d_{out};A)$ is isomorphic both to the singular homology of the punctured $b''$-spherical horn neighborhood $S\calH_{b'',\eta}(X)\setminus\{O\}$, and to the singular homology of the punctured $b''$-horn neighborhood $\calH_{b'',\eta}(X)\setminus\{O\}$, for any $b''\in (b,b')$. 
\end{cor}
\proof
The assertion for spherical horn neighborhoods is an immediate consequence of the results in the previous section. The assertion for horn neighborhoods is proved as follows. Since we have the nesting $\calH_{b''',\eta}(X)\subset S\calH_{b'',\eta}(X)\subset  \calH_{b'',\eta}(X)$ for any $b'''>b''$, we have the isomorphism 
$$\lim_{\stackrel{\longrightarrow}{b'>b}}MDH^{\infty}_k(S\calH_{b',\eta_1}(X);A)\cong \lim_{\stackrel{\longrightarrow}{b'>b}}MDH^{\infty}_k(\calH_{b',\eta}(X);A).$$

It remains to show the existence of a $b'$ such that the second limit is isomorphic to the singular homology of the punctured $b''$-horn neighborhood $\calH_{b'',\eta}(X)\setminus\{O\}$, for any $b''\in (b,b')$, but this follows exactly in the same way than the same statement for spherical neighborhoods in the proof of Theorem~\ref{theo:finiteness}.
\endproof

\begin{definition}
\label{def:jumpingrate}
Let $(X,O,d_X)$ be a metric subanalytic germ. An element $b\in\BB$ is called a {\em jumping rate} of  $(X,O,d_X)$ if for any positive $\epsilon$ the natural homomorphism
$$MDH^{b+\epsilon}_\bullet(X,O,d_{X};A)\to MDH^{b-\epsilon}_\bullet(X,O,d_{X};A)$$
is not an isomorphism
\end{definition}

\begin{theo}
\label{theo:finitejumps}
Let $(X,O,d_X)$ be a closed metric subanalytic germ in $\RR^n$ equipped either with the inner or the outer metric. The set of jumping rates of  $(X,O,d_X)$ is finite. 
\end{theo}
\proof
By Theorem 2.1 of \cite{reembedding} there exists a subanalytic subgerm  $(X',x'_0)\subset \RR^m$ such that $(X,x_0,d)$ is bi-Lipschitz homeomorphic to $(X',x'_0)$ both with the inner and outer metric. This reduces the proof to the case of the outer metric. 

Fixed $\eta>0$ and a small $\epsilon>0$, consider the family $S\calH_{b,\eta}(X)\cap\epsilon$, parametrized by $b\in (0,\infty]$. This is a definable family in $\RR_{an}^{\RR}$ (see the definition in \cite{Miller}). By Hardt's Triviality Theorem, there are finitely many values $0=b_1<...<b_N=\infty$ in $(0,\infty]$ such that the family is trivial in the open intervals $(b_i,b_{i+1})$. A corollary of the proof of Theorem~\ref{theo:finiteness} implies that the set of jumping numbers is contained in $\{b_1,...,b_N\}$.  
\endproof 

The following Theorem is due to A. Parusinski. It is not true in more general O-minimal structures.

\begin{theo}[Rationality]
\label{theo:rat}
Let $(X,x_0)\subset(\RR^m,x_0)$ be a subanalytic germ and $d_X$ be either the inner or the outer metric. Then the jumping rates are rational numbers.
\end{theo}
\proof
We reduce the proof to the case of the outer metric as in the previous Theorem. 

Consider the continuous subanalytic map $f:\RR^m\to\RR^2$ defined by $f(x):=(||x||,d(x,X))$. By Hardt Triviality Theorem there is a subanalytic decomposition of $\RR^2$ such that $f$ is trivial over each of the pieces. The germ at the origin of $\RR^2$ of such a decomposition consists of a finite number of separating subanalytic curve germs. Each of this curves can be parametrized by an expression of the form $(t,ct^b+ h.o.t)$, where $b$ is a rational number; we say that $b$ is the leading exponent of the curve. By Corollary~\ref{cor:birbrair} the set of jumping rates is the set of leading exponents of the jumping rates. 
\endproof

Observe that the Theorem above gives an alternative proof, which works for polynomially bounded O-minimal structures, for the finiteness of the set of jumping rates too.

\section{The \texorpdfstring{$b$}{b}-MD Homology of \texorpdfstring{$b$}{b}-cones with the outer metric}.\label{sec:stdcones}

Recall Definition~\ref{def:stdcones}. The purpose of this section is to prove the following proposition. 

\begin{prop} 
\label{prop:stdcones}
Let $(L,L_1)\subset \RR^k$ be a pair of subanalytic compact sets. {Given $b\in (0,+\infty)\cap\QQ$, consider the pair of $b$-cones $(C^b_L, C^b_{L_1},d_{out})$ with the outer metric $d_{out}$.}
\begin{itemize}
\item [(a)] If $b'<b$ then
$$
MDH_n^{b'}(C^b_L, C_{L_1}^b,d_{out};A)=\left\{\begin{array}{ll}
A&, \mbox{ if }n=0 \mbox{ and } L_1=\emptyset\\
0&, \mbox{ if }n\not =0 \mbox{ or } n=0 \mbox{ and } L_1\neq \emptyset.
\end{array}\right.
$$
\item [(b)] 
$$MDH_*^{b}(C_L^b, C_{L_1}^b,d_{out};A)=H_*(L,L_1;A),$$
where $H_*(L,L_1;A)$ is the ordinary homology with coefficients in $A$.
\end{itemize}
\end{prop}

Before giving a proof we need some preparation.

Given a l.v.a. simplex $\sigma:\hat{\Delta}_n\to C^b_L$, we write it in terms of the explicit description of $C^b_L$ (see Definition~\ref{def:stdcones}) as $(\sigma(tx,t)=((\sigma_2(tx,t))^b\sigma_1(tx,t),\sigma_2(tx,t))$. 

\begin{lema}
\label{lem:charcbeq}
Let $((\sigma_2)^b\sigma_1,\sigma_2)$ and $((\sigma'_2)^b)\sigma'_1,\sigma'_2)$ be two l.v.a. $n$-simplexes in $(C^b_L,d_{out})$. Then $\sigma\sim_{b}\sigma'$ if and only if the following two conditions hold for any l.v.a. subanalytic  continuous arc $\gamma:[0,\epsilon)\to\hat{\Delta}_n$:
\begin{enumerate}
 \item $\lim\limits_{t\to 0^+} \frac{|\sigma_2(\gamma(t))-\sigma'_2(\gamma(t))|}{t^{b}}=0$;
 \item $\lim\limits_{t\to 0^+} \sigma_{1}(\gamma(t))= \lim\limits_{t\to 0^+}\sigma'_1(\gamma(t))$.
\end{enumerate}
\end{lema}
\proof
By the arc criterium (see Lemma~\ref{lem:arccharacterization}) and the definition of $C^b(L)$ we have that $\sigma\sim_{b}\sigma'$ if and only if for any l.v.a. subanalytic  continuous arc $\gamma:[0,\epsilon)\to\hat{\Delta}_n$ the first condition of the lemma is satisfied and also we have the vanishing of the following limit:
\begin{equation}
\label{eq:local1} 
\lim\limits _{t\to 0^+} \frac{||\sigma_2(\gamma(t))^b\sigma_1(\gamma(t))-\sigma'_2(\gamma(t))^b\sigma'_1(\gamma(t))||}{t^{b}}=0.
\end{equation}

We may write 
$$||\sigma_2(\gamma(t))^b\sigma_1(\gamma(t))-\sigma'_2(\gamma(t))^b\sigma'_1(\gamma(t))||=$$
$$=||\sigma_2(\gamma(t))^b\sigma_1(\gamma(t))-\sigma_2(\gamma(t))^b\sigma'_1(\gamma(t))+ \sigma_2(\gamma(t))^b\sigma'_1(\gamma(t))  -\sigma'_2(\gamma(t))^b\sigma'_1(\gamma(t))||=$$
$$=||\sigma_2(\gamma(t))^b(\sigma_1(\gamma(t))-\sigma'_1(\gamma(t))+ (\sigma_2(\gamma(t))^b  -\sigma'_2(\gamma(t))^b)\sigma'_1(\gamma(t))||.$$

Let $\sigma_2(\gamma(t))=at+g(t)$ and $\sigma'_2(\gamma(t))=a't+g'(t)$, where $g$ and $g'$ vanish to order strictly larger than $1$. The first condition of the lemma is equivalent to the equality $a=a'$. Using this and the fact that $\sigma'_1(\gamma(t))$ is bounded because $L$ is compact we conclude that $||(\sigma_2(\gamma(t))^b  -\sigma'_2(\gamma(t))^b)\sigma'_1(\gamma(t))||$ vanishes at $0$ to order strictly larger than $b$. Moreover, by the l.v.a condition we have that $a\neq 0$. Therefore the vanishing of the limit~(\ref{eq:local1}) is equivalent to the second condition of the lemma.
\endproof

Given a subanalytic map $f:L\to L'$ we define $C^b(f):C^b(L)\to C^b(L')$ by the formula $C^b(f)(t^bx,t)=(t^bf(x),t)$. In the next lemma we extend functoriality for $MD$ homology at discontinuity parameter $b$ to this kind of maps, and prove homotopy invariance. 

\begin{lema}
\label{lem:extendedbhomotopy}
Given any continuous subanalytic map $f:L\to L'$, the map $C^b(f)$ induces a morphism of $b$-moderately discontinuous chain complexes $$C^b(f)_*:MDC^b_\bullet(C^b(L),d_{out})\to MDC^b_\bullet(C^b(L'),d_{out})).$$ 

Let $f,g:L\to L'$ be two subanalytic maps between compact subanalytic subsets. Assume that there exists a subanalytic homotopy $H:L\times I\to L'$ such that $H_0=f$ and $H_1=g$. Then the morphism $C^b(f)_*$ and $C^b(g)_*$ are chain homotopic.
\end{lema}
\proof
We define $C^b(f)_*:MDC^{pre,\infty}_\bullet(C^b(L))\to MDC^{pre,\infty}_\bullet(C^b(L'))$ by the formula $C^b(f)(\sigma_1^b\sigma_2,\sigma_2):=(\sigma_2^bf\comp\sigma_1,\sigma_2)$. Compatibility with immediate subdivision is obvious. Compatibility with the $\sim_b$ equivalence relation follows immediately by the continuity of $f$ and Lemma~\ref{lem:charcbeq}. Then by Remark~\ref{rem:checkindependence} $C^b(f)_*$ descends to a morphism from $MDC^b_\bullet(C^b(L))\to MDC^b_\bullet(C^b(L')).$

The homotopy assertion follows from the proof of Theorem~\ref{theo:homotopyinvariance} closely. The only difference is when proving that we have the equivalence $h_n(\sigma)\sim_b h_n(\sigma')$, provided that we have $\sigma\sim_b\sigma'$. Here one {has to check the two conditions in Lemma~\ref{lem:charcbeq}  for $h_n(\sigma)$ and $h_n(\sigma')$ using the continuity of the homotopy $H$.}

\endproof
\begin{proof}[Proof of Proposition~\ref{prop:stdcones}]
If $b'<b$, we use Corollary~\ref{cor:colapsdiam point}
for the map $X_L^b\to [0,\infty)$ which sends $(t^bx,t)\mapsto t$. If $L_1 = \emptyset$, the result follows from Proposition~\ref{prop:point}. If not, it follows from the relative homology sequence (see Proposition~\ref{prop:relativehomologysequence}).

For $b'=b$, we use the relative homology sequence and the $5$-lemma to reduce to the absolute case $L_1=\emptyset$.

Let $K$ be a simplicial complex and $\theta:|K|\to L$ be a subanalytic triangulation of $L$. Define 
$h:C(|K|)\to C^b_L$ by the formula $h(tx,t):=(t^b\theta(x),t)$. 
As in Section~\ref{sec:link} consider the closed subanalytic cover $\{Z_i\}_{i\in I}$ of $L$ given by the maximal triangles in $L$ of the triangulation $\theta$. The singular homology of $L$ is the homology of the nerve of the cover.

In order to compute $MDH_*^{b}(C_L^b;A)$ we produce a $b$-cover of $C_L^b$ as follows. Let $\theta_2:|K_2|\to L$ be the triangulation obtained by taking two iterated barycentric subdivisions of $K$. For each $i\in I$ we define $V_i$ to be the open subset obtained by taking the union of $Z_i$ together with the interior of all the simplices of $|K_2|$ which meet $Z_i$. For any finite set of indexes $J\subset I$ define $Z_J:=\cap_{i\in J}Z_i$ and $V_J:=\cap_{i\in J}V_i$. By construction there exists a subanalytic retraction $r_J:V_J\to Z_J$ and a subanalytic homotopy $H_J$ between $r_J$ and $Id_{V_J}$.

If we define $\hat{U}_i:=C_{V_i}^{b}$ we obtain that the collection $\{\hat{U}_i\}_{i\in I}$ is a $b$-extension of the closed subanalytic cover 
$\{C_{Z_i}^b\}_{i\in I}$. {Indeed, by Lemma~\ref{lem:extendedbhomotopy} it is enough to check that for any finite subset of indexes $J$ the inclusion $Z_J \hookrightarrow V_J$ admits a continuous subanalytic deformation retraction. This is true by construction.}

Hence Corollary~\ref{cor:nerve} can be applied and we obtain, as in the proof of Theorem~\ref{theo:link}, that $MDH_*^{b'}(C_L^b;A)$ coincides with the homology of the nerve of the cover. 


\end{proof}

\section{The MD homology for \texorpdfstring{$b=1$}{b=1}}  

In this section we study $b$-MD Homology for $b=1$, both for the outer and inner metric. We give a comfortable interpretation of $1$-maps in terms of spherical blow ups, prove that $1$-MD Homology coincides with the homology of the punctured tangent cone for the outer metric, and with the homology of the punctured Gromov tangent space for the inner metric. As an application we prove that if a complex analytic germ has the MD Homology of a smooth germ, then it has to be smooth.

\begin{definition}[See \cite{BirbrairFG:2017} and \cite{Sampaio:2017}]
Let $\rho_m:\SSS^{m-1}\times [0,\infty)\to\RR^m$, defined by $\rho_m(x,r):=rx$, be the {\em spherical blow up of the origin of $\RR^m$}. Let $X\subset\RR^m$ be a subanalytic set. The {\em spherical blow up} $X'$ of $X$ at the origin $O\in\RR^m$ is the  closure of $\rho_m^{-1}(X\setminus\{O\})$. The spherical exceptional divisor is defined to be the intersection $\partial X':=X'\cap (\SSS^{m-1}\times\{0\})$. Let $(X,O)\subset (\RR^m,O)$ and $(Y,O)\subset (\RR^k,O)$ be subanalytic germs. A {\em blow-spherical map} $\varphi:X\to Y$ is a subanalytic map that lifts to a continuous map from the spherical blow up of $X$ at $O$ to the spherical blow up of $Y$ at $O$. 
\end{definition}

\begin{remark}
It follows from the proof of Proposition 3.6 in \cite{Sampaio:2017} that a Lipschitz subanalytic l.v.a. map is a blow-spherical map. 
\end{remark}

Let $(X,x_0,d_{out})$ be a metric subanalytic subgerm of $\RR^m$ with the outer metric. Observe that two points $p,q:[0,\epsilon)\to X$ are $1$-equivalent if and only if their liftings to the spherical blow up meet at the same point of the exceptional divisor. This gives the following interpretation of $1$ maps:

\begin{remark}
\label{rem:extendfunct}
Let $(X,d_X)$ and $(Y,d_Y)$ be metric subanalytic germs, where $d_X$ and $d_Y$ are the outer metrics. A $1$-map between them is a finite collection $\{(C_i,f_i)\}_{i\in I}$, where $\{C_i\}_{i\in I}$ is a cover by closed subanalytic subsets of $X$ and $f_i:C_i\to Y$ are l.v.a. blow-spherical subanalytic maps satisfying the following: for any $1$-equivalent pair of points $p$, $q$ contained in $C_i$ resp. $C_j$, $f_i\circ p$ and $f_j\circ q $ have liftings to the spherical blow up of $Y$ meeting at the same point of the exceptional divisor.
\end{remark}

Let us remind the notion of tangent cone.
\begin{definition}
Let $X\subset \RR^m$ be a subanalytic set such that $x_0$ is a non-isolated point of $\overline{X}$.
We say that $v\in \RR^m$ is a tangent vector of $X$ at $x_0\in\RR^m$ if there are a sequence of points $\{x_i\}\subset X\setminus \{x_0\}$ tending to $x_0$ and sequence of positive real numbers $\{t_i\}$ such that 
$$\lim\limits_{i\to \infty} \frac{1}{t_i}(x_i-x_0)= v.$$
Let $T_{x_0}X$ be the set of all tangent vectors of $X$ at $x_0\in \RR^m$. We call $T_{x_0}X$ the {\em tangent cone of $X$ at $x_0$}.
\end{definition}

The tangent come is known to be subanalytic   in $\RR^m$ and $0\in T_{x_0}X$. Moreover, if $v\in T_{x_0}X$ then $\lambda v\in T_{x_0}X$ for every $\lambda\in\RR^+$. 

\begin{theo}
\label{th:speed1}
Let $(X,x_0,d_{out})\subset (\RR^m,x_0)$ be a closed subanalytic germ endowed with the outer metric. Then there is a natural isomorphism $$MDH^1_*(X,x_0,d_{out};A)\to MDH^1_*(T_{x_0}X;A)\to H_*(T_{x_0}X\setminus\{0\};A).$$
\end{theo}
\proof

Since we can find a triangulation of the unit sphere $\mathbb S^{m-1}$ centered at $x_0$ compatible with $X_0:= T_{x_0}X\cap \mathbb S^{m-1}$ (see \cite{Hardt:1976}), we have that if $\eta>0$ is small enough, then there exists a subanalytic retraction $$r\colon S\calH_{1,\eta}(T_{x_0}X)\cap\SSS^{m-1} \to T_{x_0}X\cap\SSS^{m-1}.$$
Thus, the subanalytic map $\tilde r\colon S\calH_{1,\eta}(T_{x_0}X)\to T_{x_0}X$ given by $\tilde r(ty)=tr(y)$, for any $y\in S\calH_{1,\eta}(T_{x_0}X)$ and $t\geq 0$, is a subanalytic retraction.

We have an inclusion of germs $(X,0)\subset S\calH_{1,\eta}(T_{x_0}X)$. Restricting $\tilde r$ to $(X,0)$ we obtain a subanalytic collapsing map $c\colon (X,0)\to (T_{x_0}X,0)$ preserving the distance to the origin. As both $r$ and the inclusion extend to the real blow up at $0$, the map $c$ extends as well. By Remark~\ref{rem:extendfunct} and functoriality for $1$-maps it induces a natural morphism of complexes $c_*\colon MDC^1_*(X;A)\to MDC^1_*(T_{x_0}X;A)$.

Note that the diameter condition 
\begin{equation}
\label{eq:diamcond}
 \lim\limits _{t\to 0^+}\frac{sup\{diam (c^{-1}(y)\cup \{y\}):y\in T_{x_0}X\cap\SSS_t\}}{t}=0
\end{equation}
holds. Then we apply Theorem~\ref{theo:colapsediameter2} and obtain the first isomorphism.

The second isomorphism is an application of Proposition~\ref{prop:stdcones}. 

\endproof

Now we clarify the $1$-MD Homology for germs with the inner metric. Let $(X,x_0,d_{inn})\subset (\RR^m,x_0)$ be a closed subanalytic germ endowed with the inner metric. Let $T^{Gromov}_{x_0}X$ be the Gromov tangent space of $X$ at $x_0$ and $C^{Alexandrov}_{x_0}X$ be the Alexandrov cone of $X$ at $x_0$ (see~\cite{Bernig}). In~\cite{Bernig} a tangent map $T^{Gromov}_{x_0}X\to T_{x_0}X$ from the Gromov tangent space to the ordinary tangent space is defined. We would like to thank A. Parusinski for suggesting to look in this direction. 

\begin{theo}
\label{th:speed1inner}
Let $(X,x_0,d_{inn})\subset (\RR^m,x_0)$ be a closed subanalytic germ endowed with the inner metric. Then there are natural isomorphisms $MDH^1_*(X,x_0,d_{inn};A)\cong H_*(T^{Gromov}_{x_0}X\setminus\{x_0\};A)\cong H_*(C^{Alexandrov}_{x_0}X\setminus\{x_0\};A)$. Moreover, under this isomorphism and the isomorphism proved in Theorem~\ref{th:speed1} the homomorphism $$MDH^1_*(X,x_0,d_{inn};A)\to MDH^1_*(X,x_0,d_{out};A)$$ coincides with the homomorphism in singular homology induced from the tangent map $T^{Gromov}_{x_0}X\to T_{x_0}X$.
\end{theo}
\proof
By Theorem 2.1 of \cite{reembedding} there exists a subanalytic subgerm  $(X',x'_0)\subset \RR^m$ such that $(X,x_0,d)$ bi-Lipschitz homeomorphic to $(X',x'_0)$ both with the inner and outer metric. For $X'$ the Gromov tangent space and the ordinary tangent space are homeomorphic as pointed topological spaces. This shows the first isomorphism. The second holds by Theorem~1.2~of~\cite{Bernig}. The identification of the homomorphism $MDH^1_*(X,x_0,d_{inn};A)\to MDH^1_*(X,x_0,d_{out};A)$ with the homomorphism in singular homology induced from the tangent map is clear by the proof of the isomorphisms above. 
\endproof

Finally we prove the characterization of smooth points in terms of MD Homology for the outer metric.

\begin{theo}
\label{th:characsmooth}
Let $(X,x_0,d_{out})$ be a complex analytic germ endowed with the outer metric. If the moderately discontinuous homology $MDH^\bullet_*(X,x_0,d_{out};\ZZ)$ coincides with the moderately discontinuous homology of a smooth germ, then $(X,x_0)$ is a smooth germ.
\end{theo}
\proof
By Theorem\ref{th:speed1} we can identify  $MDH^1_*(X;\ZZ)$ with $MDH^1_*(T_{x_0}X;\ZZ)$, and by Proposition~\ref{prop:stdcones} we have an isomorphism $MDH^\infty_*(T_{x_0}X;\ZZ)\cong MDH^1_*(T_{x_0}X;\ZZ)$; combining this with Theorem~\ref{theo:link} and using that  $MDH^1_*(X;\ZZ)$ coincides with the moderately discontinuous homology of a smooth germ, we obtain the link of the tangent cone $T_{x_0}X$ is an integral homology sphere.  

$T_{x_0}X\setminus \{0\}$ is fibered over the projectivization $\PP(T_{x_0}X)$ with fibre $\CC^*$. Computing the integral cohomology via the Leray spectral sequence as in~\cite{FFS}, and imposing the integral homology sphere condition, we obtain that each of the $d_2$ differentials
$$\ZZ=H^0(\PP(T_{x_0}X),\ZZ)\stackrel{d_2}{\longrightarrow}H^2(\PP(T_{x_0}X),\ZZ)\stackrel{d_2}{\longrightarrow} ...\stackrel{d_2}{\longrightarrow} H^{2d-2}(\PP(T_{x_0}X),\ZZ)=\ZZ$$
is an isomorphism. The composition of all the differentials equals the degree. Hence the tangent cone is of degree $1$, and therefore smooth.

Since $MDH^\infty_*(X;\ZZ)$ coincides with the moderately discontinuous homology of a smooth germ of dimension $d$. Then by Theorem~\ref{theo:link} $MDH^\infty_*(X;\ZZ)$ is the usual homology of a sphere of dimension $2d-1$. This in particular shows that $X$ is irreducible.

The natural homomorphism $MDH^\infty_{2d-1}(X;\ZZ)\cong\ZZ \to MDH^1_*(X;\ZZ)\cong\ZZ$ is multiplication by an integer $k$. By the identification of $MDH^\infty_{2d-1}(X;\ZZ)$ with $H_{2d-1}(X\setminus\{x_0\};\ZZ)$, and of $MDH^1_*(X;\ZZ)$ with $H_{2d-1}(T_{x_0}X\setminus\{0\};\ZZ)$ the number $k$ equals the relative multiplicity of $X$ at the only component of its tangent cone (see for example \cite{KurRab}). Since the homomorphism $MDH^\infty_{2d-1}(X;\ZZ)\to MDH^1_*(X;\ZZ)$ coincides with the same homomorphism for the case of a smooth germ, we conclude the equality $k=1$.

So $X$ is irreducible, has smooth tangent cone, and its relative multiplicity equals $1$. This implies that $X$ is smooth.

\endproof

\section{Framed Moderately Discontinuous Homology}

In this section we enrich MD Homology by adding to it the choice of certain fundamental classes. Being a quite innocent improvement we will see later that it helps to capture much richer Lipschitz information. Although a more general formulation is possible, we choose to work here in the complex analytic setting and for the outer metric, since this suffices for the applications we have in mind. 

\begin{definition}
\label{def:frame}
A {\em $(b_1,...,b_l)$-framed graded $\BB$-abelian group at degree $k$} is given by a pair $(H^\bullet_*,\calB_{1},...,\calB_{l})$, where $H^\bullet_*$ be an element of $\BB-GrAb$ such that $H^{b_i}_k$ is free for any $i$ and $\calB_i$ is a basis of $H^{b_i}_k$. An {\em isomorphism of $(b_1,...,b_l)$-framed graded $\BB$-abelian groups} is an isomorphism of graded $\BB$-abelian group sending each basis at the source to the corresponding basis at the target.
\end{definition}

It is easy to produce frames at degree $0$, due to the description of the MD Homology at degree $0$. However the most useful frames for us are at maximal degree and use fundamental classes.

Let $(X,x_0,d_{out})$ be a complex analytic germ of dimension $d$ with the outer metric. Let $\calT$ be a subanalytic triangulation of $X$ compatible with $x_0$. Let $X=\cup_{i=1}^rX_i$ be the decomposition of $X$ in irreducible components. For any $i\leq r$ let $[X_i]$ be the unique homology class in  $MDH^\infty_{2d-1}(X,x_0,d_{out};\ZZ)$ given by the sum of l.v.a positively oriented parametrizations of the simplexes of dimension $2d$ contained in $X_i$ and which meet the vertex. By Theorem~\ref{theo:link} the group $MDH^\infty_{2d-1}(X,x_0,d_{out};\ZZ)$ is free and generated by the classes $\{[X_i]\}_{i=1}^r$. 

Let $T_{x_0}X=\cup_{i=1}^sY_i$ be the decomposition in irreducible components of the tangent cone of $X$ at $x_0$. By Theorems~\ref{th:speed1} and~\ref{theo:link} the group $MDH^1_{2d-1}(X,x_0,d_{out};\ZZ)$ is free and generated by the classes $\{[Y_i]\}_{i=1}^s$.

\begin{definition}
\label{def:framed}
Let $(X,x_0,d_{out})$ be a complex analytic germ of dimension $d$ with the outer metric. Its {\em $\{\infty,1\}$-framed MD Homology at degree $2d-1$} is the triple
$$(MDH^\bullet_*(X,x_0,d_{out},\ZZ),\{[X_i]\}_{i=1}^r,\{[Y_i]\}_{i=1}^s),$$
where the classes $[X_i]$ and $[Y_i]$ are defined in the discussion above. The {\em $\infty$-framed or $1$-framed MD Homology} is the same but only considering one of the basis.
\end{definition}

\begin{remark}
\label{rem:framedinv}
The $\{\infty,1\}$-framed MD Homology at degree $2d-1$ is invariant by subanalytic bi-Lipschitz isomorphism, and hence by complex analytic isomorphism. The same holds for the $\infty$-framed or $1$-framed MD Homology. 
\end{remark}
\proof
If $f:(X,x_0,d_{out})\to (Y,y_0,d_{out})$ is a subanalytic bi-Lipschitz isomorphism, then it takes the fundamental classes of the irreducible components of $X$ into the fundamental classes of the irreducible components of $Y$. This proves the $\infty$-framed case.

For the $1$-framed case, by Theorem 3.2 of~\cite{Sampaio:2016} (see also Proposition 3.6 of~\cite{Sampaio:2017} or Remark 2.2 of~\cite{Bernig}) there is a subanalytic homeomorphism from $(T_{x_0}X,0)$ to $(T_{y_0}Y,0)$, which send the fundamental classes of the irreducible components of $T_{x_0}X$ to the fundamental classes of the irreducible components of $T_{y_0}Y$.
\endproof

The following proposition and example illustrates the advantage of taking framed MD Homology.  


\begin{prop}
\label{cor:mixed_multiplicities}
Let $(X,x_0,d_{out})$ be an irreducible closed complex analytic germ of complex dimension $d$ endowed with the outer metric. Then, in terms of the basis given by $\{\infty,1\}$-framed MD Homology at degree $2d-1$, the morphism 
$$MDH^\infty_{2d-1}(X,x_0,d_{out};A)\to MDH^1_{2d-1}(X,x_0,d_{out};A)$$ can be identified with the morphism $A\to A^r$, sending $a\mapsto (m_1a,...,m_ra)$, where $r$ is the number of irreducible components of the tangent cone and $m_i$ are the relative multiplicities (see for example \cite{KurRab}). Consequently relative multiplicities are determined by $\{\infty,1\}$-framed MD Homology at degree $2d-1$.
\end{prop}
\proof
The natural morphism $MDH^\infty_*(X;A)\to MDH^1_*(X;A)$ has been identified above with the morphism $MDH^\infty_*(X;A)\to MDH^\infty_*(T_{x_0}X;A)$ induced by the collapsing map defined in the proof of Theorem~\ref{th:speed1}.
This in turn is identified with the usual homology homomorphism $H_*(L(X),A)\to H_*(L(T_{x_0}X),A)$, where $L(X)$ and $L(T_{x_0}X)$ are the links of $X$ and $T_{x_0}X$. Taking $*=2d-1$ we obtain the isomorphism $H_*(L(X),A)\cong A$, generated by the fundamental class, and the isomorphism $H_*(L(T_{x_0}X),A)\cong A^r$, generated by the fundamental classes of the links of each of the components. The relative multiplicities are the cover degree of $L$ into each of the links of the irreducible components.
\endproof

\begin{remark}Proposition~\ref{cor:mixed_multiplicities} gives a proof that the relative multiplicities of complex analytic sets are subanalytically bi-Lipschitz invariant. This was already proved in the proof of the Proposition 2.5 of~\cite{Valette:2010b} and a genenarization of that was proved in Proposition 1.6 of~\cite{FernandesS:2016}.
\end{remark}

Observe that for the above argument fixing the basis is important. The homomorphisms $\ZZ\to\ZZ^2$ given by $1\mapsto (1,2)$ and $1\mapsto (1,3)$ are conjugate by an automorphism of $\ZZ^2$ and therefore can not be distinguished algebraically unless one fixes a basis.

\section{MD Homology of plane curves with the outer metric.}\label{section:curves}

Throughout this subsection, whenever we say {\em curve germ}, we refer to a  complex algebraic plane curve germ in the origin equipped with the outer geometry. We are going to recall the definition of the Eggers-Wall tree of a curve germ. It uses the following correspondence between Puiseux pairs and Puiseux exponents: let $(m_1, k_1)\dots (m_l, k_l)$ denote all Puiseux pairs of a curve germ in order. Then the Puiseux exponents of that curve germ are given by $\frac{m_i}{\prod_{j=1}^i k_j}$ for $i=1,\dots l$. We call $(m_i, k_i)$ the {\em Puiseux pair corresponding to $\frac{m_i}{\prod_{j=1}^i k_j}$}.

Recall the following definition of the contact number between two branches:

\begin{definition}
Let $C$ be a curve germ. Let $f_i = \sum_{j=1}^{\infty}\alpha_{i,s} x^{s}$ 
be Puiseux parametrizations of the branches $C_i$ of $C$, where $i\in \{1,\dots,n\}$. Let $i \neq k \in I$. The {\em contact number $c(C_i, C_k)$} between $C_i$ and $C_k$ is defined as
$$
c(C_i, C_k) := \min \{ s : \alpha_{i, s} \neq \alpha_{k,s} \}
$$
\end{definition}

\begin{definition}
Let $C$ be a curve germ. In this definition, we are defining the {\em Eggers-Wall tree $\mathcal{G}_C$ of $C$}. Depending on the context, $\mathcal{G}_C$ can be interpreted either as a graph or as a topological space with a finite number of special points; we call these special points in the topological space {\em vertices} as they correspond to the vertices in the graph.

If $C$ is irreducible, we define the {\em Eggers-Wall tree $\mathcal{G}_C$ of $C$} to be the segment $[0,\infty]$ with a vertex at both ends and one vertex at each rational number in that segment that is a Puiseux exponent of $C$.
Every vertex is decorated by the corresponding value in $\mathbb{Q} \cup \{ \infty\}$. 
For two adjacent vertices at $q_1$ and $q_2$ respectively, with $q_1 < q_2$, the edge between them is weighted by the product $\prod_{i=0}^l k_i$, where $k_0 = 1$ and $(m_1,k_1),\dots, (m_l,k_l)$ are all Puiseux pairs corresponding to Puiseux exponents less than or equal to $q_1$. 

If $C$ is reducible, the {\em Eggers-Wall tree $\mathcal{G}_C$} is defined as follows. Let $C_n$ denote one of its branches and let $\hat{C}_n$ denote the union of all the other branches. Let $c$ be the greatest contact number that $C_n$ has with any of the other branches and let $C_k$ be one of the branches that $C_n$ has that contact number with. If $C$ has only two branches, we have $C_k = \hat{C}_n$.
If $\mathcal{G}_{C_n}$ does not have a vertex at $c$, add it in the following manner: let $q_1$ be the greatest vertex in $\mathcal{G}_{C_n}$ smaller than $c$ and $q_2$ the smallest one greater than $c$. We add $c$ as a vertex in $\mathcal{G}_{C_n}$ and give both edges $\{q_1, c\}$ and $\{c, q_2\}$ the weight the edge $\{q_1, q_2\}$ had before. Then, we do the same for the segment in $\mathcal{G}_{\hat{C}_n}$ corresponding to $\mathcal{G}_{C_k}$, if it does not contain $c$ as a vertex already.
Now, glue the segment from $0$ to $c$ in $\mathcal{G}_{C_n}$ to the segment from $0$ to $c$ in $\mathcal{G}_{C_k}$ by the identity on $[0,c]$. As $\mathcal{G}_{C_k}$ is naturally embedded in $\mathcal{G}_{\hat{C}_n}$, we have glued $\mathcal{G}_{\hat{C}_n}$ and $\mathcal{G}_{C_n}$ to one graph $\mathcal{G}_C$.

There is a natural map $r: \mathcal{G}_C \to [0,\infty]$ defined as follows: For a point $g \in \mathcal{G}_C$, let $C_g$ be one of the branches of $C$ for which $g$ is in the image of the natural inclusion $\mathcal{G}_{C_g} \hookrightarrow \mathcal{G}_C$. We assign to $g$ the point in the segment $[0, \infty]$ that is sent to $g$ by that inclusion.
\end{definition}

\begin{example}\label{example:EggersTree}
Let $C$ be the curve with the following four branches:
\begin{align*}
& C_1 = \{(x, y) \in \mathbb{C}^2 : y = x^{\frac{3}{2}} + x^{\frac{5}{2}}\}, \\
& C_2 = \{(x, y) \in \mathbb{C}^2 : y = x^{\frac{3}{2}} + x^{\frac{11}{4}}\},\\
& C_3 = \{(x, y) \in \mathbb{C}^2 : y = x^{\frac{3}{2}} + x^{\frac{11}{4}} + x^{\frac{37}{12}}\},\\
& C_4 = \{(x, y) \in \mathbb{C}^2 : y = x^{\frac{5}{2}} + x^{\frac{11}{4}}\}.
\end{align*}
We have visualized the Eggers-Wall tree $\mathcal{G}_C$ together with the function $r: \mathcal{G}_C \to [0,\infty]$ in Figure~\ref{figure:EggersTree}.

\begin{figure}[h!]\centering
    \includegraphics[width=.7\textwidth]{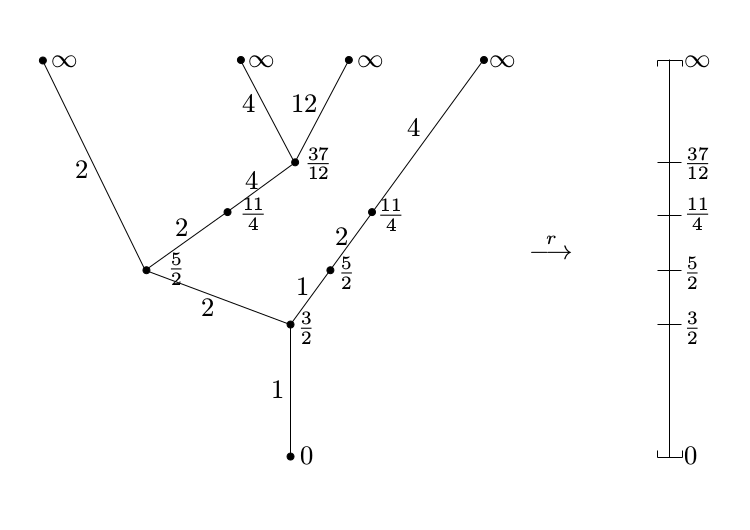}
 \caption{The Eggers-Wall tree $\mathcal{G}_C$ and $r: \mathcal{G}_C \to [0,\infty]$.}
 \label{figure:EggersTree}
\end{figure}
\end{example}

\begin{theo}\label{th:homology of curves}
Let $C$ be a curve germ. Let $\calG_C$ be the Eggers-Wall tree of $C$ with $r: \mathcal{G}_C \to [0,\infty]$ as defined above. Let $A$ be an abelian group. 
The MD homology of $C$ with respect to $A$ can be described as follows:
\begin{enumerate}
\item\label{item:curveshomgroups} For any $b \in [1, \infty]$, it is $MDH_0^b(C;A) \cong MDH_1^b(C;A) \cong A^{l_b}$, where $l_b$ is the number of points in $r^{-1}(b + \epsilon)$, where $\epsilon$ is so small that $r^{-1}((b, b + \epsilon])$ does not contain a vertex. For the case $b = \infty$, we consider $\infty + \epsilon = \infty$.
\item\label{item:curveshigherhomgroups} For any $b \in [1, \infty]$ and $n > 1$, it is $MDH_n^b(C;A) \cong \{0\}$.
\item\label{item: homomorphisms of homology of curves} For $b_1, b_2 \in [1,\infty]$ with $b_1 \geq b_2$, $h^{b_1, b_2}_0(C)$ and $h^{b_1, b_2}_1(C)$ are the homomorphisms given by multiplication with the following matrices $M_0$ and $M_1$, respectively: let $\epsilon$ be so small that $r^{-1}((b_1, b_1 + \epsilon])$ and $r^{-1}((b_2, b_2 + \epsilon])$ do not contain any vertices. Let $\mathcal{G}_1$,..., $\mathcal{G}_l$ be the connected components of $r^{-1}([b_2 + \epsilon, b_1 + \epsilon])$, where $l = l_{b_2}$. For $i \in \{1,\dots, l\}$, let $p_i$ be the unique point in $r^{-1}(b_2 + \epsilon) \cap \mathcal{G}_i$. Further, let $p_{1,i},\dots,p_{m_i, i}$ be the points in $r^{-1}(b_1 + \epsilon)\cap \mathcal{G}_i $. Notice that $\sum_{i=1}^l m_i = l_{b_1}$. We define
$$
M_0 := \left( 
   \overbrace{\begin{array}{ccc}
   1 & \dots & 1 \\
   0 & \dots & 0 \\
   0 & \dots & 0 \\
   & \vdots \\
   0 & \dots & 0
              \end{array} }^{m_1 \text{ times}}
   \overbrace{\begin{array}{ccc}
   0 & \dots & 0 \\
   1 & \dots & 1 \\
   0 & \dots & 0 \\
   & \vdots \\
   0 & \dots & 0
              \end{array} }^{m_2 \text{ times}}
              \dots
   \overbrace{\begin{array}{ccc}
   0 & \dots & 0 \\
   0 & \dots & 0 \\
   & \vdots \\
   0 & \dots & 0 \\
   1 & \dots & 1 
              \end{array} }^{m_l \text{ times}}
         \right)
$$
Now, for $i\in \{1,\dots,l\}$ and $j \in \{1,\dots, m_i\}$, let $k_{j,i} := \frac{w_{j,i}}{w_i}$, where $w_{j,i}$ and $w_i$ are the weights assigned to the edges on which $p_{j,i}$ respectively $p_i$ lie. We define
$$
M_1 := \left( 
   \begin{array}{ccc}
   k_{1,1} & \dots & k_{m_1,1} \\
   0 & \dots & 0 \\
   0 & \dots & 0 \\
   & \vdots \\
   0 & \dots & 0
              \end{array}
   \begin{array}{ccc}
   0 & \dots & 0 \\
   k_{1,2} & \dots & k_{m_2,2} \\
   0 & \dots & 0 \\
   & \vdots \\
   0 & \dots & 0
              \end{array}
              \dots
   \begin{array}{ccc}
   0 & \dots & 0 \\
   0 & \dots & 0 \\
   & \vdots \\
   0 & \dots & 0 \\
   k_{1,l} & \dots & k_{m_l,l} 
              \end{array}
         \right)
$$
\end{enumerate}

\end{theo}

The data used in the statement of this proposition is visualized in Example~\ref{example:homologycurves}.
\begin{proof}
Let $f_i \in \mathbb{C}[[x^{\frac{1}{\kappa_i}}]]$, $i\in \{1,\dots,n\}$, 
be parametrizations of the branches of $C$.
For $f_i = \sum_{j=1}^{\infty}\alpha_{i,j} x^{\frac{j}{\kappa_i}}$, $b\in [1,\infty)$, let 
$f_{i,b}$ be the truncation $\sum_{j=1}^{\left \lfloor{b}\right \rfloor }\alpha_{i,j} x^{\frac{j}{\kappa_i}}$, where $\left \lfloor{b}\right \rfloor$ denotes the greatest integer smaller than or equal to $b$. In the case of $b = \infty$, we set $f_{i,b} := f_i$. 

The proof consists in an application of the Mayer-Vietoris Theorem, which resembles the computation of the singular homology of a circle in a certain way. The subsets involved in the Mayer-Vietoris decomposition are of the following form: we write $x \in \mathbb{C} \setminus \{0\}$ as $x = r_x e^{2\pi \varphi_x}$. Let $\phi_1, \phi_2, \phi_3, \phi_4 \in \mathbb{R}$ with $\phi_1 < \phi_2$ and $\phi_3 < \phi_4$ be fixed. For $b\in [1, \infty]$, we define the subgerm $(V_b, \underline{0})$ of $(\CC^2, \underline{0})$ by
\begin{align*}
V_b := (\{(x,y) \in \mathbb{C}^2 :y = f_{i,b}(x), (x = 0 \text { or } \exists n,m \in \mathbb{Z}: &\phi_1 < \varphi_x + 2 \pi n < \phi_2\\ \text{ or } &\phi_3 < \varphi_x + 2 \pi m < \phi_4)\}.
\end{align*}
Recall Definition~\ref{def:b-eqcomponents}. Because of Proposition~\ref{prop:0-homology}, for $b_1 \geq b_2 \geq b_3$ we have the following:
\begin{itemize}

\item The map $H_0^{b_1}(V_{b_3}, \mathbb{Z}) \rightarrow H_0^{b_2}(V_{b_3}, \mathbb{Z})$ is an isomorphism, as for any $b \geq b_3$ the $b$-connected components of $V_{b_3}$ are just its connected components.
\item The map $H_0^{b_3}(V_{b_1}; \mathbb{Z}) \rightarrow H_0^{b_3}(V_{b_2}; \mathbb{Z})$ induced by the natural projection $V_{b_1} \to V_{b_2}$ is an isomorphism, as there is a $1:1$ correspondence between the $b_3$-connected components of $V_b$ and the connected components of $V_{b_3}$ for any $b \geq b_3$.
\end{itemize} 

As a consequence, in the following commutative diagram we get the indicated isomorphisms:

\begin{tikzpicture}
  \matrix (m) [matrix of math nodes,row sep=3em,column sep=4em,minimum width=2em]
  {
     MDH_0^{b_2}(V_\infty;\mathbb{Z}) & MDH_0^{b_1}(V_\infty;\mathbb{Z}) & MDH_0^{\infty}(V_\infty;\mathbb{Z}) \\
 MDH_0^{b_2}(V_{b_1}; \mathbb{Z}) & MDH_0^{b_1}(V_{b_1}; \mathbb{Z}) & MDH_0^{\infty}(V_{b_1}; \mathbb{Z}) \\
 MDH_0^{b_2}(V_{b_2}; \mathbb{Z}) & MDH_0^{b_1}(V_{b_2}; \mathbb{Z}) & MDH_0^{\infty}(V_{b_2}; \mathbb{Z}) \\};
  \path[-stealth]
    (m-1-1) edge node [right] {$\cong$} (m-2-1)
    (m-2-1) edge node [right] {$\cong$} (m-3-1)
    (m-1-2) edge node [right] {$\cong$} (m-2-2)
    (m-2-2) edge (m-3-2)
    (m-1-3) edge  (m-2-3)
    (m-2-3) edge node [right] {$\psi_V$} (m-3-3)
    (m-1-2) edge node [above] {$\varphi_V$} (m-1-1)
    (m-2-2) edge  (m-2-1)
    (m-3-2) edge node [above] {$\cong$} (m-3-1)
    (m-2-3) edge node [above] {$\cong$} (m-2-2)
    (m-3-3) edge node [above] {$\cong$} (m-3-2)
    (m-1-3) edge (m-1-2);
\end{tikzpicture}


Therefore the natural map $\varphi_V:MDH_0^{b_1}(V_\infty;\mathbb{Z})\to MDH_0^{b_2}(V_\infty;\mathbb{Z})$ coincides with the natural map $\psi_V:MDH_0^{\infty}(V_{b_1}; \mathbb{Z})\to MDH_0^{\infty}(V_{b_2}; \mathbb{Z})$ up to concatenation with isomorphisms. By Theorem~\ref{theo:link}, up to isomorphisms the latter is the same as 
 $$\hat{\psi}_V: H_0(V_{b_1} \setminus \{\underline{0} \}, \mathbb{Z}) \to H_0(V_{b_2} \setminus \{\underline{0} \}, \mathbb{Z}),$$
  where $H_0$ denotes the singular homology.

Now we introduce the specific $b$-cover that we use to apply the Mayer-Vietoris Theorem. Let $U_1 = V_\infty$ with $\phi_1 = \frac{1}{4}\pi$ and $\phi_2 = \frac{7}{4}\pi$ and $\phi_3 = \phi_4$; and let $U_2 = V_\infty$ with $\phi_1 = \frac{3}{2}\pi$ and $\phi_2 = \frac{5}{2}\pi$ and $\phi_3 = \phi_4$. We have that $U_1 \cap U_2 = V_\infty$ with $\phi_1 = \frac{1}{4} \pi$ and $\phi_2 = \frac{1}{2} \pi$ and $\phi_3 = \frac{3}{2} \pi$ and $\phi_4 = \frac{7}{4} \pi$. Note that $\{U_1, U_2\}$ is a $b$-cover of $C$ for any $b \geq 1$. The $n$-th $b$-moderately discontinuous homology groups of $U_1$ and $U_2$ and $U_1\cap U_2$ are trivial for any $b \geq 1$, if $n \geq 1$. 
So, by the Mayer-Vietoris Theorem (Theorem \ref{theo:MayerVietoris}), for any $n > 1$, the $n$-th $b$-moderately discontinuous homology of $C$ is trivial. This completes the proof of statement~(\ref{item:curveshigherhomgroups}).

Furthermore, by the Mayer-Vietoris Theorem for $b_1 \geq b_2$ this gives us the following diagram with exact rows, in which we have omitted $\mathbb{Z}$:

\begin{tikzpicture}
  \matrix (m) [matrix of math nodes,row sep=3em,column sep=0.5em,minimum width=2em]
  {
     0 & MDH_1^{b_1}(C) & MDH_0^{b_1}(U_1 \cap U_2) & MDH_0^{b_1}(U_1) \oplus MDH_0^{b_1}(U_2) & MDH_0^{b_1}(C) & 0\\
 0 & MDH_1^{b_2}(C) & MDH_0^{b_2}(U_1 \cap U_2) & MDH_0^{b_2}(U_1) \oplus MDH_0^{b_2}(U_2) & MDH_0^{b_2}(C) & 0 \\};
  \path[-stealth]
    (m-1-1) edge (m-2-1)
    (m-1-2) edge node [right] {$h_1^{b_1, b_2}$} (m-2-2)
    (m-1-3) edge node [right] {$\varphi_{U_1 \cap U_2}$} (m-2-3)
    (m-1-4) edge node [right] {$\varphi_{U_1} \oplus \varphi_{U_2}$} (m-2-4)
    (m-1-5) edge node [right] {$h_0^{b_1, b_2}$} (m-2-5)
    (m-1-6) edge (m-2-6)
    (m-1-1) edge (m-1-2)
    (m-1-2) edge (m-1-3)
    (m-1-3) edge (m-1-4)
    (m-1-4) edge (m-1-5)
    (m-1-5) edge (m-1-6)
    (m-2-1) edge (m-2-2)
    (m-2-2) edge (m-2-3)
    (m-2-3) edge (m-2-4)
    (m-2-4) edge (m-2-5)
    (m-2-5) edge (m-2-6);
    
\end{tikzpicture}



We have shown above that we can replace $MDH_0^{b_i}(V; \mathbb{Z})$ by $H_0(V_{b_i}; \mathbb{Z})$ for $i \in \{1,2\}$ and $V \in \{U_1\cap U_2, U_1, U_2, C \}$, and that we can replace $\varphi_V$ by $\hat{\psi}_V$ for $V \in \{U_1\cap U_2, U_1, U_2\}$. Comparing the result with the analogous Mayer-Vietoris sequence in singular homology, we get that $MDH_1^{b_i}(C; \mathbb{Z}) \cong H_1(C_{b_i}; \mathbb{Z})$ for $i\in \{1,2\}$ and that for $j \in \{0,1\}$ the homomorphism $h_j^{b_1, b_2}$ is the morphism induced on the $j$-th singular homology by the projection $\rho : C_{b_1} \to C_{b_2}$ which is the following covering map: the base space $C_{b_2}$ is the disjoint union of $l_{b_2}$ circles. The covering space $C_{b_1}$ is the disjoint union of $l_{b_1}$ circles. Let $l := l_{b_2}$. For $i \in \{1, \dots, l\}$ and $j \in \{1, \dots, m_i)$, let $\rho_{i,j}$ be the $k_{i,j}:1$ covering map from the circle to itself. For $i \in \{1, \dots, l\}$, let $\rho_i : \coprod_{j=1}^{m_i} \mathbb{S}^1 \to \mathbb{S}^1$ be the morphism that all $\rho_{i,j}$ together induce on the coproduct $\amalg_{j=1}^{m_i} \mathbb{S}^1$. Concretely, $\rho_i$ sends an element $x$ in the $j$-th copy of $\mathbb{S}^1$ to $\rho_{i,j}(x)$. Then, 
$$
\rho : \coprod_{i = 1}^l \coprod_{j=1}^{m_i} \mathbb{S}^1 \to \coprod_{i = 1}^l \mathbb{S}^1
$$ 
is is given by $\coprod_{i = 1}^l \rho_i$. Concretely, $\rho$ sends an element $x$ in $\coprod_{j=1}^{m_i} \mathbb{S}^1$ by $\rho_i$ into the $i$-th copy of $\mathbb{S}^1$ in $\coprod_{i = 1}^l \mathbb{S}^1$. The proof of statement~(\ref{item:curveshomgroups}) is completed by the well-known fact of how the $0$-th and first singular homology groups of $C_b$ look like. The proof of statement~(\ref{item: homomorphisms of homology of curves}) is completed by the well-known fact of how the morphism on the $0$-th and first singular homology groups induced by $\rho$ looks like.
\end{proof}

\begin{example}\label{example:homologycurves}
We continue Example~\ref{example:EggersTree} to visualize the data of the statement of Theorem~\ref{th:homology of curves}. Let $b_1 \in [\frac{11}{4}, \frac{37}{12})$ and $b_2 \in [\frac{3}{2}, \frac{5}{2})$. In Figure~\ref{figure:prophomologycurves}, we have pictured $\mathcal{G}_i$ and $p_i$ and $p_{j,i}$ for that choice of $b_1$ and $b_2$. There, $\mathcal{G}_1$ is the graph on the left hand side and $\mathcal{G}_2$ is the graph on the right hand side.

\begin{figure}[h!]\centering
\includegraphics[width=.7\textwidth]{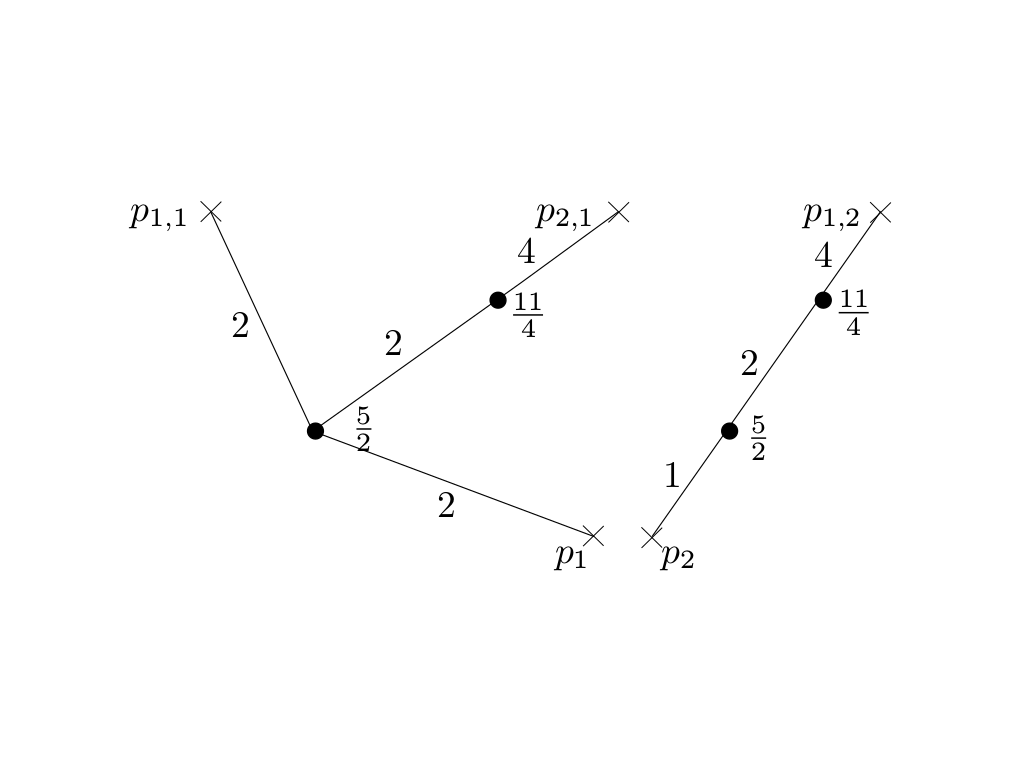}
 \caption{The data of Theorem~\ref{th:homology of curves}.}
 \label{figure:prophomologycurves}
\end{figure}

It is 
\begin{itemize}
\item $l_{b_1} = 3$ and $l := l_{b_2} = 2$,
\item $m_1 = 2$ and $m_2 = 1$,
\item $k_{1,1} = 1 $ and $k_{2,1} = 2 $ and $k_{1,2} = 4 $.
\end{itemize}

\end{example}

\begin{cor}\label{cor:homology of irreducible curves}
Let $C$ be an irreducible curve germ. We use the same notation as in Theorem \ref{th:homology of curves}. The MD homology of $C$ with respect to $A$ is as follows:
\begin{enumerate}
\item For any $b\in [1,\infty]$, it is $MDH_0^b(C;A) \cong MDH_1^b(C;A) \cong A$ and \\
$MDH_n^b(C;A) \cong \{0\}$, if $n > 1$.
\item For any $b_1, b_2 \in [1, \infty]$, $b_1 \geq b_2$, it is $h_0^{b_1, b_2} = \text{id}_A$.
\item For $b_1, b_2 \in [1, \infty]$, $b_1 \geq b_2$, it is 
\begin{itemize}
\item $h_1^{b_1, b_2} = \text{id}_A$, if $(b_2, b_1]$ does not contain any Puiseux exponent,
\item $h_1^{b_1, b_2}(x) = kx$, if $(b_2, b_1]$ contains one Puiseux exponent with corresponding Puiseux pair $(m,k)$ for some $m \in \mathbb{N}$.
\end{itemize}
If $(b_2, b_1]$ contains more than one Puiseux exponent, $h_1^{b_1, b_2}$ can be determined by concatenation.
\end{enumerate}
\end{cor}
\begin{proof}
This corollary follows directly from Theorem \ref{th:homology of curves}.
\end{proof}

By \cite{Teissier:1982} (see also \cite{Neumann:2014} and \cite{Fernandes:2003}), the classification of curve germs by its outer bi-Lipschitz geometry coincides with the classification of curve germs by its embedded topology. Therefore, we get the following corollary:

\begin{cor}\label{cor:puiseuxpairscuvres}
Let $C$ be a irreducible plane curve germ. The MD homology for the outer metric of $C$ with respect to $\mathbb{Z}$ detects all Puiseux pairs of $C$. Therefore the integral MD Homology for the outer geometry determines the outer geometry and the embedded topology of irreducible plane curve singularities.
\end{cor}
\begin{proof}
We use the same notation as in Theorem \ref{th:homology of curves}. By Corollary \ref{cor:homology of irreducible curves}, the set $P$ of all Puiseux exponents of $C$ can be described as follows:
\begin{equation}\label{eq:set of Puiseux exponents}
P = \{b \in (1, \infty) : \text{there is no } \delta > 0 \text{ such that } h_1^{b, b - \delta} \text{ is an isomorphism} \}.
\end{equation}
\end{proof}

Now we analyze what information is detected in the case of reducible plane curve singularities.

\begin{remark}
\label{rem:reduciblecurves}
If $C$ is reducible, the union of set of all Puiseux exponents of all branches of $C$ with the set of all contact exponents is the set of values $b$ such that there is no $\delta > 0$ such that $h_1^{b, b - \delta}$ is an isomorphism. However, in general we do not know how this set of Puiseux exponents and contact orders distribute among the different branches, as the following example shows.
\end{remark}

\begin{example}\label{example:reducible curve}
We use the same notation as in Theorem \ref{th:homology of curves}.
Let $C$ and $D$ be the curves with the following five components respectively:
\begin{align*}
& C_1 = \{(x,y) \in \mathbb{C}^2 : y = x + x^2 + x^{\frac{5}{2}} \}, 
&& D_1 = \{(x,y) \in \mathbb{C}^2 : y = x + x^2 \},\\
& C_2 = \{(x,y) \in \mathbb{C}^2 : y = x + 2x^2  \}, 
&& D_2 = \{(x,y) \in \mathbb{C}^2 : y = x + 2x^2  \},\\
& C_3 = \{(x,y) \in \mathbb{C}^2 : y = 2x +  x^2 \},
&& D_3 = \{(x,y) \in \mathbb{C}^2 : y = 2x +  x^2 \},\\
& C_4 = \{(x,y) \in \mathbb{C}^2 : y = 2x +  2x^2 \},
&& D_4 = \{(x,y) \in \mathbb{C}^2 : y = 2x +  2x^2 \},\\
& C_5 = \{(x,y) \in \mathbb{C}^2 : y = 2x +  3x^2 \},
&& D_5 = \{(x,y) \in \mathbb{C}^2 : y = 2x +  3x^2 + x^{\frac{5}{2}} \}
\end{align*}
The embedded topological types of the two curves do not coincide since their Eggers-Wall trees are not isomorphic as trees. 
But their MD homology with respect to $\mathbb{Z}$ are isomorphic: we denote the morphisms $h_{\ast}^{b_1, b_2}$ of the MD homology of $C$ and $D$ by $h_{\ast}^{b_1, b_2}(C)$ and $h_{\ast}^{b_1, b_2}(D)$ respectively. Having a look at their Eggers-Wall trees, it becomes clear that the $0$-th and first $b$-moderately discontinuous homology groups coincide for any $b$ and so do the morphisms $h_0^{b_1, b_2}(D)$ and $h_0^{b_1, b_2}(D)$ for any $b_1 \geq b_2$. As the Eggers-Wall trees of $C$ and $D$ coincide on $r^{-1}([0, \frac{5}{2}))$ and  $r^{-1}((\frac{5}{2}, \infty])$, $h_1^{b_1, b_2}(C)$ and $h_1^{b_1, b_2}(D)$ also coincide, if $b_1, b_2 < \frac{5}{2}$ or $b_1, b_2 > \frac{5}{2}$. If $b_1 \geq \frac{5}{2}$ and $b_2 < \frac{5}{2} $, $h_1^{b_1, b_2}(C)$ and $h_1^{b_1, b_2}(D)$ are the same up to concatenation with isomorphisms on the right and on the left. For example, if $b_1 \geq \frac{5}{2}$ and $b_2 \in [1, 2)$, $h_1^{b_1, b_2}(C)$ and $h_1^{b_1, b_2}(D)$ are given by matrix multiplication with $M_1(C)$ and $M_1(D)$, respectively, where
$$
M_1(C) := \left( 
   \begin{array}{ccccc}
   2 & 1 & 0 & 0 & 0 \\
   0 & 0 & 1 & 1 & 1 
              \end{array}
         \right),         
M_1(D) := \left( 
   \begin{array}{ccccc}
   1 & 1 & 0 & 0 & 0 \\
   0 & 0 & 1 & 1 & 2 
              \end{array}
         \right).
$$
\end{example}

\begin{cor}\label{cor:fremeeggers}
Let $(C,O)$ be a plane curve singularity. Then the $\infty$-framed MD Homology at degree $1$ determines the Eggers tree and hence the embedded topology of $(C,O)$.
\end{cor}
\proof
Once taken the basis for $MD$ Homology for $b=\infty$ in degree $1$, the statement is an immediate consequence of~Assertion (3) in Theorem~\ref{th:homology of curves}.
\endproof

\section{Final Remarks and Open Questions.}
\subsection{On the Lipschitz Normally Embedded problem}

A germ $(X,x_0)\subset (\RR^m,x_0)$  is said to be \emph{Lipschitz Normally Embedded} (LNE for short) if $(X,x_0,d_{inn})$ and $(X,x_0,d_{out})$ are bi-Lipschitz equivalent. The LNE Problem tries to characterize the singularities that are LNE. 

Since the identity map $Id:(X,x_0,d_{in})\to (X,x_0,d_{out})$ is a l.v.a Lipschitz subanalytic morphism we have a homomorphism of $\BB$-groups
\begin{equation}
\label{eq:obsLNE}
MDH_\bullet^*(X,x_0,d_{in};A)\to MDH_\bullet^*(X,x_0,d_{out})
\end{equation}
which is an isomorphism if $X$ is LNE. In this sense MD Homology is an invariant obstructing the LNE property.  



We do not have counter-examples for a positive answer to the following question:

\begin{problem}
Let $X\subset\RR^m$ be a metric subanalytic subset. Suppose that for any point $x_0\in X$ the homomorphism~(\ref{eq:obsLNE}) is an isomorphism. Is $X$ LNE?
\end{problem}

Note that it is known that if $X$ is compact and locally normally embedded at any point then it is LNE (see \cite{reembedding}).

\end{document}